\documentclass[11pt]{article}
\usepackage[english]{babel}
\usepackage{amsfonts,amsmath,amsxtra,amsthm,amssymb,latexsym}
\usepackage{hyperref}
\usepackage{enumerate}
\usepackage{amssymb}
\usepackage[usenames,dvipsnames]{color}
\usepackage{verbatim}

\newtheorem{theorem}{Theorem}[section]
\newtheorem{corollary}[theorem]{Corollary}
\newtheorem{lemma}[theorem]{Lemma}

\newtheorem{definition}[theorem]{Definition}
\newtheorem{remark}[theorem]{Remark}

\newtheorem{proposition}[theorem]{Proposition}

\numberwithin{equation}{section}

 \DeclareMathOperator{\dom}{dom}
\DeclareMathOperator{\ran}{ran} \DeclareMathOperator{\Ext}{Ext}

\DeclareMathOperator{\diag}{diag}\DeclareMathOperator{\loc}{loc}
\DeclareMathOperator{\comp}{comp}

\newcommand\C{{\mathbb{C}}}
\newcommand\R{{\mathbb{R}}}

\newcommand\N{{\mathbb{N}}}
\newcommand\Z{{\mathbb{Z}}}

\newcommand\gH{{\mathfrak{H}}}

\newcommand{\gG}{{\Gamma}}

\newcommand{\gd}{{d}}
\newcommand{\gA}{{\alpha}}
\newcommand{\gB}{{\beta}}
\newcommand{\I}{\mathrm{i}}

\newcommand{\rd}{\mathrm{d}}
\newcommand\cH{{\mathcal{H}}}
\newcommand\cI{{\mathcal{I}}}
\newcommand\cB{{\mathcal{B}}}
\newcommand\cA{{\mathcal{A}}}
\newcommand\cC{{\mathcal{C}}}
\newcommand\cN{{\mathfrak{N}}}

\newcommand\rD{{\rm{d}}}

\def\Ext{{\rm Ext}}

\def\wt#1{{{\widetilde #1} }}

\title
{Deficiency indices and discreteness property of block Jacobi matrices and Dirac operators  with point  interactions}

\author{ Viktoriya~S.~Budyka$^{1,2}$, Mark~M.~Malamud$^2$\\
$^1$ Donetsk Academy
of
Management and Public Administration, Donetsk, \\ e-mail: {budyka.vik@gmail.com}\\
$^2$Peoples Friendship University of Russia (RUDN University), Moscow, Russia,\\ e-mail: {malamud3m@gmail.com}}

\begin{document}

\maketitle
\begin{abstract}

The paper concerns with  infinite symmetric  block Jacobi  matrices  $\bf J$ with $p\times p$-matrix entries.
 We present  new conditions for general  block  Jacobi  matrices  to be selfadjoint and have  discrete spectrum.

 In our previous papers there was established  a close relation between  a class of such matrices and  symmetric
 $2p\times  2p$  Dirac operators $\mathrm{\bf D}_{X,\alpha}$
 with point interactions in $L^2(\Bbb R; \Bbb C^{2p})$. In particular, their deficiency indices are related
by $n_\pm(\mathrm{\bf D}_{X,\alpha})= n_\pm({\bf J}_{X,\gA})$.
 For  block  Jacobi matrices  of this class we present several
  conditions ensuring  equality $n_\pm({\bf J}_{X,\gA})=k$ with any $k \le p$.
    Applications to  matrix Schrodinger and Dirac  operators with point interactions  are given. It is worth mentioning that
    a connection between Dirac and Jacobi operators  is employed here in both directions for the first time.
    In particular, to prove the equality $n_\pm({\bf J}_{X,\gA})=p$ for ${\bf J}_{X,\gA}$ we first  establish it
    for Dirac operator $\mathrm{\bf D}_{X,\alpha}$.

\end{abstract}

%

\renewcommand{\contentsname}{Contents}
\tableofcontents

%
%
%

%



\section{Introduction}\label{s1}
The main object of the paper is the infinite  block Jacobi matrix
\begin{equation}\label{Jacobi_m_intro}
{\bf J}=\left(
    \begin{array}{cccccc}
      \mathcal A_0 & \mathcal B_0 & \mathbb O_p & \mathbb O_p & \mathbb O_p & \ldots \\
      \mathcal B_0^* & \mathcal A_1 & \mathcal B_1 & \mathbb O_p & \mathbb O_p & \ldots \\
      \mathbb O_p & \mathcal B_1^* & \mathcal A_2 & \mathcal B_2 & \mathbb O_p & \ldots \\
      \vdots & \vdots & \vdots & \vdots & \vdots & \ddots \\
            \end{array}
  \right),
\end{equation}
with $p\times p$-matrix entries  $\mathcal A_j=\mathcal A_j^*\in\mathbb C^{p\times p}$, $\mathcal B_j \in \mathbb C^{p\times p}$ is invertible, $j\in\mathbb N_0$,  $\mathbb O_p\  (\in\mathbb C^{p\times p})$ denotes zero matrix.
Following M.G. Krein  (see  \cite{Krein, Krein49})  matrix   $\bf J$ is also called
Jacobi matrix with matrix entries.

Let $l_0^2(\mathbb N;\mathbb C^p)$ be a subset of finite sequences in $l^2(\mathbb N;\mathbb C^p)$.
 The mapping $l_0^2(\mathbb N;\mathbb C^p) \ni f \to {\bf J}f$ defines a linear symmetric,
 but not closed operator ${\bf J}^0$.  The closure of the operator ${\bf J}^0$
 defines  a minimal closed  symmetric operator  ${\bf J}_{\min}$
 in $l^2(\mathbb N;\mathbb C^p)$.
  In what follows we will identify the minimal operator ${\bf J}_{\min}$ with
 the  matrix ${\bf J}$ of form \eqref{Jacobi_m_intro} and write ${\bf J}_{\min} = {\bf J}$.
 We also put  ${\bf J}_{\max} = {\bf J}^*$.

Note that  ${\bf J}$ is symmetric, ${\bf J} \subset {\bf J}^*$,  while it is not necessary
self-adjoint, i.e. the deficiency indices  $n_\pm({\bf J}):=\dim\frak N_{\pm i}({\bf J}) := \dim\ker ({\bf J}^* \mp iI)$ are non-trivial.
The most simple and widely  known test of selfadjointness of $\bf J$
is the matrix version of the Carleman test (see \cite[Theorem VII.2.9]{Ber68},
 and also \cite{KosMir98, KosMir99}), which reads as follows.
 \begin{theorem}[\cite{Ber68}, Carleman Test]\label{CarT}
 If
\begin{equation}\label{Car}
\sum\limits_{j=0}^\infty\|\mathcal B_j\|^{-1}=+\infty,
\end{equation}
then the (minimal) Jacobi operator $\bf J$ is self--adjoint,
i.e. ${\bf J}={\bf J}_{\min}={\bf J}_{\max}={\bf J}^*$.
 \end{theorem}
Condition \eqref{Car} is not necessary for selfadjointness of $\bf J$ in general  even in the
 scalar case  $(p=1)$ while according to Berezansky's result (see \cite[Theorem VII.1.5]{Ber68} and \cite{Akh})
 it becomes  necessary
under certain additional  assumptions on the (scalar) entries $\mathcal A_n$ and $\mathcal B_n$.
A matrix version of his result as well as generalizations can be found
in \cite{KosMir98, KosMir99}.

In general, one has $0\leq  n_{\pm}({\bf J}) \leq p$ (see  \cite{Ber68, Krein, Krein49}).
Besides, there is one more restriction: indices achieve the maximal value
only simultaneously, i.e. $n_+({\bf J}) =  p \Longleftrightarrow  n_-({\bf J}) =p$ (see \cite{Kogan70}).
  It is shown in  \cite{Dyuk06, Dyuk10} that the converse statement is also true:  for any pair of numbers $\{n_-,n_+\}$, satisfying either $0\leq n_-,n_+< p$, or $n_{\pm}=p$, there exists
  a Jacobi matrix $\bf J$ such that $n_{\pm}({\bf J}) = p$.

The problem of computing the deficiency indices of Jacobi matrices is the first main problem
naturally arising  in the spectral theory of  such matrices as well as in
the corresponding moment problem. It was established by M.G. Krein  (see  \cite{Krein, Krein49}) that with any Jacobi matrix $\bf J$, one can associate
a matrix moment problem. This problem has a unique solution (normalized in a certain
sense) if and only if one of the numbers $n_-$ or $n_+$ is equal to zero. This topic has  attracted  substantial attention, in particular during
the last two decades (see for instance,
\cite{Bro_Mir, BroMir18, BudMal19, BudMal20, BudMalPos18, CarMalPos13, Dyuk06, Dyuk10, Dyuk20, KM10, KM, KM_rev, KMN, KosMir98, KosMir99, KosMir01, Krein, Krein49,  MirSaf16}).
We especially mention recent publications \cite{PetVelaz14} ($p=1$) and \cite{BroMir18}, \cite{BudMalPos18}, \cite{BudMal20}  ($p\geq1$)
where new  different  conditions
for  block Jacobi matrices to be selfadjoint  were found.

New applications of Jacobi matrices to Schr\"{o}dinger operators with $\delta$--interactions given by  formal differential expression
  \begin{equation}\label{difexpr}
l_{X,\gA}:= -\frac{\rd^2}{\rd x^2}+\sum_{n=1}^{\infty}\gA_n\delta(x-x_n)
\end{equation}
was recently discovered in  \cite{KM10}, \cite{KM}.
Here  $X=\{x_n\}_{n=1}^\infty\subset \mathcal I = (0,b)$, $b\le \infty$, is  a strictly increasing
sequence with $x_0:=0$, $x_{n+1}>x_{n},\ n\in \N$, and  such that $x_n\to b$,
and  $\{\alpha_n\}_1^\infty \subset\mathbb R$.

Namely, in \cite{KM10}, \cite{KM}, there was established  a close connection between
Schr\"{o}dinger operator $\mathrm H_{X,\alpha}$ with $\delta$--interactions, associated in
$L^2(\mathcal I)$ with  expression \eqref{difexpr}  on the one hand and  Jacobi matrix
\begin{equation}\label{B1_intro}
  {\bf J}_{{X,\alpha}}^{(1)}(\mathrm{\bf{H}})\!=\!\left(\begin{array}{cccccc}
\!\frac{1}{r_1^{2}}\,\widetilde{\alpha}_1\! & \!\frac{1}{r_1r_2\gd_2}\mathbb I_p\! & \!\mathbb O_p\!&   \mathbb O_p\!&   \mathbb O_p\!&   \dots\!\\
\!\frac{1}{r_1r_2\gd_2}\mathbb I_p \!&\!\frac{1}{r_2^{2}}\,\widetilde{\alpha}_2\! &\! \frac{1}{r_2r_3\gd_3}\mathbb I_p\! & \! \mathbb O_p\!& \! \mathbb O_p\!& \! \dots\!\\
\!\mathbb O_p\! &\! \frac{1}{r_2r_3\gd_3}\mathbb I_p\! &\! \frac{1}{r_3^{2}}\,\widetilde{\alpha}_3\!&\!  \frac{1}{r_3r_4\gd_4}\mathbb I_p\!& \! \mathbb O_p\!& \! \dots\!\\
\mathbb O_p & \mathbb O_p & \frac{1}{r_3r_4\gd_4}\mathbb I_p& \frac{1}{r_4^2}\widetilde{\alpha}_4& \frac{1}{r_4r_5\gd_5}\mathbb I_p& \! \dots\!\\
\!\dots\! &\! \dots\!&\! \dots\! & \!\dots\!& \! \dots\!& \! \dots\!
\end{array}\right)
\end{equation}
(with $p=1$) on the other hand. Here $d_n:=x_n-x_{n-1}$, $r_n:=\sqrt{d_n+d_{n+1}}$, and
$\widetilde{\alpha}_n:=\alpha_n+\Big(\frac{1}{\gd_n}+\frac{1}{\gd_{n+1}}\Big)\mathbb I_p,$  $n\in\mathbb N$.
Denote by  ${\bf \mathcal J}_{X,\gA}(H,p)$    the class of matrices \eqref{B1_intro}.

More precisely,  it was shown in  \cite{KM10}, \cite{KM}  that
certain spectral properties (deficiency indices, discreteness spectra, semiboundedness, positive definiteness, negative point and singular spectra, etc.) of Hamiltonian  $\mathrm{\bf H}_{X,\gA}$ are closely related
 with the corresponding  properties of the (minimal) Jacobi operator ${\bf J}_{X,\gA}^{(1)}(\mathrm {\bf H})$ (with $p=1$).
In particular, it was proved in \cite{KM10, KM}, that $n_\pm(\mathrm{\bf H}_{X,\gA}) = n_\pm({\bf J}_{X,\gA}^{(1)}(\mathrm {\bf H}))$, which implies  $n_\pm(\mathrm{\bf H}_{X,\gA})\leq 1$. This
inequality has first been established by different methods in \cite{bsw, Min_86}.
Besides, combining  the equality
$n_\pm(\mathrm{\bf H}_{X,\gA}) = n_\pm({\bf J}_{X,\gA}^{(1)}(\mathrm {\bf H}))$   with the Carleman test \eqref{Car} yields the result:
 \begin{proposition}$\cite{KM10, KM}$\label{self-adj_H_X_Carlem}
\emph{Schr\"{o}dinger operator $\mathrm{\bf H}_{X,\gA}$   with $\delta$-interactions
is  self-adjoint in $L^2(\R_+)$ for any $\gA=\{\gA_n\}_{n\in\N}\subset \R$ provided  that}
\begin{equation}\label{d_n2}
\sum_{n\in\N}\gd_n^2=\infty.
\end{equation}
 \end{proposition}
Note that  selfadjointness of $\mathrm{\bf H}_{X,\gA}$   has  earlier been established
in the case   $\gd_* :=\inf_{n}d_n >0$  by Gesztesy and Kirsh \cite{GesKir85} (see also \cite{Alb_Ges_88}
and  survey  \cite{KM_rev}).

The above mentioned connection between  $\mathrm{\bf H}_{X,\gA}$  and ${\bf J}_{X,\gA}^{(1)}(\mathrm {\bf H})$ was  extended in \cite{KMN} to the case of $p\times p$-matrix differential  expressions \eqref{difexpr} (with $\{\gA_n\}_1^\infty \subset \mathbb C^{p\times p}$)  and the block Jacobi matrices ${\bf J}_{X,\gA}^{(1)}(\mathrm {\bf H})$ with $p>1$.
In particular, it was proved in \cite{KMN} that
  \begin{equation}\label{equal_def_ind_Schrod}
n_\pm(\mathrm{\bf H}_{X,\gA})=n_\pm({\bf J}_{X,\gA}^{(1)}(\mathrm {\bf H})) \quad \text{for any}\quad p\ge 1.
  \end{equation}
It follows that  $n_\pm(\mathrm{\bf H}_{X,\gA})\leq p$. Proposition \ref{self-adj_H_X_Carlem} has attracted certain  attention
and has been developed for different classes of (scalar and matrix) Schr\"{o}dinger operators to ensure their self-adjointness
in  $L^2(\R_+; \Bbb C^p)$ in \cite{Bro}, \cite{Bro_Mir_S}, \cite{IsmKos10}, \cite{Mir14}, \cite{MirSaf16}.
In particular, it was shown by different methods in \cite{MirSaf16} and \cite{KMN} that the
condition \eqref{d_n2}
still ensures selfadjointness of $\mathrm{\bf H}_{X,\gA}$ for any $p>1$.
We note only that this result  is also immediate by combining equality  \eqref{equal_def_ind_Schrod}
with  the matrix Carleman test \eqref{Car}  (see \cite{KMN}).

One more application of Jacobi matrices recently occurred  in \cite{CarMalPos13}
in connection with   Dirac operators with $\delta$--interactions given by  formal
differential expression
  \begin{equation}\label{Dintro}
\mathrm{\bf D}_{X,\gA} := -i\,c\,\frac{\rd}{\rd x}\otimes \left(\begin{array}{cc} \mathbb O_p & \mathbb I_p
\\\mathbb I_p & \mathbb O_p \end{array}\right) + \frac{c^{2}}{2}\left(\begin{array}{cc}\mathbb I_p & \mathbb O_p
 \\\mathbb O_p&
-\mathbb I_p\end{array}\right)  + \sum_{n=1}^{\infty}\gA_n\delta(x-x_n)
\end{equation}
%
with $p=1$ and $\{\gA_n\}_1^\infty \subset  \mathbb R$,  and where $c>0$ denotes the velocity of light.
First rigorous treatment  of the operator $\mathrm{\bf D}_{X,\gA}$ associated in $L^2(\mathbb R; \C^2)$  with  expression \eqref{Dintro}  goes back  to the paper  by Gesztesy and $\check{\rm S}$eba \cite{GS}
(see formulas \eqref{delta}, \eqref{deltap} below). Therefore following \cite{CarMalPos13}  we call the operator ${\mathrm {\bf D}}_{X,\gA}$ by  Gesztesy--$\check{\rm S}$eba realization
(in short GS-realization) of Dirac differential expression.

Namely,  there was established in \cite{CarMalPos13}
that like in the Schr\"{o}dinger case, certain spectral properties of GS-realization
$\mathrm{\bf D}_{X,\gA}$ in $L^2(\mathcal I; \C^2)$ (deficiency indices,  discreteness
and other types  of spectra, etc.)  are closely related to that of Jacobi matrix
   \begin{equation}\label{IV.2.1_01_intro}
 {\bf J}_{X,\gA}' :=\left(
\begin{array}{cccccc}
  \mathbb O_p  & \frac{c}{\gd_1}\mathbb I_p & \mathbb O_p  & \mathbb O_p & \mathbb O_p   &  \dots\\
   \frac{c}{\gd_1}\mathbb I_p  &  -\frac{c}{\gd_1}\mathbb I_p &
   \frac{c}{\gd_1^{1/2}\gd_2^{1/2}}\mathbb I_p & \mathbb O_p  & \mathbb O_p &  \dots\\
  \mathbb O_p  & \frac{c}{\gd_1^{1/2}\gd_2^{1/2}}\mathbb I_p  & \frac{\alpha_1}{\gd_2}  &
   \frac{c}{\gd_2}\mathbb I_p & \mathbb O_p &   \dots\\
  \mathbb O_p  & \mathbb O_p  & \frac{c}{\gd_2}\mathbb I_p &  -\frac{c}{\gd_2}\mathbb I_p &
  \frac{c}{\gd_2^{1/2}\gd_3^{1/2}}\mathbb I_p &  \dots\\
  \mathbb O_p  & \mathbb O_p  & \mathbb O_p  & \frac{c}{\gd_2^{1/2}\gd_3^{1/2}}\mathbb I_p &
   \frac{\alpha_2}{\gd_3} &   \dots\\
\dots& \dots&\dots&\dots&\dots&\dots\\
 \end{array}%
\right)\,
   \end{equation}
under certain restrictions  on $d_n$ (with $p=1$) and  $\gd_n:=x_{n}-x_{n-1}$.  These results have been extended to the matrix
case with  $\gA = \{\gA_n\}_1^\infty \subset  \mathbb C^{p\times p}$, $\gA_n = \gA_n^*$,
in \cite{BudMalPos17, BudMalPos18, BudMal19}.
In particular, it was shown in \cite{BudMalPos17, BudMalPos18, BudMal19} that formula
\eqref{equal_def_ind_Schrod}  remains valid for the operators  $\mathrm{\bf D}_{X,\gA}$
in $L^2(\mathcal I; \C^{2p})$ and   ${\bf J}_{X,\gA}'$, i.e.
  \begin{equation}\label{equal_def_ind_Dirac}
n_\pm(\mathrm{\bf D}_{X,\alpha})=n_\pm({\bf J}_{X,\gA})=n_\pm({\bf J}_{X,\gA}') \quad \text{for any}\quad p\ge 1,
  \end{equation}
  where the Jacobi matrix ${\bf J}_{X,\gA}$ is defined by formula  \eqref{IV.2.1}.

Denote by  ${\bf \mathcal J}_{X,\gA}(p)$
the set  of Jacobi matrices \eqref{IV.2.1_01_intro}.
We also introduce other class  ${\bf \mathcal J}_{X,\gB}(p)$,
consisting  of Jacobi matrices
${\bf J}_{X,\gB}$, given by \eqref{IV.3.1_04'}.
Carleman's  condition \eqref{Car} for ${\bf J}'_{X,\gA}\in  {\bf \mathcal J}_{X,\gA}(p)$
becomes  $\sum_{n\in\N}\gd_n = \infty$,   
and combining with \eqref{equal_def_ind_Dirac}   yields $n_\pm(\mathrm{\bf D}_{X,\alpha}) = n_\pm({\bf J}_{X,\gA}) = 0$,
i.e.  $\mathrm{\bf D}_{X,\alpha} = \mathrm{\bf D}_{X,\alpha}^*$,  provided that  $\mathcal I = \R_+$.

In the present paper  we investigate  \emph{deficiency indices,  selfadjointness,  and discreteness property} of
general  block  Jacobi matrices $\bf J$ as well as their subclasses related  to matrix operators $\mathrm{\bf H}_{X,\alpha}$
and $\mathrm{\bf D}_{X,\alpha}$.
In particular,  we show (see Theorem \ref{d_spec_J})  that the matrix $\bf J$  \eqref{Jacobi_m_intro} is  {selfadjoint and discrete}
provided that
\begin{equation}\label{adj1-intro}
\underset{n\geq N}{\sup}\|\mathcal A_{n}^{-1}\mathcal B_{n}\|<\frac{1}{2},\qquad\underset{n\geq N}{\sup}\|\mathcal A_{n}^{-1}\mathcal B_{n-1}^*\|<\frac{1}{2}.
\end{equation}
Besides, in  Theorem \ref{J_disc} we  prove  that under the condition
   \begin{equation}\label{cond_disc-intro}
\limsup_{n\to\infty}\big\||\mathcal A_{n+1}|^{-1/2} \cdot\mathcal B_{n}^*\cdot
|\mathcal A_n|^{-1/2}\big\|<\frac{1}{2},
  \end{equation}
the block Jacobi operator $\bf J$ is symmetric in $l^2(\mathbb N_0;\mathbb C^p)$
  with equal  indices $n_-({\bf J}) = n_+({\bf J}) \le p$ and
such that \emph{each its selfadjoint  extension} ${\bf {\wt J}} = {\bf {\wt J}}^*$  \emph{is discrete}.
In particular, $\bf J$ \emph{is discrete} whenever ${\bf J} = {\bf J}^*$.

 In the scalar case ($p=1$)  discreteness property  of $\bf J= \bf J^*$
 under condition  \eqref{cond_disc-intro} was  established  by  Cojuhari, Janas \cite{CojJan07} and Chihara \cite{Chi62};
 other discreteness conditions were found  by Janas and Naboko \cite{JanNab01}  (see Remark \ref{chihara} (ii)).

Applying condition \eqref{cond_disc-intro} to Jacobi operator \eqref{B1_intro} and using
the  close relation between matrix Schr\"{o}dinger operator   $\mathrm{\bf H}_{X,\alpha}$ (with $\{\gA_n\}_1^\infty \subset \mathbb C^{p\times p}$)  and block Jacobi operator ${\bf J}_{X,\gA}^{(1)}(\mathrm {\bf H})$
we obtain (see Theorem \ref{Shr_disc_2})  the following  discreteness condition  for
each selfadjoint  extension $\widetilde{\mathrm{\bf H}}_{X,\alpha}$  of $\mathrm{\bf H}_{X,\alpha}$:
\begin{equation}\label{self_adj_Sh3-intro}
\limsup_{n\to\infty}\frac{1}{d_{n+1}}\Big\||\widetilde{\alpha}_n|^{-1/2}\cdot|\widetilde{\alpha}_{n+1}|^{-1/2}\Big\|<\frac{1}{2}
\end{equation}
with $\widetilde{\alpha}_n$ being the diagonal  entries of the matrix \eqref{B1_intro} defined just after \eqref{B1_intro}.

Similar considerations applied to $\mathrm{\bf J}_{X,\alpha}$  leads to  the following
\emph{discreteness condition  for each selfadjoint  extension}
$\widetilde{\mathrm{\bf D}}_{X,\alpha}$ of the GS-realization
$\mathrm{\bf D}_{X,\alpha}$ of the Dirac operator (see  Theorem \ref{D_disc_line})
 \begin{equation}\label{prop_chihara_1-intro}
\limsup_{n\to\infty}\||\alpha_n|^{-1/2}\|<\frac{1}{2\sqrt c}.
  \end{equation}
Moreover, $GS$-realization $\mathrm{\bf D}_{X,\alpha}$ in $L^2(\mathbb R_+, \C^{2p})$ is always selfadjoint and
\eqref{prop_chihara_1-intro} ensures its discreteness.

Another our  result concerns maximality of  deficiency indices of Jacobi matrices
and reads as follows:
   \begin{equation}\label{ind_J_{X,A}=p_Intro}
n_\pm({\bf J}_{X,\gA}) = n_\pm(\widehat{{\bf J}}_{X,\gA})  = p \qquad \text{for any}\qquad   {\bf J}_{X,\gA}\in {\bf \mathcal J}_{X,\gA}(p)
     \end{equation}
provided that
  \begin{equation}\label{Intro_max_ind}
\sum_{n=2}^{\infty}d_{n}\prod_{k=1}^{n-1}\left(1+\frac 1c\,\|\alpha_k\|_{\mathbb
C^{p\times p}}\right)^{2}<+\infty
  \end{equation}
(for matrices $\widehat{{\bf J}}_{X,\gA}$ some extra
conditions are required). In other words, condition \eqref{Intro_max_ind} ensures
the complete indeterminacy of the  matrix moment problems related to the Jacobi matrix ${\bf J}_{X,\gA}$.

Block matrices of the classes  ${\bf \mathcal J}_{X,\gA}(p)$  have several interesting features.
First, as a byproduct of \eqref{ind_J_{X,A}=p_Intro} we derive that \emph{each selfadjoint extension of  ${\bf J}_{X,\gA}$
has discrete spectrum whenever \eqref{Intro_max_ind} holds}.

Secondary, condition \eqref{Intro_max_ind} is obviously satisfied for $\gA=\mathbb O$, hence
Carleman's condition  \eqref{Car}  for matrices ${{\bf J}}_{X,0}$   \emph{becomes  also necessary  for selfadjointness,
while the Berezansky condition is violated}.

At the same time  under other conditions (see \eqref{4.14}) depending
on $\gA = \{\gA_n\}_1^\infty$    matrices ${\bf J}_{X,\gA}$ and  ${\bf J}_{X,\gA}^{(1)}(\mathrm {\bf H})$
may  be  either selfadjoint or  have arbitrary  intermediate  indices.
In particular, \emph{matrices ${\bf J}_{X,\gA}$ with certain $\gA$
demonstrate that Carleman's condition \eqref{Car} is not necessary  for selfadjointness}.

Moreover, condition \eqref{Intro_max_ind}   with $\gB= \{\gB_n\}_1^\infty$ in
place of $\gA = \{\gA_n\}_1^\infty$ ensures  maximality deficiency indices of  Jacobi matrices
${\bf J}_{X,\gB}$, given by \eqref{IV.3.1_04'}  (and their perturbations  $\widehat{{\bf J}}_{X,\gB}$  \eqref{IV.2.1_01''B}).
Hence, each selfadjoint extension of  ${\bf J}_{X,\gB}$
has discrete spectrum whenever condition \eqref{Intro_max_ind} with $\gB$ in
place of $\gA$   holds.

Other conditions for block Jacobi matrices  \eqref{Jacobi_m_intro}
to have maximal indices $n_\pm({\bf J}) = p$ were obtained by Kostyuchenko and Mirzoev
\cite{KosMir98}--\cite{KosMir01} (see Theorem \ref{KosMirTh} below) and Dyukarev \cite{Dyuk10}.  Note  however that matrices ${\bf J}_{X,\gA}$  \emph{never meet conditions} of Theorem \ref{KosMirTh}
(see Lemma \ref{KosMirInd} below).
It is worth to note also that one of matrices  ${\bf J}_{X,\gB}\in {\bf \mathcal J}_{X,\gB}(p)$ (see \eqref{IV.3.1_04'}) with zero diagonal coincides with the block Jacobi matrix
constructed by Dyukarev \cite{Dyuk10} in the framework of  another approach (see Proposition \ref{JBabs}).

It is worth to mention that in this paper we employ  the close relations between ${\bf J}_{X,\gA}$   and
  $\mathrm{\bf D}_{X,\alpha}$ in  both  directions.
For instance, first we find conditions for intensities  $\{\gA_n\}_1^\infty$  that ensure maximal
deficiency indices for GS-realizations  $\mathrm{\bf D}_{X,\alpha}$, and then using formulas \eqref{equal_def_ind_Dirac} derive the corresponding statements for matrices ${\bf J}_{X,\gA}$,
and more general Jacobi matrices $\widehat{{\bf J}}_{X,\gA}$.

The paper is organized as follows. In the Sections \ref{J_s} -- \ref{J_d} we indicate conditions for $p\times p$ Jacobi matrix~${{\bf J}}$  \eqref{Jacobi_m_intro}  to be  selfadjoint and discrete. In particular, we show that
conditions \eqref{adj1-intro} and \eqref{cond_disc-intro} ensure
discreteness properties of block Jacobi matrix ${{\bf J}}$.

In Section \ref{sec4} we prove the  abstract result, Theorem \ref{abs_th},  on preserving  of the deficiency
indices for  block Jacobi matrices~${{\bf J}}$  \eqref{Jacobi_m_intro} under  perturbations.

In Section \ref{sec2}   we prove that condition  \eqref{Intro_max_ind}  ensures
maximal deficiency indices for GS-realization   $\mathrm{\bf D}_{X,\alpha}$,
i.e. relations  $n_\pm(\mathrm{\bf D}_{X,\alpha})=  p$.

In Section \ref{J_'} it is shown that deficiency indices of perturbed Jacobi matrices $\widehat{{\bf J}}_{X,\gA}$ coincide with that of unperturbed  matrices ${{\bf J}}_{X,\gA}$ (of the form \eqref{IV.2.1_01''} and \eqref{IV.2.1} resp.)  under  additional conditions  (see \eqref{abs1A}--\eqref{abs3A}).
In particular,  for $\alpha=\mathbb O$ condition \eqref{Intro_max_ind}
implies:  $n_\pm({\bf J}_{X,0})=p\Longleftrightarrow\{d_n\}_{n\in\mathbb N}\in l^1(\mathbb N)$.

In Section \ref{Shr}  we apply  previous discreteness  results on Jacobi matrices ${\bf J}_{X,\gA}^{(1)}(\mathrm {\bf H})$ and ${\bf J}_{X,\gA}$ to establish  selfadjointness and discreteness properties   of  Schr\"{o}dinger and Dirac
operators ${\bf H}_{X,\alpha}$ and ${\bf D}_{X,\alpha}$, respectively, using the above mentioned connection between these operators and Jacobi matrices.

In Section \ref{sec5} we indicate certain conditions on $\gA_n$ and $d_n$ that guaranty intermediate
deficiency indices for ${{\bf J}}_{X,\gA}$ and  $\widehat{{\bf J}}_{X,\gA}$.
In particular,  we present  selfadjointness conditions  for  ${{\bf J}}_{X,\gA}$ and  $\widehat{{\bf J}}_{X,\gA}$.

In Section \ref{sec6} we  compare condition \eqref{Intro_max_ind} and similar condition on $\gB$
with the conditions of
Theorem \ref{KosMirTh} (= \cite[Theorem 1]{KosMir01}) also ensuring equalities $n_\pm({\bf J}) = p$.
In particular, we show that matrices
${{\bf J}}_{X,\gA}$   never meet conditions of  \cite[Theorem 1]{KosMir01}.
Finally, in Section \ref{sec7} we show that
Dukarev's result from \cite{Dyuk10} is a special case of
our Proposition \ref{JBabs}  on matrices  ${\bf J}_{X,\gB}$ with zero diagonal
$(\beta_n=-d_n\mathbb I_p)$ and a special choice of
$\{d_n\}_1^\infty$, $d_n\sim\frac{1}{(n+1)\sqrt{n^2+1}}$.

The main results of the paper were announced without proofs in  \cite{BudMal19}, \cite{BudMal20}.

\vskip15pt\noindent \textbf{Notations.} Throughout  the paper  $\mathfrak{H}$ and $\cH$
denote  separable Hilbert spaces.
$\mathcal{C}(\mathfrak{H})$ and  $\widetilde{\mathcal{C}}(\mathfrak{H})$ are the sets of closed
linear operators and linear  relations in $\mathfrak{H}$, respectively; $\dom (T)$ denotes  the domain  of $T\in \mathcal{C}(\mathfrak{H})$. Let  $I_p$,
and $\mathbb O_{p}$ be the unit and zero operators in $\mathbb C^p$, respectively.
For any subset $I$ of $\Z$ we put $l^2(I;\mathbb C^p) := l^2(I)\otimes
\mathbb C^p$; $l^2_0(I;\mathbb C^p)$ is a
subset of finite sequences in $l^2(I;\mathbb C^p)$. We also put $L^{2}([x_{n-1},x_{n}];\C^{2p}):=L^{2}[x_{n-1},x_{n}]\otimes\C^{2p}$ and $L^2(\mathcal I;\mathbb C^{2p}):=L^2(\mathcal I)\otimes\mathbb C^{2p}$. Let also $\mathbb N_0 := \mathbb N\cup\{0\}$.

\begin{center}
\Large{\bf Part I. Abstract results on the deficiency indices and discreteness property  of general  block Jacobi matrices}
\end{center}
\section{The selfadjointness conditions for general block Jacobi matrices}\label{J_s}
Consider the block Jacobi matrix

\begin{equation}\label{Jm}
{\bf J}=\left(
    \begin{array}{ccccccccccc}
      \mathcal A_0 & \mathcal B_0 & \mathbb O_p & \mathbb O_p & \mathbb O_p & \ldots & \mathbb O_p & \mathbb O_p & \mathbb O_p & \mathbb O_p & \ldots\\
      \mathcal B_0^* & \mathcal A_1 & \mathcal B_1 & \mathbb O_p & \mathbb O_p & \ldots & \mathbb O_p & \mathbb O_p & \mathbb O_p & \mathbb O_p & \ldots \\
      \mathbb O_p & \mathcal B_1^* & \mathcal A_2 & \mathcal B_2 & \mathbb O_p & \ldots & \mathbb O_p& \mathbb O_p & \mathbb O_p & \mathbb O_p & \ldots\\
      \vdots & \vdots & \vdots & \vdots & \vdots & \ldots & \vdots & \vdots & \vdots & \vdots & \ldots \\
      \mathbb O_p &\mathbb O_p &\mathbb O_p &\mathbb O_p &\mathbb O_p &\ldots &\mathcal B_{n-1}^* &\mathcal A_n &\mathcal B_n & \mathbb O_p & \ldots\\
      \vdots & \vdots & \vdots & \vdots & \vdots & \vdots & \vdots & \vdots & \vdots & \vdots & \ddots \\
            \end{array}
  \right),
\end{equation}
with $\mathcal A_n=\mathcal A_n^*$, $\mathcal B_n \in\mathbb C^{p\times p}$ and invertible matrices $\mathcal B_n$,
i.e. $\det \mathcal B_n \not =0$,  $n\in \Bbb N_0$.

As usual we identify $\bf J$ with the minimal (closed) symmetric Jacobi operator associated in $l^2(\mathbb N_0;\mathbb C^p)$  in a standard way with matrix $\bf J$ (see \cite{Akh, Ber68}).

\begin{definition}[\cite{Kato66}]
Let $K$ and $T$ be densely defined linear operators in  $\frak H$. Then the operator $K$
is said to be subordinate to the operator $T$ if $\dom T\subset\dom K$ and the following
inequality holds
  \begin{equation}\label{subbord}
\|Ku\|\leq a\|Tu\|+b\|u\|,\qquad a>0,\ b\geq0, \ u\in\dom T.
 \end{equation}
It is  said  that $K$ is strongly subordinated  to the operator $T$ if \eqref{subbord}
holds with $a<1$.
\end{definition}

By the Kato-Rellich theorem (see \cite[Theorem 5.4.3]{Kato66}), $\dom(T+K)=\dom T$,
whenever $K$ is  strongly subordinated to $T$. Moreover, $K$ is symmetric and $T = T^*$,
then the operator $T+K$ is also selfadjoint on $\dom T$. Moreover, according to
W\"{u}st's theorem (see \cite[Theorem X.14]{RedSim78II}), $T+K$ is essentially
selfadjoint on $\dom T$ whenever  $K$ is subordinated  to $T=T^*$ with $a=1$ (see
\eqref{subbord}).

\begin{theorem}\label{Th_self_adj}
Let ${\bf J}$ be the block Jacobi matrix of the form \eqref{Jm} and  let $\mathcal A:=\diag\{\mathcal A_0,\ldots,\mathcal A_n,\ldots\}$,  $\ker\mathcal A=\{0\}$. Let for some $N\in\mathbb N_0$
\begin{equation}\label{self_adj}
a_1(N):=\underset{n\geq N}{\sup}\Big(\|\mathcal A_n^{-1} \mathcal
B_{n}\|+\|\mathcal A_{n}^{-1}\mathcal B_{n-1}^*\|\Big)<\infty,
\end{equation}
\begin{equation}\label{self_adj0}
a_2(N):=\underset{n\geq N}{\sup}\Big(\|\mathcal A_{n}^{-1}\mathcal B_{n}\|+\|\mathcal A_{n+2}^{-1}\mathcal B_{n+1}^*\|\Big)<\infty.
\end{equation}
If
\begin{equation}\label{adj}
{a_1(N)a_2(N)}\leq1,
\end{equation}
then the operator ${\bf J}$ is essentially selfadjoint on $\dom\mathcal A(\subset
l^2(\mathbb N_0;\mathbb C^p))$. Herewith ${\bf J}$ is selfadjoint on $\dom{\bf
J}=\dom\mathcal A$ provided that inequality  \eqref{adj} is strict, i.e.
${a_1(N)a_2(N)}<1$.
\end{theorem}
\begin{proof}
Introduce the block Jacobi submatrix   ${\bf J}_N$ of the matrix  ${\bf J}$, by setting
 \begin{equation}\label{J_N}
 {\bf J}_N := \left(
\begin{array}{cccccc}
      \mathcal A_{N} & \mathcal B_{N} &   \mathbb O_p & \ldots & \mathbb O_p & \ldots \\
      \mathcal B_{N}^* & \mathcal A_{N+1} &   \mathcal B_{N+1} & \ldots & \mathbb O_p & \ldots  \\
      \mathbb O_p & \mathcal B_{N+1}^* &   \mathcal A_{N+2} & \ldots & \mathbb O_p& \ldots \\
      \vdots & \vdots &  \vdots & \ldots & \vdots & \ddots
                  \end{array}
                  \right).
\end{equation}

$(i)$ First, let $N=0$ and let $K$ be the minimal symmetric operator associated to  the
Jacobi matrix
\begin{equation}\label{K}
K=\left(
       \begin{array}{cccccc}
         \mathbb O_p & \mathcal B_0 & \mathbb O_p & \mathbb O_p & \mathbb O_p & \ldots \\
         \mathcal B_0^* & \mathbb O_p & \mathcal B_1 & \mathbb O_p &  \mathbb O_p & \ldots \\
        \mathbb O_p & \mathcal B_1^* & \mathbb O_p & \mathcal B_2  & \mathbb O_p & \ldots \\
         \mathbb O_p & \mathbb O_p & \mathcal B_2^* &\mathbb O_p  & \mathcal B_3 & \ldots \\
         \ldots & \ldots & \ldots & \ldots & \ldots & \ldots\\
       \end{array}
     \right).
\end{equation}
Then on finite vectors the following equality holds
\begin{equation}\label{repres}
{\bf J}h={\bf J}_0h=\mathcal Ah+Kh,\qquad h\in l^2_0(\mathbb N_0;\mathbb C^p).
\end{equation}
Moreover, the operator $\mathcal A^{-1}K$  on the lineal $l_0^2(\mathbb N_0;\mathbb C^p)$  is given by the matrix
\begin{equation}\label{K_1}
\mathcal A^{-1}K=\left(
       \begin{array}{cccccc}
         \mathbb O_p & \mathcal A_0^{-1}\mathcal B_0 & \mathbb O_p & \mathbb O_p & \mathbb O_p & \ldots \\
         \mathcal A_1^{-1}\mathcal B_0^* & \mathbb O_p & \mathcal A_1^{-1}\mathcal B_1 & \mathbb O_p &  \mathbb O_p & \ldots \\
        \mathbb O_p & \mathcal A_2^{-1}\mathcal B_1^* & \mathbb O_p & \mathcal A_2^{-1}\mathcal B_2  & \mathbb O_p & \ldots \\
         \mathbb O_p & \mathbb O_p & \mathcal A_3^{-1}\mathcal B_2^* &\mathbb O_p  & \mathcal A_3^{-1}\mathcal B_3 & \ldots \\
         \ldots & \ldots & \ldots & \ldots & \ldots & \ldots\\
       \end{array}
     \right).
\end{equation}
According to the Schur's test (\cite[Theorem 2.10.1]{BirSol87}), conditions \eqref{self_adj} and \eqref{self_adj0}  guarantee
 boundedness of matrix \eqref{K_1} in  $l^2(\mathbb N_0;\mathbb C^p)$ and the estimate
$\|\mathcal A^{-1}Kh\| \le \sqrt{a_1(N)a_2(N)}\|h\|\leq\|h\|$, $h\in l_0^2(\mathbb
N_0;\mathbb C^p)$.

Since $K\subseteq K^*$ and $\mathcal A=\mathcal A^*$, then $K\mathcal A^{-1} \subset
K^*\mathcal A^{-1} \subseteq (\mathcal A^{-1}K)^*$ (see \cite[VIII.115]{RissSek79}).
Therefore, the operator $\mathcal A^{-1}K$ is bounded on $l_0^2(\mathbb N_0;\mathbb
C^p)$ and its closure $\overline {\mathcal A^{-1}K}\in\mathcal B(l^2(\mathbb N_0;\mathbb
C^p))$, and $\|K\mathcal A^{-1}\|=\|(\mathcal A^{-1}K)^*\|\leq1$ provided \eqref{adj}
holds. Hence  $\dom\mathcal A = \ran \mathcal A^{-1} \subseteq \dom K$ and
\begin{equation}\label{K_sub}
\|Kf\|=\|K\mathcal A^{-1}\mathcal Af\|\leq\|K\mathcal A^{-1}\|\cdot\|\mathcal
Af\|\leq\|\mathcal Af\|,\quad f\in\dom\mathcal A\subseteq \dom K.
\end{equation}
Therefore, the operator $K$ is subordinated  to the operator $\mathcal A$ with constants
$a=1$ and $b=0$ (see~\eqref{subbord}). By the W\"{u}st's theorem,
operator $\mathcal A+K\,(\subseteq{\bf J})$ is essentially selfadjoint on $\dom\mathcal
A$. Thus, the  operator ${\bf J}={\bf J}_{\min}=\overline{\mathcal A+K}$ is selfadjoint.
For $a<1$ and $b=0$  the Kato-Rellich theorem applies and ensures  selfadjointness of
the operator $\mathcal A+K\,(\subseteq{\bf J})$ on  $\dom(\mathcal A+K)=\dom\mathcal A$,
i.e.  ${\bf J}= {\bf J}^*$ on $\dom {\bf J} = \dom\mathcal A$.

$(ii)$ Let $N>0$. Then operator ${\bf J}$ admits the following representation
  \begin{equation}\label{J-decompositionN}
{\bf J}  =  {\bf J}_N'\oplus{\bf J}_N + \mathcal B_{N-1}',
   \end{equation}
   in which
\begin{equation}\label{J_N'}
{\bf J}_N' :=\left(
\begin{array}{ccccccc}
      \mathcal A_0 & \mathcal B_0 &  \mathbb O_p & \ldots &\mathbb O_p & \mathbb O_p & \mathbb O_p \\
      \mathcal B_0^* & \mathcal A_1 & \mathcal B_1  & \ldots &\mathbb O_p & \mathbb O_p & \mathbb O_p  \\
            \vdots & \vdots &   \vdots & \ldots& \vdots & \vdots & \vdots  \\
            \mathbb O_p &\mathbb O_p   &\mathbb O_p &\ldots & \mathcal B_{N-3}^*&\mathcal A_{N-2} &\mathcal B_{N-2}\\
      \mathbb O_p &\mathbb O_p   &\mathbb O_p &\ldots &\mathbb O_p &\mathcal B_{N-2}^* &\mathcal A_{N-1}
                  \end{array}
                  \right)
\end{equation}
and $\mathcal B_{N-1}'$  is a bounded block matrix with the only nontrivial entries
$\mathcal B_{N-1}$ and $\mathcal B_{N-1}^*$, located in $Nth$ and $(N+1)th$ rows.
Clearly  ${\bf J}_N'= ({\bf J}_N')^*$. Now it follows from
\eqref{J-decompositionN} that
$$
n_{\pm}\left({\bf J}\right)  = n_{\pm}\left({\bf J}_N'\oplus{\bf J}_N\right) =
n_{\pm}\left({\bf J}_N\right)=0.
$$
This completes the proof.
\end{proof}

\begin{corollary}\label{cor_adj}
Let under the conditions of Theorem \ref{Th_self_adj} at least one of the following conditions be satisfied:

$(i)$ $a_1(N)\leq1$ and $a_2(N)\leq1$,

$(ii)$ $a_1(N)\leq1$ and $a_2(N)<1$ $($$a_1(N)<1$ and $a_2(N)\leq1$$)$.

Then the operator $\bf J$ is  selfadjoint.
\end{corollary}
\begin{corollary}\label{Th_self_adj1}
Let ${\bf J}$ be the block Jacobi matrix of the form \eqref{Jm},  let $\mathcal A:=\diag\{\mathcal A_0,\ldots,\mathcal A_n,\ldots\}$,  $\ker\mathcal A=\{0\}$,  and let $s\in[1;+\infty)$.
Assume also that  for some $N\in\mathbb N_0$ at least one of the following conditions holds:
 \begin{equation}\label{self_adj1}
(i)\quad b_1(N,s):=\underset{n\geq N}{\sup}\Big(\|\mathcal A_{n}^{-1}\mathcal B_{n}\|^s +
\|\mathcal A_{n}^{-1}\mathcal B_{n-1}^*\|^s\Big)\leq\frac{1}{2^{s-1}};
\end{equation}
\begin{equation}\label{self_adj*}
(ii)\quad b_2(N,s):=\underset{n\geq N}{\sup}\Big(\|\mathcal A_{n}^{-1}\mathcal B_{n}\|^s+\|\mathcal A_{n+2}^{-1}\mathcal B_{n+1}^*\|^s\Big)\leq\frac{1}{2^{s-1}};
\end{equation}
   \begin{equation}\label{adj1}
(iii)\quad\underset{n\geq N}{\sup}\|\mathcal A_{n}^{-1}\mathcal B_{n}\|\leq\frac{1}{2},\qquad\underset{n\geq N}{\sup}\|\mathcal A_{n}^{-1}\mathcal B_{n-1}^*\|\leq\frac{1}{2}.\qquad\qquad\
  \end{equation}
Then the operator ${\bf J}$ is
essentially selfadjoint on $\dom\mathcal A$. Herewith ${\bf J}$ is selfadjoint on
$\dom{\bf J}=\dom\mathcal A$ provided that inequalities \eqref{self_adj1} and \eqref{self_adj*} are strict.
   \end{corollary}
\begin{proof}
It  follows from the simple inequality $(a+b)^s \le 2^{s-1}(a^s + b^s)$, \ $a,b>0$, that
condition \eqref{self_adj1} (resp. \eqref{self_adj*}) implies condition \eqref{self_adj}
(resp. \eqref{self_adj0}). Besides, the implication \eqref{adj1} $\Longrightarrow$
\eqref{self_adj1} is obvious. Theorem~\ref{Th_self_adj}  completes the proof.
  \end{proof}

  \begin{remark}\label{s-a-remark}
 $(i)$ Corollary \ref{Th_self_adj1} shows  that conditions \eqref{self_adj1} and \eqref{self_adj*} are more restrictive
 than conditions \eqref{self_adj} and \eqref{self_adj0}.
 Moreover, in accordance with the power-mean inequality  the  functions $b_1(N,s)$ and $b_2(N,s)$  monotonically increase in $s$,
 i.e. $s_1 < s_2 \Longrightarrow b_j(N,s_1) < b_j(N,s_2)$, $j\in \{1,2\}$.

$(ii)$ Corollary \ref{Th_self_adj1}(ii) in the case of strict inequality in
\eqref{self_adj*} and $s=2$, was proved by another method in~\cite{BroMir18}. Namely,
the proof in~\cite{BroMir18} substantially relies  on the theory of matrix orthogonal
polynomials of the second kind.

$(iii)$ Other conditions for Jacobi matrix ${\bf J}$  to be selfadjoint were obtained in
the recent publications by G. {\'{S}widerski}  and B. Trojan  \cite[Theorem A]{SwidTro20} ($p=1$) and {\'{S}widerski} \cite[Theorem 1]{Swiderski18} ($p>1$).
\end{remark}

\section{Discreteness property of general block  Jacobi matrices}\label{J_d}
The following statement complements known results on the spectrum of perturbations
(see e.g. \cite[Theorems 4.1.16, 4.3.17, 6.3.4]{Kato66}).
Recall that as usual $\mathcal S_q(\frak H)$ denotes the Neumann-Schatten ideal in $\frak H$.
\begin{lemma}\label{lem_discreteness}
Let $T = T^*\in\cC(\frak H)$ and let $K(\subset K^*)$ be a symmetric operator in $\frak H$
strongly subordinated to $T$, i.e. $\dom K\supset  \dom T$ and estimate \eqref{subbord} is  fulfilled with $a\in(0,1)$.
Then for each $q\in (0, \infty]$ the following implication holds.
   \begin{equation}\label{eq:implicat-n_on_S_p-classes}
(T - \lambda)^{-1}\in {\bf \mathcal S}_q(\frak H) \Longrightarrow  (R -
\lambda)^{-1}\in {\bf \mathcal S}_q(\frak H).
   \end{equation}
In particular, if $T$ has discrete spectrum, then the operator  $R:= T+K (= R^*)$  is
also discrete.
    \end{lemma}
     \begin{proof}
By the Kato-Rellich theorem $R = R^*$ and $\dom R = \dom T$. Clearly, alongside
\eqref{subbord} similar estimate holds with $T - \lambda I$ in place of $T$, the same
$a$, and $b+|\lambda|$ in place of $b$.

Besides, it easily follows  (and also known) that $K$ is subordinated to the operator $R
- \lambda I$ with $(R - \lambda I)$-bound $a_{R -\lambda I}(K) \le a(1-a)^{-1}$. Hence
$K(R - \lambda)^{-1} \in \cB(\frak H)$ for $\lambda\in \Bbb C_+$ and the resolvent
difference admits the factorization
   \begin{equation}\label{Krein_type_for-la}
(R - \lambda)^{-1} - (T - \lambda)^{-1} = (T - \lambda)^{-1} \left(K (R -
\lambda)^{-1}\right), \qquad \lambda\in\Bbb C_+.
   \end{equation}
Since $(T - \lambda)^{-1}\in {\bf \mathcal S}_q(\frak H)$  and   $K(R -
\lambda)^{-1} \in \cB(\frak H)$, one gets that $(R - \lambda)^{-1}\in {\bf \mathcal
S}_q(\frak H)$.

In particular, for $q=\infty$ just proved implication
\eqref{eq:implicat-n_on_S_p-classes} means that the operator $R:= T+K (= R^*)$  is
discrete whenever $T$ is.
     \end{proof}

Next we present  several  conditions for  Jacobi operator to be  discrete.
\begin{theorem}\label{d_spec_J}
Let ${\bf J}$ be the minimal Jacobi operator associated in  $l^2(\mathbb N_0;\mathbb C^p)$ with
a block Jacobi matrix~\eqref{Jm} and let $\mathcal
A:=\diag\{\mathcal A_0,\mathcal A_1,\ldots,\mathcal A_n,\ldots\}$, $\ker\mathcal A=\{0\}$. Assume that
$\mathcal A^{-1}\in\mathcal S_q(l^2(\mathbb N_0;\mathbb C^p))$ and  $q\in (0, \infty]$. Let also $a_1(N)$ and $a_2(N)$ be  given by
\eqref{self_adj}~and~\eqref{self_adj0}, respectively, and let at least one of the
following conditions is  satisfied:

$(i)$  the following inequality holds
\begin{equation}
a_1(N)a_2(N)<1;
\end{equation}

$(ii)$ $a_1(N)\leq1$ and $a_2(N)<1$ $($$a_1(N)<1$ and $a_2(N)\leq1$$)$;

$(iii)$ for some $N\in\mathbb N_0$ and $s\in[1;+\infty)$
\begin{equation}\label{self_adj1a}
\underset{n\geq N}{\sup}\Big(\|\mathcal A_{n}^{-1}\mathcal B_{n}\|^s+\|\mathcal A_{n}^{-1}\mathcal B_{n-1}^*\|^s\Big)<\frac{1}{2^{s-1}};
\end{equation}

$(iv)$ for some $N\in\mathbb N_0$ and $s\in[1;+\infty)$
\begin{equation}\label{self_adj*a}
\underset{n\geq N}{\sup}\Big(\|\mathcal A_{n}^{-1}\mathcal B_{n}\|^s+\|\mathcal A_{n+2}^{-1}\mathcal B_{n+1}^*\|^s\Big)<\frac{1}{2^{s-1}};
\end{equation}

$(v)$ for some $N\in\mathbb N_0$
 \begin{equation}\label{adj1ep}
\underset{n\geq N}{\sup}\|\mathcal A_{n}^{-1}\mathcal B_{n}\|<\frac{1}{2},\qquad\underset{n\geq N}{\sup}\|\mathcal A_{n}^{-1}\mathcal B_{n-1}^*\|<\frac{1}{2}.
  \end{equation}
 Then the operator $\bf J$ is self-adjoint and  ${\bf J}^{-1}\in\mathcal S_q(l^2(\mathbb N_0;\mathbb C^p))$.
\end{theorem}
\begin{proof}
In accordance with the proof of Theorem \ref{Th_self_adj},  the operator $K$ of the form
\eqref{K} is strongly subordinate to $\mathcal A=\mathcal A^*$ provided that at least one  of the
conditions $(i)$ -- $(v)$ holds. Since  $\mathcal A^{-1}\in\mathcal S_q(l^2(\mathbb N_0;\mathbb C^p))$,  Lemma~\ref{lem_discreteness} applies and ensures the inclusion ${\bf J}^{-1}\in\mathcal S_q(l^2(\mathbb N_0;\mathbb
C^p))$.

In particular, if $\mathcal A$  is discrete, then so is $\bf J$.
\end{proof}

As usual, the modulus of the matrix $\mathcal A_n$  is denoted by $|\mathcal A_n|
:= \sqrt{\mathcal A_n^2}$.
  \begin{theorem}\label{J_disc}
Let $\bf J$ be the minimal Jacobi operator associated in $l^2(\mathbb N_0;\mathbb C^p)$
with a block Jacobi matrix of the form \eqref{Jm} and let $\mathcal A := \diag\{\mathcal
A_0,\mathcal A_1,\ldots,\mathcal A_n,\ldots\}(=\mathcal A^*)$.

$(1)$  Assume that  $0\in\rho({\mathcal A})$,  $s\in[1;+\infty)$, and let
 at least one of the following conditions be also  satisfied
 \begin{equation}\label{self_adjA}
(i)\quad\limsup_{n\to\infty}\left(\big\||\mathcal A_{n}|^{-1/2}\cdot\mathcal
B_{n}\cdot|\mathcal A_{n+1}|^{-1/2}\big\|+\big\||\mathcal A_{n+2}|^{-1/2}\cdot\mathcal
B_{n+1}^*\cdot|\mathcal A_{n+1}|^{-1/2}\big\|\right)<1;\quad
\end{equation}
\begin{equation}\label{self_adj1h}
(ii)\quad\limsup_{n\to\infty}\Big(\big\||\mathcal A_{n}|^{-1/2}\cdot\mathcal B_{n}\cdot|\mathcal A_{n+1}|^{-1/2}\big\|^s+\big\||\mathcal A_{n}|^{-1/2}\cdot\mathcal B_{n-1}^*\cdot|\mathcal A_{n-1}|^{-1/2}\big\|^s\Big)<\frac{1}{2^{s-1}};
\end{equation}
 \begin{equation}\label{cond_disc}
(iii)\quad\limsup_{n\to\infty}\big\||\mathcal A_{n+1}|^{-1/2} \cdot\mathcal B_{n}^*\cdot
|\mathcal A_n|^{-1/2}\big\|<\frac{1}{2}.\quad\quad\quad\quad\quad\quad\quad\quad\quad\quad\quad\quad\quad\quad\qquad\quad
  \end{equation}

 Then the operator $\bf J$ is symmetric in $l^2(\mathbb N_0;\mathbb C^p)$
  with equal deficiency indices and  admitting  a selfadjoint  extension ${\bf {\wt J}} = {\bf {\wt J}}^*$
  with  $0\in\rho(\bf {{\wt J}})$.

$(2)$  If, in addition,  $\mathcal A^{-1}\in\mathcal S_q(l^2(\mathbb N_0;\mathbb
C^p))$, with some $q\in(0;\infty]$, then the resolvent of any selfadjoint extension is
of class $\mathcal S_q(l^2(\mathbb N_0;\mathbb C^p))$, in particular, ${\bf {\wt
J}}^{-1}\in\mathcal S_q(l^2(\mathbb N_0;\mathbb C^p))$. If,  moreover,  ${\bf J}= {\bf
J}^*$, then  ${\bf J}^{-1}\in\mathcal S_q(l^2(\mathbb N_0;\mathbb C^p))$.

$(3)$  If $\mathcal A$ has a discrete spectrum,  then any selfadjoint extension of
$\bf J$, including  $\bf {\wt J}$, also has a discrete spectrum.
 In particular, if $\bf J = {\bf J}^*$, then its spectrum is discrete.
\end{theorem}
\begin{proof}
 $(1.i)$  Let $U$ denote   the unilateral shift  in $l^2(\mathbb N_0)$:
  \begin{equation}\label{U_shift}
Ue_{n}=e_{n+1},\quad U^*e_{n}=e_{n-1}, \quad
e_{-1}=0,\quad n\in\mathbb N_0,
\end{equation}
and let $U_p := U\otimes \mathbb I_p$ be the unilateral shift  in $l^2(\mathbb N_0;\mathbb C^p)$,
where $\mathbb I_p$ is the unit operator in $\mathbb C^p$.
Clearly, $U_p^* = U^*\otimes \mathbb I_p$.
Consider the polar decomposition of the matrix $\mathcal A$
  \begin{equation}\label{polar_rep}
\mathcal A=\frak J|\mathcal A|=|\mathcal A|^{1/2}\frak J\,|\mathcal A|^{1/2},\qquad
\frak J=\mathrm{sign}\mathcal A,
\end{equation}
and set  $\mathcal B:=\diag\{\mathcal B_0,\mathcal B_1,\ldots,\mathcal B_n,\ldots\}$.
With this notation, alongside operator ${\bf J}$ we  define an operator
\begin{equation}\label{eq:def_of_ext-n_wt_J}
{\bf {\wt J}} := \frak J|\mathcal A|^{1/2}\Big(\mathbb I +\frak J|\mathcal
A|^{-1/2}\left(U_p \mathcal B^* + \mathcal B U_p^*\right) |\mathcal
A|^{-1/2}\Big)|\mathcal A|^{1/2}
\end{equation}
on  its maximal domain of definition $\dom ({\bf {\wt J}})$  in $l^2(\mathbb N_0;\mathbb C^p)$.
Clearly, ${\bf {\wt J}}$ is symmetric  and $\dom ({\bf {\wt J}}) \supset l^2_0(\mathbb N_0;\mathbb C^p)$.
We show that ${\bf {\wt J}}$   is a selfadjoint
extension of  ${\bf J}$. To this end we set for brevity
   \begin{equation}\label{eq:def_T}
T:=|\mathcal A|^{-1/2}(U_p\mathcal B^* +  \mathcal B U_p^*)|\mathcal A|^{-1/2},
      \end{equation}
and note that due to \eqref{U_shift} and \eqref{eq:def_T}  the operator $T$ is well defined and symmetric
on $l^2_0(\mathbb N_0;\mathbb C^p)$ and  admits the following block matrix representation with respect to the standard
basis in $l^2(\mathbb N_0;\mathbb C^p)$:
\begin{equation}\label{K_1A}
T=\left(
 \begin{array}{ccccc}
         \mathbb O_p & |\mathcal A_0|^{-1/2}\mathcal B_0|\mathcal A_1|^{-1/2} & \mathbb O_p & \mathbb O_p &  \ldots \\
         |\mathcal A_1|^{-1/2}\mathcal B_0^*|\mathcal A_0|^{-1/2} & \mathbb O_p & |\mathcal A_1|^{-1/2}\mathcal B_1|\mathcal A_2|^{-1/2} & \mathbb O_p &   \ldots \\
         \mathbb O_p & |\mathcal A_2|^{-1/2}\mathcal B_1^*|\mathcal A_1|^{-1/2} & \mathbb O_p & |\mathcal A_2|^{-1/2}\mathcal B_2|\mathcal A_3|^{-1/2}  &  \ldots \\
         \mathbb O_p & \mathbb O_p & |\mathcal A_3|^{-1/2}\mathcal B_2^*|\mathcal A_2|^{-1/2} &\mathbb O_p   & \ldots \\
         \ldots & \ldots & \ldots & \ldots &  \ldots\\
       \end{array}
     \right).
\end{equation}

Now  Schur's test applies and due to symmetry of $T$  condition \eqref{self_adjA} guarantees the boundedness of the minimal (symmetric) operator $T$, associated with the matrix \eqref{K_1A}, and the estimate $\|T\|=1-\varepsilon<1$.
 Thus,  the operator
$\mathbb I + \frak J T$ is boundedly invertible in $l^2(\mathbb N_0;\mathbb C^p)$, i.e.
$0\in\rho(\mathbb I + \frak J T)$. Combining this
fact with  \eqref{eq:def_of_ext-n_wt_J} and definition \eqref{eq:def_T}  implies 
    \begin{equation}\label{eq:inverse_wt_J^{-1}}
{\bf {\wt J}}^{-1}=|\mathcal A|^{-1/2}(\mathbb I + \frak J T)^{-1}|\mathcal
A|^{-1/2}\frak J\in\mathcal B(l^2(\mathbb N_0;\mathbb C^p)).
   \end{equation}
Thus,  $0\in\rho({\bf {\wt J}})$ and  the operator ${\bf {\wt J}}$ is selfadjoint, ${\bf
{\wt J}} = {\bf {\wt J}}^*$.

Further, it is easily seen from \eqref{Jm} and \eqref{eq:def_of_ext-n_wt_J}  that  both operators ${\bf {\wt J}}$   and   ${\bf J}$ coincide on $l^2_0(\mathbb N_0;\mathbb C^p)$, i.e.
   \begin{equation}\label{eq:J-wt_J_on_finite_vectors}
{\bf J}h = \mathcal Ah  + (U_p\mathcal B^* + \mathcal B U_p^*)h  = \frak J|\mathcal
A|^{1/2}(\mathbb I + \frak J T)|\mathcal A|^{1/2}h = {\bf {\wt J}}h,
\qquad h\in l^2_0(\mathbb N_0;\mathbb C^p).
\end{equation}
Since $l^2_0(\mathbb N_0;\mathbb C^p)$ is a core for  ${\bf J}$
(but not necessary for ${\bf {\wt J}}$), we obtain
by taking the closures  in identity \eqref{eq:J-wt_J_on_finite_vectors} that
${\bf J} \subseteq {\bf {\wt J}} = {\bf {\wt J}}^*$.  This inclusion  proves the statement.

$(1.ii)$  The implication $\eqref{self_adj1h}\Longrightarrow\eqref{self_adjA}$ is immediate from inequality $(a+b)^s \le 2^{s-1}(a^s + b^s)$, \  $a,b>0$.

$(1.iii)$  Condition \eqref{cond_disc} obviously  implies condition \eqref{self_adjA}.

$(2)$ Let now $\mathcal A^{-1}\in\mathcal S_q(l^2(\mathbb N_0;\mathbb C^p))$. Then
$|\mathcal A|^{-1/2}\in \mathcal S_{2q}(l^2(\mathbb N_0;\mathbb C^p))$ and in accordance
with \eqref{eq:inverse_wt_J^{-1}}  ${\bf {\wt J}}^{-1} \in \mathcal S_q(l^2(\mathbb
N_0;\mathbb C^p)).$   Since the deficiency indices of ${\bf J}$  are finite, $n_{\pm}({\bf J})\le p$,
the resolvent of any selfadjoint extension is also of class $\mathcal S_q(l^2(\mathbb N_0;\mathbb C^p))$.

$(3)$
Since the property of an operator $T\in \cC(\cH)$  to have  discrete spectrum is
equivalent to the condition  $(T -\lambda)^{-1}  \in {\bf\mathcal S}_\infty(\cH)$, the
statement is immediate from item (2) with $q=\infty$.
    \end{proof}
\begin{corollary}\label{cor_differ_J^{-1}-A^{-1}_in_S_q}
Let the  conditions of Theorem \ref{J_disc} be satisfied and let  $T$  be  the minimal operator
assosiated to the Jacobi matrix  \eqref{K_1A}. Assume in addition that  $T\in \mathcal S_q(l^2(\mathbb
N_0;\mathbb C^p))$ for some $q\in (0,\infty]$.  Then:

 $(i)$ ${\bf {\wt J}}^{-1} - {\mathcal A}^{-1} \in \mathcal S_q(l^2(\mathbb
N_0;\mathbb C^p)).$

$(ii)$ Assume that for some $q\in [1,\infty)$
\begin{equation}\label{C3.4-ABA}
\sum\limits_{n=0}^\infty \big\||\mathcal A_{n}|^{-1/2}\cdot\mathcal
B_{n}\cdot|\mathcal A_{n+1}|^{-1/2}\big\|^q<\infty. \quad  
\end{equation}
Then the inclusion $(i)$ holds.
\end{corollary}
   \begin{proof}
 (i)  Since $T \in\mathcal S_q(l^2(\mathbb N_0;\mathbb C^p))$,  it follows from \eqref{polar_rep} and
\eqref{eq:inverse_wt_J^{-1}} that
    \begin{equation}\label{eq:difference_A^{-1}-wt_J^{-1}}
{\mathcal A}^{-1} - {\bf {\wt J}}^{-1}  =
 |\mathcal A|^{-1/2}\left[\frak J T (\mathbb I + \frak J T)^{-1}\right]|\mathcal
A|^{-1/2}\frak J \in \mathcal S_q(l^2(\mathbb
N_0;\mathbb C^p)).
   \end{equation}
This completes the proof.

(ii) Introducing the infinite diagonal matrix
$$
T_1 := \diag\{|\mathcal A_{0}|^{-1/2}\cdot\mathcal
B_{0}\cdot|\mathcal A_{1}|^{-1/2}, \ldots, |\mathcal A_{n}|^{-1/2}\cdot\mathcal
B_{n}\cdot|\mathcal A_{n+1}|^{-1/2},\ldots\}.
$$
one gets that inequality \eqref{C3.4-ABA} is equivalent to
the inclusion $T_1\in \mathcal S_q(l^2(\mathbb N_0;\mathbb C^p))$. Further,
in accordance with   \eqref{K_1A} the operator  $T$ admits a representation $T= U_p T_1^* + T_1 U_p^*$
that ensures the  inclusion  $T\in \mathcal S_q(l^2(\mathbb N_0;\mathbb C^p))$.
The required statement is now immediate from   \eqref{eq:difference_A^{-1}-wt_J^{-1}}.
\end{proof}
   \begin{corollary}[\cite{Chi62, CojJan07}]\label{cor_Cojuhari_generalization}
Let $p=1$ and let $\bf{J}$ be the scalar Jacobi matrix of the form \eqref{Jm} with
entries $a_n = \cA_n\in \Bbb R$ and $b_n = \cB_n\in \Bbb C$. Assume also that
      \begin{equation}\label{eq:lim_a_n=infty}
 \lim\limits_{n\to\infty}|a_n| = \infty \qquad \text{and}\qquad \limsup_{n\to\infty}\
 \frac{|b_n|^2}{|a_na_{n+1}|}<\frac{1}{4}.
      \end{equation}
If $\bf{J}=\bf{J}^*$, then its spectrum is discrete.
  \end{corollary}
  \begin{proof}
The first condition in \eqref{eq:lim_a_n=infty} means that the diagonal operator $\cA$
is discrete, while the second condition obviously turns into  condition
\eqref{cond_disc} with $p=1$.
  \end{proof}
\begin{remark}\label{chihara}
$(i)$ Corollary \ref{cor_Cojuhari_generalization}
was obtained by Cojuhari and Janas \cite{CojJan07} by another  method.  In other terms,
the discreteness condition in the scalar case is contained in \cite{Chi62}.  If $\bf{J} \not= \bf{J}^*$, then $n_{\pm}({\bf{J}}) = 1$ and each selfadjoint  extension  of $\bf{J}$ is discrete. In this case conditions  \eqref{eq:lim_a_n=infty}
are obsolete for discreteness. 

$(ii)$ In the scalar case ($p=1$), the discreteness of the spectrum of the Jacobi matrix $\bf{J}=\bf{J}^*$ was proved by Janas and Naboko in \cite{JanNab01} in a different method, under the conditions
 $$
 \lim\limits_{n\to\infty}a_n=\infty,\quad \lim\limits_{n\to\infty}\frac{b_n^2+b_{n-1}^2}{a_n^2}<\frac{1}{2},\quad a_n\in\mathbb R,\ b_n>0, \ n\in\mathbb N.
 $$
Note, that the first condition means discreteness of  the diagonal part $\cA$ of $\bf{J}$, while the second one  coincides with condition \eqref{self_adj1a} for  $s=2$.  

$(iii)$
In the scalar case ($p=1$) another discreteness condition  of  Jacobi matrix $\bf{J}$ was established  by
 \'{Swiderski} in  the recent paper \cite[Theorem B (c)]{Swiderski181}.

$(iv)$ It is worth  mentioning  that conditions \eqref{self_adjA} -- \eqref{cond_disc} do not ensure selfadjointness  of $\bf{J}$,
in general. Moreover,  for any $k\le p$ there exist Jacobi matrices  $\bf{J}$   with nontrivial indices
$n_{+}(\bf{J}) = n_{-}({\bf{J}}) = k$ and satisfying  \eqref{cond_disc}.
Simple  examples of such matrices $\bf{J}$  naturally occur in connection with Dirac
operators with point interactions (see Proposition~\ref{max_s-a} below).

$(v)$
Note that for  $q\in [1,2]$  the implication  $\eqref{C3.4-ABA} \Rightarrow T\in \mathcal S_q(l^2(\mathbb N_0;\mathbb C^p))$ in the proof of Corollary  \ref{cor_differ_J^{-1}-A^{-1}_in_S_q}  follows  from
 general known  result  (see e.g. \cite[Lemma XI.9.32]{DunSchw} and \cite[Chapter 3.7]{GohKre}).
    \end{remark}

\section{Abstract results on deficiency indices of perturbed Jacobi matrices}\label{sec4}

Here, along with the matrix \eqref{Jm}, we consider the Jacobi matrix
\begin{equation}\label{Jm'}
\widehat{{\bf J}}=\left(
    \begin{array}{ccccccccccc}
      \widehat{\mathcal A}_0 & \widehat{\mathcal B}_0 & \mathbb O_p & \mathbb O_p & \mathbb O_p & \ldots & \mathbb O_p & \mathbb O_p & \mathbb O_p & \mathbb O_p & \ldots\\
      \widehat{\mathcal B}_0^* & \widehat{\mathcal A}_1 & \widehat{\mathcal B}_1 & \mathbb O_p & \mathbb O_p & \ldots & \mathbb O_p & \mathbb O_p & \mathbb O_p & \mathbb O_p & \ldots \\
      \mathbb O_p & \widehat{\mathcal B}_1^* & \widehat{\mathcal A}_2 & \widehat{\mathcal B}_2 & \mathbb O_p & \ldots & \mathbb O_p& \mathbb O_p & \mathbb O_p & \mathbb O_p & \ldots\\
      \vdots & \vdots & \vdots & \vdots & \vdots & \ldots & \vdots & \vdots & \vdots & \vdots & \ldots \\
      \mathbb O_p &\mathbb O_p &\mathbb O_p &\mathbb O_p &\mathbb O_p &\ldots &\widehat{\mathcal B}_{n-1}^* &\widehat{\mathcal A}_n &\widehat{\mathcal B}_n & \mathbb O_p & \ldots\\
      \vdots & \vdots & \vdots & \vdots & \vdots & \vdots & \vdots & \vdots & \vdots & \vdots & \ddots \\
            \end{array}
  \right),
\end{equation}
where $\widehat{\mathcal A}_n=\widehat{\mathcal A}_n^*$,  $\widehat{\mathcal B}_n\in \mathbb C^{p\times p}$  and
 $\det \widehat{\mathcal B}_n\not =0$,   $n\in \Bbb N_0$.

 As usual we identify the matrix  $\widehat{{\bf J}}$ with  the  minimal (closed)  symmetric Jacobi operator  generated  in $l^2(\mathbb N;\mathbb C^p)$ by the matrix \eqref{Jm'} (see \cite{Akh, Ber68}).

 \begin{theorem}\label{abs_th}
Let  $\bf J$ and $\widehat{{\bf J}}$ be the Jacobi matrices given by  \eqref{Jm} and \eqref{Jm'}, respectively.
Let for some $N\in\mathbb N_0$ the following conditions hold:
\begin{equation}\label{abs1}
(i)\quad\underset{n\geq N}{\sup}\|\mathbb I_p-\widehat{\mathcal B}_n^*(\mathcal B_n^*)^{-1}\|_{\mathbb
C^{p\times p}}=a_N<1;\qquad\qquad\qquad\qquad\qquad\qquad\qquad\qquad\qquad\qquad\qquad\
\end{equation}
\begin{equation}\label{abs2}
(ii)\quad\underset{n\geq N}{\sup}\|\widehat{\mathcal B}_n-\widehat{\mathcal B}_{n-1}^*(\mathcal B_{n-1}^*)^{-1}\mathcal B_n\|_{\mathbb C^{p\times p}}=C_B<\infty;\qquad\qquad\qquad\qquad\qquad\qquad\qquad\qquad\qquad
\end{equation}

 $(iii)$\  Let for each  $\varepsilon>0$ there exist $N_1 = N_1(\varepsilon)\in\mathbb N$  and $C'_A(\varepsilon)>0$  such that for each vector $f\in l_0^2(\mathbb N_0;\mathbb C^p)$
\begin{equation}\label{abs3}
\sum\limits_{n\geq N_1}\|(\widehat{\mathcal A}_n-\widehat{\mathcal B}_{n-1}^*(\mathcal B_{n-1}^*)^{-1}\mathcal A_n)f_n\|_{\mathbb C^p}^2  \leq \varepsilon\|{\bf J}f\|_{l^2}^2 + C_A'(\varepsilon)\|f\|_{l^2}^2.
\end{equation}
Then:

$(a)$ $\dom {\bf J} =  \dom \widehat{{\bf J}}$  and  $n_\pm({\bf J})=n_\pm(\widehat{{\bf J}})$;

$(b)$ Moreover, if $\bf{J}=\bf{J}^*$ and the spectrum of $\bf J$ is discrete, then the spectrum of
$\widehat{{\bf J}}=\widehat{{\bf J}}^*$ is discrete too.
\end{theorem}
\begin{proof}
$(a)$ $(1)$ 
Let $\delta_0:=1-a_N^2\,(>0)$. Choose
$\varepsilon_0<\frac{\delta_0}{1-\delta_0}$ and
\begin{equation}\label{eps_eq}
\varepsilon_1<\frac{\varepsilon_0}{2}\left(\frac{1}{1+\varepsilon_0}-a_N^2\right).
 \end{equation}
 Then we can find  $N_1=N_1(\varepsilon_1)$ such that condition \eqref{abs3} holds with
 $\varepsilon_1$ instead of  $\varepsilon$ and choose   $M\geq\max\{N,N_1\}$. Finally, we define the block Jacobi submatrix  ${\bf J}_M$ of the matrix  ${\bf J}$ of the form \eqref{J_N} with $M$ instead of $N$.
  Similarly, we define the block Jacobi matrix  $\widehat{{\bf J}}_M$ by the same formula
 \eqref{J_N} with $\widehat{\mathcal A}_k$,  $\widehat{\mathcal B}_k$ and $\widehat{\mathcal B}_k^*$ instead of $\mathcal A_k$, $\mathcal B_k$ and $\mathcal B_k^*$, respectively.
 Summing up we obtain from \eqref{Jm}, \eqref{J_N}, \eqref{Jm'} -- \eqref{eps_eq} that for any finite vector $f=\left(
         \begin{array}{cccc}
           f_0 & f_1 & \ldots & f_m \\
         \end{array}
       \right)^\top$ $(\in l^2_0(\mathbb N_0;\mathbb C^p))$
\begin{eqnarray}\label{J-J'}
&&\!\|(\widehat{{\bf J}}_M-{\bf J}_M)f\|_{l^2(\mathbb N_0;\mathbb C^p)}^2\nonumber\\
&=&\!\sum\limits_{n\geq M}\left\|(\mathcal B_{n-1}^*-\widehat{\mathcal B}_{n-1}^*)f_{n-1}+(\mathcal A_n-\widehat{\mathcal A}_n)f_n+(\mathcal B_n-\widehat{\mathcal B}_n)f_{n+1}\right\|_{\mathbb C^p}^2\nonumber\\
&\leq&\!(1+\varepsilon_0)\sum\limits_{n\geq M}\left\|(\mathbb I_p-\widehat{\mathcal B}_{n-1}^*(\mathcal B_{n-1}^*)^{-1})(\mathcal B_{n-1}^*f_{n-1} +\mathcal A_{n}f_{n}+\mathcal B_{n}f_{n+1})\right\|_{\mathbb C^p}^2 \nonumber\\
&+&\!\left(1+\frac{1}{\varepsilon_0}\right)\sum\limits_{n\geq M}\left\|(\mathcal A_n-\widehat{\mathcal A}_n)f_{n} -(\mathbb I_p-\widehat{\mathcal B}_{n-1}^*(\mathcal B_{n-1}^*)^{-1})\mathcal A_{n}f_{n}+(\mathcal B_n-\widehat{\mathcal B}_n)f_{n+1}\right.\nonumber\\
&-&\!\left.(\mathbb I_p-\widehat{\mathcal B}_{n-1}^*(\mathcal B_{n-1}^*)^{-1})\mathcal B_{n}f_{n+1}\right\|_{\mathbb C^p}^2\!\!\!
\\
&\leq&\!(1+\varepsilon_0)\sum\limits_{n\geq M}\left\|\mathbb I_p-\widehat{\mathcal B}_{n-1}^*(\mathcal B_{n-1}^*)^{-1}\right\|_{\mathbb C^p}^2\|({\bf J}_Mf)_n\|_{\mathbb C^p}^2\nonumber\\
&+&\!\!\!\Big(1+\frac{1}{\varepsilon_0}\Big)\!\!\sum\limits_{n\geq M}\!\!\left\| (\widehat{\mathcal A}_n-\widehat{\mathcal B}_{n-1}^*(\mathcal B_{n-1}^*)^{-1}\mathcal A_n)f_n\!+\!(\widehat{\mathcal B}_n-\widehat{\mathcal B}_{n-1}^*(\mathcal B_{n-1}^*)^{-1}\mathcal B_n)f_{n+1}\right\|_{\mathbb C^p}^2
\nonumber\\
&\leq&\!\!\!\!a_N^2(1+\varepsilon_0)\!\sum\limits_{n\geq M}\!\!\|({\bf J}_M f)_n\|_{\mathbb C^p}^2+2\Big(1+\frac{1}{\varepsilon_0}\Big)\!\sum\limits_{n\geq M}\!\!\left[\|(\widehat{\mathcal A}_n-\widehat{\mathcal B}_{n-1}^*(\mathcal B_{n-1}^*)^{-1}\mathcal A_n)f_n\|_{\mathbb C^p}^2\right.\nonumber\\&+&\left.\|(\widehat{\mathcal B}_n-\widehat{\mathcal B}_{n-1}^*(\mathcal B_{n-1}^*)^{-1}\mathcal B_n)f_{n+1}\|_{\mathbb C^p}^2\right]\nonumber\\
&\leq&\!\! \Big(a_N^2(1+\varepsilon_0)+2\Big(1+\frac{1}{\varepsilon_0}\Big)\varepsilon_1\Big)\|{\bf J}_Mf\|_{l^2}^2+2\Big(1+\frac{1}{\varepsilon_0}\Big)\left(C_{A}'(\varepsilon_1)\|f\|_{l^2}^2+C_{B}^2\|f\|_{l^2}^2\right)\nonumber\\
&=&\! c_1\|{\bf J}_Mf\|_{l^2}^2+c_0\|f\|_{l^2}^2.\nonumber
\end{eqnarray}
Here as it follows from \eqref{eps_eq},
   \begin{equation}\label{constant_1< 1}
c_1 = a_N^2(1+\varepsilon_0)+2\left(1+\frac{1}{\varepsilon_0}\right)\!\varepsilon_1
\!<\!a_N^2(1+\varepsilon_0)+\frac{2(1+\varepsilon_0)}{\varepsilon_0}\!\cdot\!\frac{\varepsilon_0}{2}\!\cdot\!
\frac{1-a_N^2(1+\varepsilon_0)}{1+\varepsilon_0}=1
     \end{equation}
and
$c_0=2\left(1+\frac{1}{\varepsilon_0}\right)(C_{A}'(\varepsilon_1)+C_{B}^2)<\infty$. By the definition of the minimal operator ${\bf J}_M = ({\bf J}_M)_{\min}$,    
each vector $f\in\dom {\bf J}_M$ is approximated
by finite vectors $f_n \in l^2_0(\mathbb N;\mathbb C^p)$ in the graph norm of ${\bf J}_M$.
 Therefore estimate  \eqref{J-J'} remains valid  for any  $f\in\dom {\bf J}_M$.
 In particular, this yields  the inclusion  $\dom (\widehat{{\bf J}}_M- {\bf J}_M) \subset \dom {\bf J}_M$, hence  $\dom {\bf J}_M =  \dom \widehat{{\bf J}}_M$.
Since $c_1<1$,    the Kato--Rellich theorem (\cite[Ch. 4, Theorem 9]{BirSol87})
applied to  estimate  \eqref{J-J'}  yields   $n_\pm({\bf J}_M)=n_\pm(\widehat{{\bf J}}_M)$.

$(2)$ Let us establish the equality $n_\pm({\bf J})=n_\pm(\widehat{{\bf J}})$. To do this, we introduce a finite block Jacobi matrix
${\bf J}_M'$ of the form \eqref{J_N'} with $M$ instead of $N$.
 Similarly, we define the block Jacobi matrix
$\widehat{{\bf J}}_M'$ of the form  \eqref{J_N'} with $\widehat{\mathcal A}_k$, $\widehat{\mathcal B}_k$ and $\widehat{\mathcal B}_k^*$ instead of $\mathcal A_k$, $\mathcal B_k$ and $\mathcal B_k^*$, respectively, $k\in\{0,\ldots,M-1\}$.

Then  the operators  ${\bf J}$ and $\widehat{{\bf J}}$ admit the following representations
  \begin{equation}\label{J-decomposition}
{\bf J}  =  {\bf J}_M'\oplus{\bf J}_M + \mathcal B_{M-1}',\qquad \widehat{{\bf J}}  =  \widehat{{\bf J}}_M'\oplus\widehat{{\bf J}}_M + \widehat{\mathcal B}_{M-1}',
   \end{equation}
where $\mathcal B_{M-1}'$ and $\widehat{\mathcal B}_{M-1}'$ are bounded selfadjoint block--matrices
with four nontrivial  entries
$\mathcal B_{M-1}$, $\mathcal B_{M-1}^*$ and $\widehat{\mathcal B}_{M-1}$, $\widehat{\mathcal B}_{M-1}^*$, respectively.
Clearly,  ${\bf J}_M'= ({\bf J}_M')^*$ and $\widehat{{\bf J}}_M'= (\widehat{{\bf J}}_M')^*$. Therefore it follows from
\eqref{J-decomposition} that
$$
n_{\pm}\left({\bf J}\right)  = n_{\pm}\left({\bf J}_M'\oplus{\bf J}_M\right) =
n_{\pm}\left({\bf J}_M\right),\qquad n_{\pm}\left(\widehat{{\bf J}}\right)  = n_{\pm}\left(\widehat{{\bf J}}_M'\oplus\widehat{{\bf J}}_M\right) =
n_{\pm}\left(\widehat{{\bf J}}_M\right).
$$
Therefore one completes the proof by combining these identities  with  the equalities  $n_\pm({\bf J}_M)=n_\pm(\widehat{{\bf J}}_M)$  established  in the previous step.

$(b)$ Further, from representations similar to \eqref{J_N}, \eqref{J_N'}, and  \eqref{J-decomposition} it follows that for each $M \ge 1$ operator
${\bf J}_M$ is a finite-dimensional selfadjoint perturbation of the operator ${\bf J}$.  Hence, operators ${\bf J}$ and ${\bf J}_M$ are self-adjoint and discrete at the same time.
On the other hand, the estimates \eqref{J-J'}
and \eqref{constant_1< 1}, obtained in the proof of item $(a)$, show that the operator $K_M := \widehat {\bf J}_M - {\bf J}_M$
is strictly subordinate to ${\bf J}_M$. Therefore Lemma  \ref{lem_discreteness}
applies  and guarantees selfadjointness and  discreteness  property  of the Jacobi operator $\widehat {\bf J}_M$.
Using the representations \eqref{J_N}, \eqref{J_N'}, \eqref{J-decomposition} again, we conclude that the operator
$\widehat {\bf J}$ is a finite-dimensional selfadjoint perturbation of $\widehat {\bf J}_M$,
and, therefore, its spectrum is also discrete.
\end{proof}

 \begin{corollary}\label{abs_cor}
 Let for some $N\in\mathbb N_0$ conditions \eqref{abs1} and \eqref{abs2} of Theorem \ref{abs_th} be fulfilled. Assume also the following condition holds
\begin{equation}\label{abs3'}
\underset{n\geq N}{\sup}\|\widehat{\mathcal A}_n-\widehat{\mathcal B}_{n-1}^*(\mathcal B_{n-1}^*)^{-1}\mathcal A_n\|_{\mathbb C^{p\times p}}=C_A<\infty.
\end{equation}
Then $\dom {\bf J} =  \dom \widehat{{\bf J}}$  and  $n_\pm({\bf J})=n_\pm(\widehat{{\bf J}})$.
If, in addition,  $\bf{J}=\bf{J}^*$ and the spectrum of $\bf J$ is discrete, then
$\widehat{{\bf J}}=\widehat{{\bf J}}^*$  and its  spectrum is  also discrete.
\end{corollary}
\begin{proof}
Condition \eqref{abs3'} implies condition \eqref{abs3}  with $\varepsilon=0$.
\end{proof}

\begin{corollary}\label{abs_cor1}
Let ${\bf J}$ and  $\widehat{{\bf J}}$ be as above and let
 $\mathcal A_n=\widehat{\mathcal A}_n=\mathbb O_p$, $n\in \Bbb N$. Then
 $n_\pm({\bf J})=n_\pm(\widehat{{\bf J}})$  provided that conditions \eqref{abs1} and \eqref{abs2}
 are satisfied.
    \end{corollary}

  \begin{center}
\Large{\bf Part II. Dirac operators and associated Jacobi matrices}
\end{center}

\section{Dirac operators on a finite interval with  maximal deficiency indices}\label{sec2}

Consider  one--dimensional Dirac operator associated  with the differential expression
  \begin{equation}\label{1.2Intro}
\mathrm {\bf D} := -i\,c\,\frac{\rd}{\rd x}\otimes \left(\begin{array}{cc}\mathbb O_p & \mathbb I_p
\\\mathbb I_p & \mathbb O_p\end{array}\right) + \frac{c^{2}}{2}\otimes \left(\begin{array}{cc}\mathbb I_p & \mathbb O_p
 \\\mathbb O_p&
-\mathbb I_p\end{array}\right)
  \end{equation}
in $L^{2}(\cI; \mathbb C^{2p})$.  Here  $c>0$ denotes the velocity of
light.

Let  $d_n:=x_{n}-x_{n-1}>0$. We also assume $X=\{x_{n}\}_{n=0}^\infty (\subset\mathcal I)$ to
be a discrete subset of the interval,  $x_{n-1}<x_{n},\ n\in \N.$ We set $x_{0}=0,\quad b := \sup X\equiv\lim_{n\to\infty}x_n  < \infty$.

Following \cite{GS} (see also \cite{Alb_Ges_88}) we define two
families of symmetric extensions, which  turn out to be closely
related to their non-relativistic counterparts $\delta$- and
$\delta'-$interactions.

To this end we set $f=\{f_1,f_2,\ldots,f_{2p}\}^\top \in L^{2}(\mathcal I;\mathbb C^{2p})$ and
$$
 f=\left(\!
               \begin{array}{c}
                 f_I \\
                 f_{II} \\
               \end{array}
             \!\right),\quad   f_I:=\left\{\!\begin{array}{cccc}
                                                            f_1 & f_2 &\! \ldots\! & f_p
                                                          \end{array}\!\right\}^\top,\ \
 f_{II}:=\left\{\!\begin{array}{cccc}
                                                            f_{p+1} & f_{p+2} &\! \ldots\! &
                                                            f_{2p}
                                                          \end{array}\!\right\}^\top.$$
Here $\top$  means a transpose operation. We define the following
Sobolev spaces
\begin{eqnarray*}
  W^{1,2}(\cI\setminus X; \mathbb{C}^{2p}):=\bigoplus_{n=1}^{\infty}W^{1,2}([x_{n-1},x_n];\mathbb{C}^{2p}),\\
W^{1,2}_{\mathrm{comp}}(\cI\setminus X;\C^{2p}):=\{f\in W^{k,2}(\cI\setminus X;\C^{2p}): \mathrm{supp} f\ \text{is compact in}\ \cI\}.
\end{eqnarray*}

Two families of  Gesztesy-\v{S}eba operators (in short, GS-operators or GS-realizations)
on the interval $(a,b)$ are defined to be the closures of the operators
   \begin{equation}\label{delta}
\begin{split}
 \mathrm{\bf D}_{X,\alpha}^0=& \mathrm {\bf D}\upharpoonright \dom(\mathrm{\bf D}_{X,\alpha}^0),\\
\dom(\mathrm{\bf D}_{X,\alpha}^0) = &\Big\{f\in W^{1,2}_{\comp}(\cI \backslash X; \C^{2p}): f_I\in AC_{\loc}(\cI),\  f_{II}\in AC_{\loc}(\cI\backslash X);\\&
 f_{II}(a+)=0\,,\quad  f_{II}(x_{n}+)- f_{II}(x_{n}-)
=-\frac{i\alpha_{n}}{c}f_I(x_{n}),\,\,n\in\N\Big\},
\end{split}
  \end{equation}
and
  \begin{equation}
\begin{split}\label{deltap}
{\mathrm {\bf D}}_{X,\beta}^0 = & \mathrm {\bf D} \upharpoonright \dom({\mathrm D}_{X,\beta}^0), \\
\dom({\mathrm {\bf D}}_{X,\beta}^0) = & \Big\{f\in W^{1,2}_{\comp}(\cI \backslash X;\C^{2p}): f_I\in AC_{\loc}(\cI\backslash X),\  f_{II}\in AC_{\loc}(\cI);\\&
 f_{II}(a+)=0\,,\quad f_I(x_{n}+)-f_I(x_{n}-) = i\beta_{n}c
f_{II}(x_{n}),\,\,n\in\N\Big\},
\end{split}
\end{equation}
respectively, i.e. $\mathrm{\bf D}_{X,\alpha} =  \overline{ \mathrm{\bf D}_{X,\alpha}^0}$ and  $\mathrm{\bf D}_{X,\beta} =  \overline{ \mathrm{\bf D}_{X,\beta}^0}$.
It is easily seen that both operators  $\mathrm{\bf D}_{X,\alpha}$ and  $\mathrm{\bf D}_{X,\beta}$ are  symmetric,  but not necessarily self-adjoint, in general.

In the sequel  we need he following proposition established in
\cite{CarMalPos13} ($p = 1$) and \cite{BudMalPos17,BudMalPos18} ($p>1$).
%
\begin{proposition}[\cite{CarMalPos13,BudMalPos17,BudMalPos18}]\label{1}
Let $\mathrm{\bf D}_{X,\alpha}$ be  realization of the Dirac operator in
$L^2(\cI;\mathbb C^{2p})$. Then the operator $\mathrm{\bf D}_{X,\alpha}^* := (\mathrm{\bf D}_{X,\alpha})^*$ adjoint to the symmetric
operator $\mathrm{\bf D}_{X,\alpha}$  is given by
   \begin{equation}\label{delta_adjoint}
\begin{split}
 \mathrm{\bf D}_{X,\alpha}^*=& \mathrm {\bf D}\upharpoonright \dom(\mathrm{\bf D}_{X,\alpha}^*),\\
\dom(\mathrm{\bf D}_{X,\alpha}^*) = &\Big\{f\in W^{1,2}(\cI \backslash X;\mathbb C^{2p}): f_I\in
AC_{\loc}(\cI),\  f_{II}\in AC_{\loc}(\cI\backslash X);\\&  f_{II}(a+)=0\,,\quad
 f_{II}(x_{n}+)- f_{II}(x_{n}-)
=-\frac{i\alpha_{n}}{c}f_I(x_{n}),\,\,n\in\N\Big\}.
\end{split}
  \end{equation}
  Similarly, the operator $\mathrm{\bf D}_{X,\beta}^*$ adjoint  to $\mathrm{\bf D}_{X,\beta}$ is given by the
expression \eqref{deltap}  with  $W^{1,2}_{\comp}(\cI \backslash X;\mathbb C^{2p})$
replaced by $W^{1,2}(\cI \backslash X; \mathbb C^{2p})$.
   \end{proposition}

\subsection{Realizations ${\mathrm {\bf D}}_{X,\gA}$ with maximal deficiency indices}
It is well known that GS-realizations $\mathrm{\bf D}_{X,\alpha}$ and $\mathrm{\bf D}_{X,\beta}$ are  selfadjoint whenever $|\cI| = \infty$ (see \cite{CarMalPos13, BudMalPos17, BudMalPos18}).
Here  assuming  that $|\cI| < \infty$,
we  construct a class of  $2p\times 2p$-matrix  GS-realizations $\mathrm{\bf D}_{X,\alpha}$ with  maximal deficiency indices
$n_\pm(\mathrm{\bf D}_{X,\alpha}) = p$. Our main result in this direction reads as follows.
    \begin{theorem}\label{VarIndices}
 Let   $\gA:=\{\alpha_n\}_1^\infty (\subset\C^{p\times p})$  be a sequence of selfadjoint
matrices, $\alpha_n=\alpha_n^*$, $n\in\N$.  Then $n_{\pm}(\mathrm{\bf D}_{X,\alpha})=p$
provided that  the following condition holds
\begin{equation}\label{5.26}
\sum_{n=2}^{\infty}d_{n}\prod_{k=1}^{n-1}\left(1+\frac 1c\,\|\alpha_k\|_{\mathbb
C^{p\times p}}\right)^{2}<+\infty.
\end{equation}
\end{theorem}
\begin{proof}
(i) We  examine the  operator  %
\begin{equation}
T_{X,\alpha}:=\mathrm{\bf D}_{X,\alpha}-\frac{c^{2}}{2}\otimes\left(
\begin{matrix}\mathbb I_p&\mathbb O_p\\
\mathbb O_p&-\mathbb I_p\end{matrix}\right)=-i\,c\,\frac{\rd}{\rd x}\otimes\left(
\begin{matrix}\mathbb O_p&\mathbb I_p\\
\mathbb I_p&\mathbb O_p\end{matrix}\right)
\end{equation}
since obviously $n_{\pm}(\mathrm{\bf D}_{X,\alpha}) = n_{\pm}(T_{X,\alpha}).$
It suffices to show that under the  assumption \eqref{5.26} the equation
$(T^{*}_{X,\alpha}\pm i)F=0$ has a non-trivial matrix $L^2(\cI, \C^{2p\times p})$-solution of full rank, i.e. $\mathrm{rank}\,F(x)=p$.

We restrict ourselves to the case of equation  $(T^{*}_{X,\alpha}+ i)F=0$, which   is equivalent to
the system

\begin{equation}\label{dif_eq}
\frac{\rd}{\rd x}F=\frac{1}{c}AF, \quad A=\left(
                                                    \begin{array}{cc}
                                                      \mathbb O_p & \mathbb I_p \\
                                                      \mathbb I_p & \mathbb O_p \\
                                                    \end{array}
                                                  \right).
\end{equation}
Here
$F=\left(
\begin{array}{cccc}
F^1 & F^2 & \ldots & F^p \\                                                                                              \end{array}                                                                                            \right)\in\C^{2p\times p}$ and
$$F^j=\left(
                                      \begin{array}{c}
                                        F_I^j \\
                                        F_{II}^j \\
                                      \end{array}
                                    \right),\qquad  F_I^j=\left(
                                                       \begin{array}{c}
                                                       f_1^j \\
                                                      \vdots \\
                                                      f_p^j  \\                                                 \end{array}
                                                     \right),\qquad  F_{II}^j=
\left(
                                      \begin{array}{c}
                                       f_{p+1}^j \\
                                        \vdots \\
                                        f_{2p}^j \\
                                      \end{array}
                                    \right),\ \ j\in \{1,\ldots,p\}.$$
  Equation \eqref{dif_eq}  has the following  piecewise smooth $2p\times p$-matrix solutions
$$
F=\bigoplus_{n=1}^{\infty}\left(
      \begin{array}{c}
        F_{I,n} \\
        F_{II,n} \\
      \end{array}
    \right),
$$
where
  \begin{equation}\label{FI}
F_{I}=\bigoplus_{n=1}^{\infty}F_{I,n}\,,\quad
F_{I,n}(x)=\mathcal U_{n}e^{-(x_{n}-x)/c}+\mathcal V_{n}e^{(x_{n}-x)/c}\,,\quad x\in [x_{n-1},x_{n}]\,,
  \end{equation}
\begin{equation}\label{FII}
F_{II}=\bigoplus_{n=1}^{\infty}F_{II,n}\,,\quad
F_{II,n}(x)=\mathcal U_{n}e^{-(x_{n}-x)/c}-\mathcal V_{n}e^{(x_{n}-x)/c}\,,\quad x\in [x_{n-1},x_{n}]\,,
\end{equation}
where $\{\mathcal U_{n}\}_{1}^{\infty}\subset\C^{p\times p}$, $\{\mathcal V_{n}\}_{1}^{\infty}\subset\C^{p\times p}$.

According to the description of ${\dom}(\mathrm {\bf D}^{*}_{X,\alpha})$ (see Proposition  \ref{1})
each component $F^j$ of $F = \binom{F_I}{F_{II}}$, $j\in \{1,\ldots,p\}$,
 should satisfy boundary conditions \eqref{delta_adjoint}.
First we  find  recursive relations for  matrix sequences $\{\mathcal U_{n}\}_{1}^{\infty}$, $\{\mathcal V_{n}\}_{1}^{\infty}$
 that ensure satisfying these  conditions.

The  condition $F_{II,n}(x_{0}+)=\mathbb O_p$ yields  $\mathcal U_{1}e^{-d_{1}/c}-\mathcal V_{1}e^{d_{1}/c}=\mathbb O_p$. Further,  the  condition
$$
F_{I,n}(x_{n}+)=F_{I,n}(x_{n}-), \qquad n\in \N,
$$
due to the formulas \eqref{FI}, \eqref{FII} is transformed into
\begin{equation}\label{UV_cont}
\mathcal U_{n+1}e^{-d_{n+1}/c}+\mathcal V_{n+1}e^{d_{n+1}/c}=\mathcal U_{n}+\mathcal V_{n}\,, \qquad n\in \N.
\end{equation}
Moreover,  the jump condition
$$
F_{II,n}(x_{n}+)-F_{II,n}(x_{n}-)=-i\,\frac{\alpha_{n}}{c}\,F_{I,n}(x_{n}),\qquad n\in \N,
$$
due to the formulas \eqref{FI}, \eqref{FII} is equivalent to
\begin{equation}\label{UV_jump}
\mathcal U_{n+1}e^{-d_{n+1}/c}-\mathcal V_{n+1}e^{d_{n+1}/c}-(\mathcal U_{n}-\mathcal V_{n})=-i\,\frac{\alpha_{n}}{c}\,(\mathcal U_{n}+\mathcal V_{n})\,.
\end{equation}
Combining \eqref{UV_cont} and \eqref{UV_jump} we arrive to  the following recursive equations
  \begin{eqnarray}\label{5.28}
\mathcal U_{n+1}=\left(\mathcal U_{n}-i\,\frac{\alpha_{n}}{2c}\,(\mathcal U_{n}+\mathcal V_{n})\right)\,e^{d_{n+1}/c}\,,\qquad n\in \N,  \\
\mathcal V_{n+1}=\left(\mathcal V_{n}+i\,\frac{\alpha_{n}}{2c}\,(\mathcal U_{n}+\mathcal V_{n})\right)\,e^{-d_{n+1}/c}\,,\qquad n\in \N,
\end{eqnarray}
for sequences $\{\mathcal U_{n}\}_{1}^{\infty}$ and $\{\mathcal V_{n}\}_{1}^{\infty}$
with the following initial data
\begin{equation}\label{int_dat}
\mathcal U_{1}= e^{d_{1}/c}\mathbb I_p \quad \text{and} \quad  \mathcal V_{1}= e^{-d_{1}/c}\mathbb I_p\,.
\end{equation}

(ii) Let us  prove that $\mathrm{rank}\, F_n(x) = p$ for each $x\in[x_{n-1},x_n]$ and $n\in \Bbb N$.
It follows from \eqref{FI} and \eqref{FII} that
  \begin{align}\label{F*F}
&\left(
    \begin{array}{cc}
      F_{I,n}^*(x) & F_{II,n}^*(x) \\
    \end{array}
  \right)
\cdot \left(
        \begin{array}{c}
          F_{I,n}(x) \\
          F_{II,n}(x) \\
        \end{array}
      \right)       \nonumber\\
& =\left(
                       \begin{array}{cc}
                         \mathcal U_{n}^*e^{-(x_{n}-x)/c}+\mathcal V_{n}^*e^{(x_{n}-x)/c} & \mathcal U_{n}^*e^{-(x_{n}-x)/c}-\mathcal V_{n}^*e^{(x_{n}-x)/c} \\
                       \end{array}
                     \right)\left(
                              \begin{array}{c}
                                \mathcal U_{n}e^{-(x_{n}-x)/c}+\mathcal V_{n}e^{(x_{n}-x)/c} \\
                                \mathcal U_{n}e^{-(x_{n}-x)/c}-\mathcal V_{n}e^{(x_{n}-x)/c} \\
                              \end{array}
                            \right)\nonumber\\
&=\mathcal U_{n}^*\mathcal U_{n}e^{-2(x_{n}-x)/c}+\mathcal U_{n}^*\mathcal V_{n}+\mathcal V_{n}^*\mathcal U_{n}+\mathcal V_{n}^*\mathcal V_{n}e^{2(x_{n}-x)/c}+\mathcal U_{n}^*\mathcal U_{n}e^{-2(x_{n}-x)/c}-\mathcal U_{n}^*\mathcal V_{n}-\mathcal V_{n}^*\mathcal U_{n}+
\mathcal V_{n}^*\mathcal V_{n}e^{2(x_{n}-x)/c}\nonumber\\
&= 2(\mathcal U_{n}^*\mathcal U_{n}e^{-2(x_{n}-x)/c}+\mathcal V_{n}^*\mathcal V_{n}e^{2(x_{n}-x)/c}).
\end{align}
It follows from \eqref{F*F}, that for each $h\in\C^{p}$  and $x\in[x_{n-1},x_n]$
 \begin{equation}\label{UV0}
\left\|\left(
        \begin{array}{c}
          F_{I,n}(x) \\
          F_{II,n}(x) \\
        \end{array}
      \right)h\right\|^2_{\C^{p}} =
2e^{-2(x_{n}-x)/c}\|\mathcal U_{n}h\|_{\C^{p}}^2 + 2e^{2(x_{n}-x)/c}\|\mathcal V_{n}h\|_{\C^{p}}^2.
  \end{equation}
It follows that for any fixed  $a\in[x_{n-1},x_n]$ the following  important equivalence holds
  \begin{equation}\label{5.eq:equi-ce_ker_F_n(x)--kerU_n=kerV_n}
F_n(a)h=0 \Longleftrightarrow  \mathcal U_{n}h=\mathcal V_{n}h=0.
   \end{equation}
Let us prove by induction that $\mathrm{rank}\,F_{n}(x) = p$ for each $x\in[x_{n-1},x_n]$.
For $n=1$ it follows from \eqref{int_dat}.
 Assuming  that  $\mathrm{rank}\,F_{k}(x) =p$ for each $x\in[x_{k-1},x_k]$ let us prove
that $\mathrm{rank}\,F_{k+1}(x) = p$ for  $x\in[x_{k},x_{k+1}]$.
Assuming  the contrary
we find $\widehat x\in[x_{k},x_{k+1}]$ and  vector $h\in\C^{p}$ such that $F_{k+1}(\widehat x)h=0$.
Combining  this relation with equivalence \eqref{5.eq:equi-ce_ker_F_n(x)--kerU_n=kerV_n} (with $a=\widehat x$ and $n=k+1$)
yields
$
\mathcal U_{k+1}h = \mathcal V_{k+1}h =0.
$

In turn, inserting these relations in   \eqref{UV_cont} and   \eqref{UV_jump} implies
$$
\mathcal U_{k}h = \mathcal V_{k}h =0.
$$
Due to  equivalence \eqref{5.eq:equi-ce_ker_F_n(x)--kerU_n=kerV_n}   these  relations are equivalent to
$F_{k}(x)h = 0$ for  $x\in[x_{k-1},x_{k}]$ which  contradicts the induction hypothesis.
Thus,  $\mathrm{rank}\,F(x) =p$ for each $x\in \Bbb R_+$.

Note also that  in passing we have proved that the matrices $F_{I}(x)$ and $F_{II}(x)$ are nonsingular  for each $x\in\Bbb R_+$.

(iii)  It remains to check that under condition  \eqref{5.26} the inclusion $F_{I}, F_{II}\in L^{2}(\cI;\C^{p\times p})$ holds. It follows from  \eqref{FI} and \eqref{FII}
   \begin{align*}
\|F_{k}\|_{2}^{2}=&\sum_{n=1}^{\infty}\|F_{k,n}\|^{2}_{2}\le{2}
\sum_{n=1}^{\infty}\int_{0}^{d_{n}}(\|\mathcal U_{n}\|_{\C^{p\times p}}^{2}e^{-2x/c}+ \|\mathcal V_{n}\|_{\C^{p\times p}}^{2}e^{2x/c})dx\\
=&
c\sum_{n=1}^{\infty}\left(\|\mathcal U_{n}\|_{\C^{p\times p}}^{2}(1-e^{-2d_{n}/c})+\|\mathcal V_{n}\|_{\C^{p\times p}}^{2}(e^{2d_{n}/c}-1)\right),\quad k=\{I,II\}.
\end{align*}
Since $\sum_{n=1}^{\infty}d_{n}=|\cI|<+\infty$, $d_{n}\to0$ and therefore   $(1-e^{-2d_{n}/c})\sim (e^{2d_{n}/c}-1)\sim 2d_{n}/c$ as $n\to \infty$.
This implies inequality  $\|F_{k}\|_{2}<+\infty$ provided that
  \begin{equation}\label{5.30}
\sum_{n=1}^{\infty}\left(\|\mathcal U_{n}\|_{\C^{p\times p}}^{2}+\|\mathcal V_{n}\|_{\C^{p\times p}}^{2}\right)d_{n}<+\infty\,.
  \end{equation}

Let us prove by induction the following estimates
    \begin{equation}\label{5.29}
\|\mathcal U_{n+1}\|_{\C^{p\times p}}, \ \|\mathcal V_{n+1}\|_{\C^{p\times p}} \le \exp\bigl((d_1+\ldots +d_{n+1})/c\bigr)\cdot\prod^n_{k=1}\left(1+\frac{\|\alpha_k\|_{\C^{p\times p}}}{c}\right), \quad n\in \N.
  \end{equation}

For $n=1$ these estimates are obvious. Assume that \eqref{5.29} are proved for $n\le m-1$. Then using \eqref{5.28} for $n=m$ we obtain
    \begin{eqnarray*}
\|\mathcal U_{m+1}\|_{\C^{p\times p}}\le \left(\|\mathcal U_m\|_{\C^{p\times p}}\left(1+\frac{\|\alpha_m\|_{\C^{p\times p}}}{2c}\right) + \frac{\|\alpha_m\|_{\C^{p\times p}}}{2c}\|\mathcal V_m\|_{\C^{p\times p}}\right) e^{d_{m+1}/c}  \nonumber \\
\le \prod^{m-1}_{k=1}\left(1+\frac{\|\alpha_k\|_{\C^{p\times p}}}{c}\right) \left[\left(1+\frac{\|\alpha_m\|_{\C^{p\times p}}}{2c}\right) + \frac{\|\alpha_m\|_{\C^{p\times p}}}{2c}\right] e^{(d_1+ \ldots + d_m +d_{m+1})/c}\nonumber \\
= \exp\bigl((d_1 + \ldots + d_{m+1})/c\bigr)\cdot\prod^m_{k=1}\left(1+\frac{\|\alpha_k\|_{\C^{p\times p}}}{c}\right).
   \end{eqnarray*}
This inequality proves  the inductive hypothesis  \eqref{5.29} for $\mathcal U_{n}$.
The estimate for $\mathcal V_{m+1}$ is proved similarly. Thus, both  inequalities \eqref{5.29} are established.
Combining \eqref{5.30} with  \eqref{5.29} and  the assumption \eqref{5.26}  we conclude  that $F_I, F_{II} \in L^2(\cI;\C^{p\times p})$.
\end{proof}
\begin{corollary}
If condition \eqref{5.26} is fulfilled, then any selfadjoint extension of the operator $\mathrm{\bf D}_{X,\alpha}$ has  discrete spectrum.
\end{corollary}
   \begin{remark}
Clearly, the condition \eqref{5.26} implies  the inclusion  $\{d_n\}\in l^1(\Bbb N)$, i.e. $|\cI| < \infty$.
   \end{remark}

\begin{corollary}\label{corA}
The GS realization $\mathrm{\bf D}_{X,\alpha}$ is symmetric with $n_{\pm}(\mathrm{\bf D}_{X,\alpha}) =p$ whenever  $\{\alpha_n\}_1^\infty\in l^1(\N;\C^{p\times p})$.
\end{corollary}
\begin{proof}
Clearly,  for any positive sequence $\{s_{k}\}_{1}^{\infty}$
$$
\prod_{k=1}^{\infty}(1+s_{k})\le \exp\left(\sum_{k=1}^{\infty}s_{k}\right)\,.
$$
It follows with account of the inclusion $\{\gA_n\}_1^\infty\in l^1(\N;\mathbb C^{p\times p})$
that
\begin{equation*}
\sum_{n=2}^{\infty}d_{n}\prod_{k=1}^{n-1}\left(1 + \frac 1c\,\|\alpha_k\|_{\mathbb
C^{p\times p}}\right)^{2} \le \exp \left(\frac {2}{c}\sum_{k=1}^{\infty}
\|\alpha_k\|_{\mathbb C^{p\times p}}\right) \sum_{n=2}^{\infty}d_{n} \le|\cI|\exp
\left(\frac {2}{c}\sum_{k=1}^{\infty} \|\alpha_k\|_{\mathbb C^{p\times
p}}\right)\,.
\end{equation*}
It remains to apply Theorem \ref{VarIndices}.
\end{proof}

\begin{corollary}
The GS realization $\mathrm{\bf D}_{X,\alpha}$ is symmetric with $n_{\pm}(\mathrm{\bf D}_{X,\alpha}) =p$ whenever
  \begin{equation}\label{5.31}
\limsup_{n\to\infty}\,\frac{d_{n+1}}{d_{n}}\left(1+\frac{\|\alpha_{n}\|_{\mathbb C^{p\times
p}}}c
\right)^{2}<1\,.
   \end{equation}
In particular,  $n_{\pm}(\mathrm{\bf D}_{X,\alpha}) =p$  provided that one of the following  conditions is satisfied
\item $(i)$   $\limsup_{n\to\infty}(d_{n+1}/d_{n})=0$ and  $\sup_{n\in \N}\|\alpha_{n}\|_{\mathbb C^{p\times
p}}<\infty$;
\item $(ii)$  $\limsup_{n\to\infty}(d_{n+1}/d_{n})=:(1/d)$ with  $d>1$ and
$\sup_{n\in \N} \|\alpha_{n}\|_{\mathbb C^{p\times
p}} < c(\sqrt{d}-1).$
\end{corollary}

\begin{proof}
By the ratio test condition \eqref{5.31} yields  the convergence of the series
\eqref{5.26}. It remains to apply Theorem \ref{VarIndices}.
  \end{proof}

  \subsection{Realizations ${\mathrm {\bf D}}_{X,\gB}$ with maximal deficiency indices}
Here  we present  similar results for GS-realizations $\mathrm{\bf D}_{X,\beta}$.

\begin{theorem}\label{VarIndicesBeta}
 Let  $\gB:=\{\beta_n\}_1^\infty \subset\C^{p\times p}$  be a sequence of selfadjoint
 ${p\times p}$-matrices, $\beta_n=\beta_n^*$, $n\in\N$. Then $n_{\pm}(\mathrm{\bf D}_{X,\beta})=p$
 provided that  the following condition is fulfilled
\begin{equation}\label{5.26B}
\sum_{n=2}^{\infty}d_{n}\prod_{k=1}^{n-1}\left(1+c\,\|\beta_k\|_{\mathbb
C^{p\times p}}\right)^{2}<+\infty.
\end{equation}
 \end{theorem}
\begin{proof}
The proof is similar to that of Theorem \ref{VarIndices} and is omitted.
\end{proof}
\begin{corollary}
If condition \eqref{5.26B} is fulfilled, then any selfadjoint extension of the operator $\mathrm{\bf D}_{X,\beta}$ has  discrete spectrum.
\end{corollary}

\begin{corollary}\label{corB}
The GS realization $\mathrm{\bf D}_{X,\beta}$ is symmetric with $n_{\pm}(\mathrm{\bf D}_{X,\beta}) =p$ whenever   $\{\beta_n\}_1^\infty\in l^1(\N;\C^{p\times p})$.
\end{corollary}

\begin{remark}
Clearly, the condition \eqref{5.26B} implies  the inclusion  $\{d_n\}\in l^1(\Bbb N)$, i.e. $|\cI| < \infty$. Note in this connection that in the opposite case, $|\cI| = \infty$, the realization $\mathrm{\bf D}_{X,\beta}$ is always selfadjoint (see \cite{CarMalPos13, BudMalPos17, BudMalPos18}).
   \end{remark}
\begin{corollary}\label{corB1}
The GS realization $\mathrm{\bf D}_{X,\beta}$ is symmetric with $n_{\pm}(\mathrm{\bf D}_{X,\beta}) =p$ whenever
  \begin{equation}\label{5.31B}
\limsup_{n\to\infty}\,\frac{d_{n+1}}{d_{n}}\left(1+c\|\beta_{n}\|_{\mathbb C^{p\times
p}}\right)^{2}<1\,.
   \end{equation}
In particular,  $n_{\pm}(\mathrm{\bf D}_{X,\beta}) =p$  provided that one of the following  conditions is satisfied
\item $(i)$   $\limsup_{n\to\infty}(d_{n+1}/d_{n})=0$ and  $\sup_{n\in \N}\|\beta_{n}\|_{\mathbb C^{p\times
p}}<\infty$;
\item $(ii)$  $\limsup_{n\to\infty}(d_{n+1}/d_{n})=:(1/d)$ with  $d>1$ and
$\sup_{n\in \N} \|\beta_{n}\|_{\mathbb C^{p\times
p}} < c(\sqrt{d}-1).$
\end{corollary}

\section{Jacobi matrices  with maximal deficiency indices generated by Dirac operators with point interactions}\label{J_'}
\subsection{Block Jacobi matrices $\widehat{{\bf J}}_{X,\gA}$ with maximal deficiency indices}
Here we apply results of Sections \ref{sec4} and \ref{sec2} to block  Jacobi matrices
{\small \begin{equation}\label{IV.2.1_01''}
\widehat{{\bf J}}_{X,\gA}\!=\!\left(
\begin{array}{cccccc}
  \mathcal A_0' \!& \widehat{\nu}_0(\mathbb I_p+\mathcal B_0')\! & \mathbb O_p \! & \mathbb O_p\! & \mathbb O_p\! &  \dots\!\\
  \widehat{\nu}_0(\mathbb I_p+\mathcal B_0') \! &  -\widehat{\nu}_0(\mathbb I_p+\mathcal A_1') \!&
   \widetilde{\nu}_0(\mathbb I_p+\mathcal B_1') \!& \mathbb O_p \! &  \mathbb O_p \!& \dots\!\\
  \mathbb O_p \! & \widetilde{\nu}_0(\mathbb I_p +\mathcal B_1') \!& \frac{\alpha_1}{\gd_2}(\mathbb I_p +\mathcal A_2')\! &   \widehat{\nu}_1(\mathbb I_p+\mathcal B_2') \!&  \mathbb O_p \!&  \dots\!\\
  \mathbb O_p \! & \mathbb O_p \! & \widehat{\nu}_1(\mathbb I_p +\mathcal B_2')\!&  -\widehat{\nu}_1(\mathbb I_p+\mathcal A_3')\! &\widetilde{\nu}_1(\mathbb I_p +\mathcal B_3')\!&
    \dots\!\\
  \mathbb O_p \! & \mathbb O_p\!  & \mathbb O_p \! & \widetilde{\nu}_1(\mathbb I_p+\mathcal B_3')\! &\frac{\alpha_2}{\gd_3}(\mathbb I_p+\mathcal A_4')\! &
    \dots\!\\
\dots\!& \dots\!&\dots\!&\dots\!&\dots\!&\dots\!\\
 \end{array}%
\right)\,,
   \end{equation}}
where $\widehat{\nu}_n:=\frac{\nu(\gd_{n+1})}{\gd_{n+1}^{2}}$,\: $\widetilde{\nu}_n:= \frac{\nu(\gd_{n+1})}{\gd_{n+1}^{3/2}\gd_{n+2}^{1/2}}$,\,
$\gA_n=\alpha_n^*$,  $\mathcal A_n' = (\mathcal A_n')^*,\ \mathcal B_n' = (\mathcal B_n')^* \  \in  \mathbb C^{p\times p}$ ($n\in \mathbb N_0$)  and matrices $\mathbb I_p + \mathcal B_n'$ are invertible, i.e.  $\det (\mathbb I_p +\mathcal B_n')\not =0$,\  $n\in\mathbb N_0$
and
   \begin{equation}\label{IV.2.1_01NU}
\nu(x):=\frac{1}{\sqrt{1+(c^2x^2)^{-1}}}=\frac{cx}{\sqrt{1+c^2x^2}}.
   \end{equation}
In what follows we  keep also the notation   $\widehat{{\bf J}}_{X,\gA}$ for the minimal block  Jacobi
operator associated in  a  standard  way  with the matrix $\widehat{{\bf J}}_{X,\gA}$
in $l^2(\mathbb N;\mathbb C^p)$ (see \cite{Akh, Ber68}, and  also \cite{Krein}).
Clearly the operator $\widehat{{\bf J}}_{X,\gA}$ is symmetric and as it is known $\mathrm{n}_\pm(\widehat{{\bf J}}_{X,\gA}) \leq\,p$ (see \cite{Ber68, Krein, Krein49}).

Our investigation of the deficiency indices   $n_{\pm}(\widehat{{\bf J}}_{X,\gA})$
substantially relies on the connection between GS-realizations $\mathrm{\bf D}_{X,\alpha}$ (see \eqref{delta}) of the Dirac operator $\mathrm{\bf D}$ on the one hand and
the block   Jacobi matrices
\begin{equation}\label{IV.2.1}
{\bf J}_{X,\alpha}=\left(
\begin{array}{cccccc}
  \mathbb O_p  & \frac{\nu(\gd_{1})}{\gd_1^{2}}\mathbb I_p & \mathbb O_p  & \mathbb O_p & \mathbb O_p   &  \dots\\
   \frac{\nu(\gd_{1})}{\gd_1^{2}}\mathbb I_p  &  -\frac{\nu(\gd_{1})}{\gd_1^{2}}\mathbb I_p &
   \frac{\nu(\gd_{1})}{\gd_1^{3/2}\gd_2^{1/2}}\mathbb I_p & \mathbb O_p  & \mathbb O_p &  \dots\\
  \mathbb O_p  & \frac{\nu(\gd_{1})}{\gd_1^{3/2}\gd_2^{1/2}}\mathbb I_p  & \frac{\alpha_1}{\gd_2}  &
   \frac{\nu(\gd_{2})}{\gd_2^{2}}\mathbb I_p & \mathbb O_p &   \dots\\
  \mathbb O_p  & \mathbb O_p  & \frac{\nu(\gd_{2})}{\gd_2^{2}}\mathbb I_p &  -\frac{\nu(\gd_{2})}{\gd_2^{2}}\mathbb I_p &
  \frac{\nu(\gd_{2})}{\gd_2^{3/2}\gd_3^{1/2}}\mathbb I_p &  \dots\\
  \mathbb O_p  & \mathbb O_p  & \mathbb O_p  & \frac{\nu(\gd_{2})}{\gd_2^{3/2}\gd_3^{1/2}}\mathbb I_p &
   \frac{\alpha_2}{\gd_3} &   \dots\\
\dots& \dots&\dots&\dots&\dots&\dots\\
 \end{array}%
\right)\,,
   \end{equation}
on the other hand. This connection was first  discovered in \cite{CarMalPos13}  for $p=1$ and extended
to the  matrix  case $(p>1)$ in \cite{BudMalPos17, BudMalPos18}.  It is briefly  explained in Appendix (see Proposition \ref{prop_IV.2.1_01}). 

   \begin{proposition}\label{def_ind_B}
Let  ${\bf J}_{X,\alpha}$  be the minimal Jacobi operator
associated to the Jacobi matrix of the form \eqref{IV.2.1}.
 Let also  a sequence  $\alpha:=\{\alpha_n\}_1^\infty (\subset \mathbb C^{p\times p})$
of selfadjoint matrices satisfy condition  \eqref{5.26}. Then the operator ${\bf J}_{X,\alpha}$ has maximal deficiency indices,  $n_\pm({\bf J}_{X,\alpha})=p$.
  \end{proposition}
    \begin{proof}
 It follows from condition \eqref{5.26}  that $\sum_k d_k =: b<\infty$. Consider the minimal
 Dirac operator ${\mathrm {\bf D}}_{X}$ in  $ L^{2}([0,b]; \C^{2p})$ 
 and the boundary triplet $\Pi = \{\cH, \Gamma_0,\Gamma_1\}$ for ${\mathrm {\bf D}}_{X}^*$  given  by \eqref{IV.1.1_12}.
 By Proposition \ref{prop_IV.2.1_01}, the GS-realization $\mathrm{\bf D}_{X,\alpha}$  is given by \eqref{B-op-D},
 i.e. the  boundary operator of $\mathrm{\bf D}_{X,\alpha}$ with respect to the triplet $\Pi$ is
  the minimal operator associated  to the Jacobi matrix  ${\bf B}_{X,\gA}$ \eqref{IV.2.1_01}.

On the other hand,  ${\bf B}_{X,\gA}$  is unitarily equivalent to the minimal Jacobi operator associated
to the  matrix ${\bf J}_{X,\alpha}$ with positive offdiagonal entries. Hence,
$n_{\pm}({\bf J}_{X,\alpha}) = n_{\pm}({\bf B}_{X,\alpha}).$
 Combining this  equality with  Proposition \ref{prop_IV.2.1_01}
 imply $n_\pm(\mathrm{\bf D}_{X,\alpha})=n_\pm({\bf B}_{X,\gA})=n_{\pm}({\bf J}_{X,\alpha})$.  It remains to apply Theorem \ref{VarIndices}.
\end{proof}
\begin{corollary}\label{crit-max-alpha}
Let  ${\bf J}_{X,0}$  be the minimal Jacobi operator
associated to the Jacobi matrix~\eqref{IV.2.1} with zero $\{\alpha_n\}_1^\infty\equiv\mathbb O$. Then the operator ${\bf J}_{X,0}$  has maximal deficiency indices  if and only if $\{d_n\}_{n\in\mathbb N}\in l^1(\mathbb N)$.
\end{corollary}
  \begin{proof}
Clearly, that for $\alpha=\mathbb O$ the  condition
$\{d_n\}_{n\in\mathbb N}\in l^1(\mathbb N)$ is equivalent to  estimate \eqref{5.26}. Thus, sufficiency follows from
Proposition \ref{def_ind_B}. Necessity is immediate from the Carleman test \eqref{Car}.
\end{proof}
\begin{theorem}\label{abs_thA}
Let $\widehat{{\bf J}}_{X,\alpha}$ be the matrix of the form \eqref{IV.2.1_01''} and the sequence  $\alpha:=\{\alpha_n\}_1^\infty (\subset \mathbb C^{p\times p})$,  $\gA_n=\alpha_n^*$, $n\in \Bbb N,$  satisfy condition  \eqref{5.26}. Assume also that for some $N\in\mathbb N_0$ the following conditions hold:
\begin{equation}\label{abs1A}
(i)\quad\underset{n\geq N}{\sup}\,\|\mathcal B_n'\|_{\mathbb C^{p\times p}}= a_N<1;\qquad\qquad\qquad\qquad\qquad\qquad
\end{equation}
\begin{equation}\label{abs2A}
(ii)\quad\underset{j\geq 0}{\sup}\,\frac{1}{d_{j+1}}\left\|\mathcal B_{2j}'-\mathcal B_{2j-1}'\right\|_{\mathbb C^{p\times p}}=C_B'<\infty,\qquad\qquad\qquad
\end{equation}
\begin{equation}\label{abs2A'}
\underset{j\geq 0}{\sup}\,\frac{1}{(\gd_{j+1}\gd_{j+2})^{1/2}}\left\|\mathcal B_{2j+1}'-\mathcal B_{2j}'\right\|_{\mathbb C^{p\times p}}=C_B''<\infty;
\end{equation}
\begin{equation}\label{abs3A}
(iii)\quad\underset{j\geq 0}{\sup}\,\frac{1}{\gd_{j+1}}\left\|\alpha_j(\mathcal A_{2j}'-\mathcal B_{2j-1}')\right\|_{\mathbb C^{p\times p}}=C_A'<\infty,\qquad\qquad\quad
\end{equation}
\begin{equation}\label{abs3A'}
\underset{j\geq 0}{\sup}\,\frac{1}{d_{j+1}}\left\|\mathcal A_{2j+1}'-\mathcal B_{2j}'\right\|_{\mathbb C^{p\times p}}=C_A''<\infty.\qquad\qquad
\end{equation}
Then the matrix $\widehat{{\bf J}}_{X,\alpha}$  is in the complete indeterminant case, i.e.
$n_\pm(\widehat{{\bf J}}_{X,\alpha})=p$.
\end{theorem}
\begin{proof}
 Let us check conditions of Corollary \ref{abs_cor} for Jacobi matrices $\widehat{{\bf J}}_{X,\alpha}$ and ${\bf J}_{X,\alpha}$  given  by   \eqref{IV.2.1_01''} and \eqref{IV.2.1}, respectively.
First we check condition \eqref{abs1}  considering the case $n=2j+1$ only. Applying  condition  \eqref{abs1A}  to  these matrices   we easily derive
\begin{equation}
 \underset{j\geq N}{\sup}\|\mathbb I_p-\widehat{\mathcal B}_{2j+1}\mathcal B_{2j+1}^{-1}\|_{\mathbb C^{p\times p}}=\underset{j\geq N}{\sup}\left\|\mathbb I_p-\frac{\nu(\gd_{j+1})}{\gd_{j+1}^{3/2}\gd_{j+2}^{1/2}}(\mathbb I_p+\mathcal B_{2j+1}')\frac{\gd_{j+1}^{3/2}\gd_{j+2}^{1/2}}{\nu(\gd_{j+1})}\right\|_{\mathbb C^{p\times p}}
  =\underset{j\geq N}{\sup}\left\|\mathcal B_{2j+1}'\right\|_{\mathbb C^{p\times p}}\leq a_N<1.
 \end{equation}
 The case $n=2j$ is considered similarly.

 Note also that condition \eqref{abs1A}  ensures invertibility of all off-diagonal entries of
 the Jacobi  matrix \eqref{IV.2.1_01''}.

Next we check condition \eqref{abs2}. Consider the case $n=2j$ only.  Applying  condition  \eqref{abs2A} to
the pair  $\{{\bf J}_{X,\alpha},  \widehat{{\bf J}}_{X,\alpha}\}$ and noting that $\nu(d_{j+1}) < c\,d_{j+1}$ and
$\nu(d_{j+1})\thicksim c\,d_{j+1}$ as $d_{j}\to0$, we   obtain
   \begin{align}
 &\underset{j\geq 0}{\sup}\|\widehat{\mathcal B}_{2j}-\widehat{\mathcal B}_{2j-1}\mathcal B_{2j-1}^{-1}\mathcal B_{2j}\|_{\mathbb C^{p\times p}}=\\
 &=\underset{j\geq 0}{\sup}\left\|\frac{\nu(\gd_{j+1})}{d_{j+1}^2}(\mathbb I_p+\mathcal B_{2j}')-\frac{\nu(\gd_{j})}{\gd_{j}^{3/2}\gd_{j+1}^{1/2}}(\mathbb I_p+\mathcal B_{2j-1}')\frac{\gd_{j}^{3/2}\gd_{j+1}^{1/2}\nu(\gd_{j+1})}{\nu(\gd_{j})d_{j+1}^2}\right\|_{\mathbb C^{p\times p}}\nonumber\\
 &=\underset{j\geq 0}{\sup}\frac{\nu(\gd_{j+1})}{d_{j+1}^2}\left\|\mathcal B_{2j}'-\mathcal B_{2j-1}'\right\|_{\mathbb C^{p\times p}}  \leq\underset{j\geq 0}{\sup}\,\frac{c}{d_{j+1}}\left\|\mathcal B_{2j}'-\mathcal B_{2j-1}'\right\|_{\mathbb C^{p\times p}} \leq  C_B'<\infty.\nonumber
 \end{align}

 The case $n=2j+1$ is treated  similarly.

Finally, we check condition \eqref{abs3'}  considering  the  case $n=2j$ only. The case $n=2j+1$ is considered similarly.  Applying  condition  \eqref{abs3A} to
the pair  $\{{\bf J}_{X,\alpha},  \widehat{{\bf J}}_{X,\alpha}\}$
we  derive
   \begin{align}
& \underset{j\geq 0}{\sup}\|\widehat{\mathcal A}_{2j}-\widehat{\mathcal B}_{2j-1}\mathcal B_{2j-1}^{-1}\mathcal A_{2j}\|_{\mathbb C^{p\times p}}=\underset{j\geq 0}{\sup}\left\|\frac{\alpha_j}{d_{j+1}}(\mathbb I_p+\mathcal A_{2j}')-\frac{\nu(\gd_{j})}{\gd_{j}^{3/2}\gd_{j+1}^{1/2}}(\mathbb I_p+\mathcal B_{2j-1}')\frac{\gd_{j}^{3/2}\gd_{j+1}^{1/2}\cdot\alpha_j}{\nu(\gd_{j})d_{j+1}}\right\|_{\mathbb C^{p\times p}}\nonumber\\
  &=\underset{j\geq 0}{\sup}\frac{1}{d_{j+1}}\left\|\alpha_j(\mathcal A_{2j}'-\mathcal B_{2j-1}')\right\|_{\mathbb C^{p\times p}}=C_A'<\infty.
 \end{align}
Thus,  Corollary \ref{abs_cor} together with Proposition \ref{def_ind_B} ensures  $n_{\pm}(\widehat{{\bf J}}_{X,\alpha}) = n_{\pm}({\bf J}_{X,\alpha}) =  p$.
  \end{proof}

Alongside with the matrix ${\bf J}_{X,\alpha}$ of form \eqref{IV.2.1} we  consider the following
 Jacobi matrix
   \begin{equation}\label{B''}
{\bf J}_{X,\gA}'=\left(
\begin{array}{cccccc}
  \mathbb O_p  & \frac{c}{\gd_1}\mathbb I_p & \mathbb O_p  & \mathbb O_p & \mathbb O_p   &  \dots\\
   \frac{c}{\gd_1}\mathbb I_p  &  -\frac{c}{\gd_1}\mathbb I_p &
   \frac{c}{\gd_1^{1/2}\gd_2^{1/2}}\mathbb I_p & \mathbb O_p  & \mathbb O_p &  \dots\\
  \mathbb O_p  & \frac{c}{\gd_1^{1/2}\gd_2^{1/2}}\mathbb I_p  & \frac{\alpha_1}{\gd_2}  &
   \frac{c}{\gd_2}\mathbb I_p & \mathbb O_p &   \dots\\
  \mathbb O_p  & \mathbb O_p  & \frac{c}{\gd_2}\mathbb I_p &  -\frac{c}{\gd_2}\mathbb I_p &
  \frac{c}{\gd_2^{1/2}\gd_3^{1/2}}\mathbb I_p &  \dots\\
  \mathbb O_p  & \mathbb O_p  & \mathbb O_p  & \frac{c}{\gd_2^{1/2}\gd_3^{1/2}}\mathbb I_p &
   \frac{\alpha_2}{\gd_3} &   \dots\\
\dots& \dots&\dots&\dots&\dots&\dots\\
 \end{array}%
\right)\,,
   \end{equation}
obtained from \eqref{IV.2.1}  by replacing $\nu(d_{n})$ by  $cd_n$. Since
$\nu(d_{n})\thicksim c\,d_n$ as $\lim_{n\to\infty} d_n = 0$ it is naturally
to suppose that $n_{\pm}({\bf J}_{X,\alpha}') = n_{\pm}({\bf J}_{X,\alpha})$.
Below we confirm this hypotheses under an additional assumption on  $d_n$ although
we don't know whether it is true in general.
  \begin{proposition}\label{JA-simple}
 Let the Jacobi matrices  ${\bf J}_{X,\alpha}$ and ${\bf J}_{X,\alpha}'$ be of the  form \eqref{IV.2.1} and \eqref{B''}, respectively. Assume that $\lim\limits_{n\to\infty}\frac{d_{n-1}^2}{d_{n}}=0$. Then $n_\pm({\bf J}_{X,\alpha})=n_\pm({\bf J}_{X,\alpha}')$.

In particular, if the
sequence $\alpha:=\{\alpha_n\}_1^\infty (\subset \mathbb C^{p\times p})$,
   satisfy condition  \eqref{5.26}, then $n_{\pm}({\bf J}_{X,\alpha}')=p$.
  \end{proposition}
  \begin{proof}
Let us check the conditions of Theorem \ref{abs_th} for the pair $\{{\bf J}_{X,\alpha},{\bf J}_{X,\alpha}'\}$.

Now, the entries of the matrices  ${\bf J}_{X,\alpha}$ and ${\bf J}_{X,\alpha}'$ read as follows
  \begin{equation}\label{B_form}
\widehat{\mathcal B}_{n}=\left\{\begin{array}{ll}
                           \frac{\nu(\gd_{j+1})}{d_{j+1}^2}\mathbb I_p,&n=2j, \\
                           \frac{\nu(\gd_{j+1})}{\gd_{j+1}^{3/2}\gd_{j+2}^{1/2}}\mathbb I_p,&n=2j+1,
                         \end{array}\right.\quad
\mathcal B_{n}=\left\{\begin{array}{ll}
                        \frac{c}{d_{j+1}}\mathbb I_p, & n=2j, \\
                        \frac{c}{(\gd_{j+1}\gd_{j+2})^{1/2}}\mathbb I_p, & n=2j+1,
                      \end{array}\right.
\end{equation}
and
\begin{equation}
\widehat{\mathcal A}_{n}=\left\{\begin{array}{ll}
                                  \frac{\alpha_j}{d_{j+1}}, & n=2j, \\
                                  -\frac{\nu(\gd_{j+1})}{d_{j+1}^2}\mathbb I_p, & n=2j+1,
                                \end{array}\right.
\quad \mathcal A_{n}=\left\{\begin{array}{ll}
                                  \frac{\alpha_j}{d_{j+1}}, & n=2j, \\
                                  -\frac{c}{d_{j+1}}\mathbb I_p, & n=2j+1,
                                \end{array}\right.
\end{equation}
respectively.

 First we check condition \eqref{abs1}.  Consider the case $n=2j$. Noting that $\nu(d_{j+1}) < c\,d_{j+1}$ and
$\nu(d_{j+1})\thicksim c\,d_{j+1}$ as $d_{j}\to0$, one easily verifies that  $\lim\limits_{j\to\infty}\|\mathbb I_p-\widehat{\mathcal B}_{2j}\mathcal B_{2j}^{-1}\|_{\mathbb C^{p\times p}}=0$. Hence  for sufficiently large $N$ the  condition
  \begin{equation}
 \underset{j\geq N}{\sup}\|\mathbb I_p-\widehat{\mathcal B}_{2j}\mathcal B_{2j}^{-1}\|_{\mathbb C^{p\times p}}<1
 \end{equation}
holds.  The case $n=2j+1$ is treated  similarly.

Next we  check condition \eqref{abs2}. Consider the case $n=2j$. Since $\lim\limits_{j\to\infty}\frac{d_{j}^2}{d_{j+1}}=0$ we get
   \begin{align}
\|\widehat{\mathcal B}_{2j}-\widehat{\mathcal B}_{2j-1}\mathcal B_{2j-1}^{-1}\mathcal B_{2j}\|_{\mathbb C^{p\times p}}
=\left|\frac{c}{d_{j+1}}\left(\frac{1}{\sqrt{1+c^2d_{j+1}^2}}-\frac{1}{\sqrt{1+c^2d_{j}^2}}\right)\right|\nonumber\\
\leq\frac{c^3|d_j^2-d_{j+1}^2|}{d_{j+1}}
\leq c^3\,\max\left\{d_{j+1},\frac{d_j^2}{d_{j+1}}\right\}\to0\quad\text{as}\ j\to\infty.
 \end{align}
 This implies condition \eqref{abs2}. The case $n=2j+1$ is considered similarly.

 Finally, we check condition \eqref{abs3}. In the case $n=2j+1$ this condition is obvious, because $\widehat{\mathcal B}_{2j}=-\widehat{\mathcal A}_{2j+1}$ and $\mathcal B_{2j}^{-1}\mathcal A_{2j+1}=-\mathbb I_p$.

 Consider the  case $n=2j$. Since $\lim_{j\to\infty}d_j = 0$ and  $\lim\limits_{j\to\infty}\frac{d_{j}^2}{d_{j+1}}=0$, for any $\varepsilon>0$ we can find $N=N(\varepsilon)$, such that $c^2 d_j^2<\varepsilon$ and $\frac{d_{j}^2}{d_{j+1}}<\varepsilon$ for all $j\geq N$. Thus, for any $f\in l_0^2(\mathbb N;\mathbb C^p)$ with $f_{2j}=0$ for $j<N$, we obtain
\begin{eqnarray}
 \sum\limits_{j\geq N}\|(\widehat{\mathcal A}_{2j}-\widehat{\mathcal B}_{2j-1}\mathcal B_{2j-1}^{-1}\mathcal A_{2j})f_{2j}\|_{\mathbb C^p}^2\leq\sum\limits_{j\geq N}\left\|\frac{\alpha_j}{d_{j+1}}\left(1-\frac{1}{\sqrt{1+c^2d_j^2}}\right)f_{2j}\right\|_{\mathbb C^p}^2\nonumber\\
 =\sum\limits_{j\geq N}\left\|\left(({\bf J}_{X,\alpha}'f)_{2j}-c\left(\frac{1}{d_{j}^{1/2}d_{j+1}^{1/2}}f_{2j-1}+\frac{1}{d_{j+1}}f_{2j+1}\right)\right)\frac{c^2d_{j}^2}{\sqrt{1+c^2d_{j}^2}\Big(1+\sqrt{1+c^2d_{j}^2}\Big)}\right\|_{\mathbb C^p}^2\\
 \leq\sum\limits_{j\geq N}\left\|c^2d_{j}^2({\bf J}_{X,\alpha}'f)_{2j}\right\|_{\mathbb C^p}^2+c^6\sum\limits_{j\geq N}\left\|d_{j}^{1/2}\cdot\frac{d_{j}}{d_{j+1}^{1/2}}f_{2j-1}+\frac{d_{j}^2}{d_{j+1}}f_{2j+1}\right\|_{\mathbb C^p}^2
  \leq\varepsilon^2\left\|{\bf J}_{X,\alpha}'f\right\|_{l^2}^2+4C_1\|f\|_{l^2}^2,\nonumber
 \end{eqnarray}
 where $C_1=\max\{c^5\varepsilon^{3/2},c^6\varepsilon^2\}$.
 Thus,  Theorem \ref{abs_th} ensures that  $n_{\pm}({\bf J}_{X,\alpha})=n_{\pm}({\bf J}_{X,\alpha}')$.

Combining this relation with Proposition \ref{def_ind_B} yields $n_{\pm}({\bf J}_{X,\alpha}')=n_{\pm}({\bf J}_{X,\alpha}) =p$.
  \end{proof}

  \begin{corollary}\label{crit-max-alpha'}
  Let ${\bf J}_{X,0}'$ be the Jacobi matrix of the form \eqref{B''} with $\{\alpha_n\}_1^\infty\equiv\mathbb O$. Assume also that $\lim\limits_{n\to\infty}\frac{d_{n-1}^2}{d_{n}}=0$. Then the operator ${\bf J}_{X,0}'$ has maximal deficiency indices if and only if $\{d_n\}_{n\in\mathbb N}\in l^1(\mathbb N)$.
  \end{corollary}
  \begin{remark}
  The analogue of Theorem \ref{abs_thA} is also valid  for the  pair $\{\widehat{{\bf J}'}_{X,\alpha},{\bf J}_{X,\alpha}'\}$, in which the matrix $\widehat{{\bf J}}'_{X,\alpha}$ is obtained from ${\bf J}_{X,\alpha}'$ by adding bounded perturbations $\mathcal A_n'$ and $\mathcal B_n'$
     similarly to the construction of the matrix $\widehat{{\bf J}}_{X,\alpha}$ of the form \eqref{IV.2.1_01''}.
  \end{remark}
\begin{remark}
Other classes of block Jacobi matrices  $\bf{J}$ with  maximal  deficiency indices can be  found  in \cite[Theorem 3]{Swiderski18}.
\end{remark}

  \subsection{Jacobi matrices $\widehat{{\bf J}}_{X,\gB}$ with maximal deficiency indices}
Here we apply previous results to block  Jacobi matrix
{\small\begin{equation}\label{IV.2.1_01''B}
\widehat{{\bf J}}_{X,\gB}\!=\!\left(
\begin{array}{ccccc}
  \mathcal A_0'\! & \widehat{\nu}_0(\mathbb I_p+\mathcal B_0')\! & \mathbb O_p\!  & \mathbb O_p\! &   \dots\!\\
   \widehat{\nu}_0(\mathbb I_p+\mathcal B_0')\!  &  -\breve{\nu}_1(\beta_{1} + \gd_{1}\mathbb I_p +\!\mathcal A_1')\! &
   \widetilde{\nu}_0(\mathbb I_p+\mathcal B_1')\! & \mathbb O_p\!  &   \dots\!\\
  \mathbb O_p\!  & \widetilde{\nu}_0(\mathbb I_p +\mathcal B_1')\! & \mathcal A_2'\! & \widehat{\nu}_1(\mathbb I_p+\mathcal B_2')\! &    \dots\!\\
  \mathbb O_p \! & \mathbb O_p \!  & \widehat{\nu}_1(\mathbb I_p +\mathcal B_2')\! &  -\breve{\nu}_2(\beta_{2}+ \gd_{2}\mathbb I_p +\!\mathcal A_3')\!  &
    \dots\! \\
  \mathbb O_p \!  & \mathbb O_p\!   & \mathbb O_p\!   & \widetilde{\nu}_1(\mathbb I_p+\mathcal B_3')\!  &
    \dots\! \\
\dots\! & \dots\! &\dots\! &\dots\! &\dots\! \\
 \end{array}%
\right)\!,
   \end{equation}}
where $\widehat{\nu}_n:=\frac{\nu(\gd_{n+1})}{\gd_{n+1}^{2}}$,\: $\widetilde{\nu}_n:= \frac{\nu(\gd_{n+1})}{\gd_{n+1}^{3/2}\gd_{n+2}^{1/2}}$, $\breve{\nu}_n=\frac{\nu^{2}(\gd_{n})}{d_{n}^{3}}$,
$\gA_n=\alpha_n^*$,  $\mathcal A_n' = (\mathcal A_n')^*,\ \mathcal B_n' = (\mathcal B_n')^* \  \in  \mathbb C^{p\times p}$   and matrices $\mathbb I_p + \mathcal B_n'$ are invertible, i.e.  $\det (\mathbb I_p +\mathcal B_n')\not =0$,\  $n\in\mathbb N_0$.

Consider the minimal Jacobi operator associated
with the matrix
\begin{equation}\label{IV.3.1_04'}
{\bf J}_{X,\gB}:=\left(\begin{array}{ccccc} \mathbb O_p
&\frac{\nu(\gd_{1})}{\gd_{1}^{2}} \mathbb I_p &\mathbb O_p  &\mathbb O_p
&\dots\\
\frac{\nu(\gd_{1})}{\gd_{1}^{2}}\mathbb I_p & -
\frac{\nu^{2}(\gd_{1})}{d_{1}^{3}}\left(\beta_{1} + \gd_{1}\mathbb I_p \right)
&\frac{\nu(\gd_{1})}{\gd_{1}^{3/2}\,{\gd_2}^{1/2}}\mathbb I_p &\mathbb
O_p &\dots\\
\mathbb O_p  & \frac{\nu(\gd_{1})}{\gd_{1}^{3/2}{\gd_2}^{1/2}}\mathbb
I_p  & \mathbb O_p &\frac{\nu(\gd_{2})}{\gd_{2}^{2}}\mathbb
I_p &\dots\\
\mathbb O_p &\mathbb O_p &\frac{\nu(\gd_{2})}{\gd_{2}^{2}}\mathbb
I_p & - \frac{\nu^{2}(\gd_{2})}{d_{2}^{3}}\left(\beta_{2}+ \gd_{2}\mathbb
I_p \right)  &\dots\\
\mathbb O_p &\mathbb O_p &\mathbb
O_p&\frac{\nu(\gd_{2})}{\gd_{2}^{3/2}{\gd_{3}^{1/2}}}\mathbb I_p
  &\dots\\
\mathbb O_p &\mathbb O_p &\mathbb O_p &\mathbb
O_p&\dots\\
\dots&\dots&\dots&\dots&\dots
\end{array}\right).
 \end{equation}
 \begin{proposition}\label{JBabs}
Let   ${{\bf J}}_{X,\gB}$ be a minimal Jacobi operator associated in
 $l^2(\mathbb N_0;\mathbb C^p)$  with the matrix  \eqref{IV.3.1_04'}. If  the sequence  $\beta:=\{\beta_n\}_1^\infty (\subset \mathbb C^{p\times p})$
  satisfy condition  \eqref{5.26B}. Then the deficiency indices of the matrix ${\bf J}_{X,\beta}$  are maximal, i.e.
$n_\pm({\bf J}_{X,\beta})=p$.
 \end{proposition}
\begin{proof}
The proof is similar to that of Proposition \ref{def_ind_B} and is based on the fact that the operator ${{\bf J}}_{X,\gB}$ is unitary equivalent to the boundary operator ${\bf B}_{X,\gB}$ of GS-realisation $\mathrm{\bf D}_{X,\beta}$ in an appropriate 
boundary triplet. Note that  ${\bf B}_{X,\gB}$ differs from ${{\bf J}}_{X,\gB}$  by the signs of certain offdiagonal
entries in just the same way as the matrix ${\bf B}_{X,\gA}$ \eqref{IV.2.1_01} differs from the matrix ${{\bf J}}_{X,\gA}$ \eqref{IV.2.1}
(see \cite[Proposition 5.35]{CarMalPos13} for $p=1$  and  \cite[Proposition 5.1]{BudMalPos17} for $p>1$).
 \end{proof}

\begin{theorem}\label{abs_thB}
Let $\widehat{{\bf J}}_{X,\beta}$ be the matrix of the form \eqref{IV.2.1_01''B} and the sequence  $\beta:=\{\beta_n\}_1^\infty (\subset \mathbb C^{p\times p})$
  satisfy condition  \eqref{5.26B}. Assume also that for some $N\in\mathbb N_0$ the following conditions hold:
\begin{equation}\label{abs1B}
(i)\quad\underset{n\geq N}{\sup}\,\|\mathcal B_n'\|_{\mathbb C^{p\times p}} = a_N<1;\qquad\quad\qquad\qquad\qquad\qquad\qquad\qquad\
\end{equation}
\begin{equation}\label{abs2B}
(ii)\quad\underset{j\geq 0}{\sup}\,\frac{1}{d_{j+1}}\left\|\mathcal B_{2j}'-\mathcal B_{2j-1}'\right\|_{\mathbb C^{p\times p}}=C_B'<\infty,\qquad\quad\qquad\qquad\qquad
\end{equation}
\begin{equation}
\underset{j\geq 0}{\sup}\,\frac{1}{(\gd_{j+1}\gd_{j+2})^{1/2}}\left\|\mathcal B_{2j+1}'-\mathcal B_{2j}'\right\|_{\mathbb C^{p\times p}}=C_B''<\infty;\qquad\quad
\end{equation}
\begin{equation}\label{abs3B}
(iii)\quad\underset{j\geq 0}{\sup}\,\left\|\mathcal A_{2j}'\right\|_{\mathbb C^{p\times p}}=C_A'<\infty,\qquad\quad\qquad\qquad\qquad\qquad\qquad\qquad\
\end{equation}
\begin{equation}
\underset{j\geq 0}{\sup}\,\frac{1}{d_{j+1}}\left\|\mathcal A_{2j+1}'-(\beta_{j+1} + \gd_{j+1}\mathbb I_p)\mathcal B_{2j}'\right\|_{\mathbb C^{p\times p}}=C_A''<\infty.
\end{equation}
Then the deficiency indices of the minimal Jacobi  operator  $\widehat{{\bf J}}_{X,\beta}$  are maximal,
i.e.   $n_\pm(\widehat{{\bf J}}_{X,\beta})=p$.
\end{theorem}
\begin{proof}
The proof is similar to that of Theorem \ref{abs_thA}.
\end{proof}

Next we simplified  the matrix ${\bf J}_{X,\beta}$ of form \eqref{IV.3.1_04'} by considering  the following Jacobi matrix
   \begin{equation}\label{Beta''}
{\bf J}_{X,\beta}'=\left(\begin{array}{ccccc} \mathbb O_p
&\frac{c}{\gd_{1}} \mathbb I_p &\mathbb O_p &\mathbb
O_p &\dots\\
\frac{c}{\gd_{1}}\mathbb I_p & -
\frac{c^2}{d_{1}}\left(\beta_{1} + \gd_{1}\mathbb I_p \right)
&\frac{c}{(\gd_{1}\,{\gd_2})^{1/2}}\mathbb I_p &\mathbb
O_p &\dots\\
\mathbb O_p  & \frac{c}{(\gd_{1}{\gd_2})^{1/2}}\mathbb
I_p  & \mathbb O_p &\frac{c}{\gd_{2}}\mathbb
I_p &\dots\\
\mathbb O_p &\mathbb O_p &\frac{c}{\gd_{2}}\mathbb
I_p & - \frac{c^2}{d_{2}}\left(\beta_{2}+ \gd_{2}\mathbb
I_p \right)  &\dots\\
\mathbb O_p &\mathbb O_p &\mathbb
O_p&\frac{c}{(\gd_{2}{\gd_{3})^{1/2}}}\mathbb I_p
  &\dots\\
\mathbb O_p &\mathbb O_p &\mathbb O_p &\mathbb
O_p&\dots\\
\dots&\dots&\dots&\dots&\dots
\end{array}\right).
   \end{equation}
It is obtained from \eqref{IV.3.1_04'}  by replacing $\nu(d_{n})$ by  $cd_n$.
Our next statement is a counterpart  of Proposition \ref{JA-simple}.
  \begin{proposition}\label{J'beta_p}
 Let the Jacobi matrices  ${\bf J}_{X,\beta}$ and ${\bf J}_{X,\beta}'$ be of the  form \eqref{IV.3.1_04'} and \eqref{Beta''}, respectively. Assume that $\lim\limits_{n\to\infty}\frac{d_{n-1}^2}{d_{n}}=0$. Then $n_\pm({\bf J}_{X,\beta})=n_\pm({\bf J}_{X,\beta}')$.

In particular, if the 
sequence $\beta:=\{\beta_n\}_1^\infty (\subset \mathbb C^{p\times p})$
   satisfy condition  \eqref{5.26B}, then $n_{\pm}({\bf J}_{X,\beta}')=p$.
  \end{proposition}
  \begin{proof}
Let us check the conditions of Theorem \ref{abs_th} for the pair $\{{\bf J}_{X,\beta},{\bf J}_{X,\beta}'\}$.
Now, the entries $\widehat{\mathcal A}_{n}$ and  ${\mathcal A}_{n}$
of the matrices  ${\bf J}_{X,\beta}$ and ${\bf J}_{X,\beta}'$, respectively,  are given by
\begin{equation}
\widehat{\mathcal A}_{n}=\left\{\begin{array}{ll}
                                  \mathbb O_p, & n=2j, \\
                                  -\frac{\nu^2(\gd_{j+1})}{d_{j+1}^3}(\beta_{j+1}+d_{j+1}\mathbb I_p)\mathbb I_p, & n=2j+1,
                                \end{array}\right.
\end{equation}
\begin{equation}
 \mathcal A_{n}=\left\{\begin{array}{ll}
                                   \mathbb O_p, & n=2j, \\
                                  -\frac{c^2}{d_{j+1}}(\beta_{j+1}+d_{j+1}\mathbb I_p), & n=2j+1,
                                \end{array}\right.
\end{equation}
while  $\widehat{\mathcal B}_{n}$ and $\mathcal B_{n}$ are given by  \eqref{B_form}.

 Conditions \eqref{abs1} and \eqref{abs2} can be checked in just the same way as in Proposition \ref{JA-simple}.

 Let us check condition \eqref{abs3}.  In the case $n=2j$ it is obviously satisfied  because $\widehat{\mathcal A}_{2j}=\mathcal A_{2j}=\mathbb O_p$.

 Consider the  case $n=2j+1$.  Since $\lim_{j\to\infty}d_{j+1} = 0$  and $\lim\limits_{j\to\infty}\frac{d_{j+1}^2}{d_{j+2}}=0$, for any $\varepsilon>0$ we can find $N=N(\varepsilon)$, such that $c^2 d_{j+1}^2<\varepsilon$ and $\frac{d_{j+1}^2}{d_{j+2}}<\varepsilon$ for all $j\geq N$. Thus,  for any $f\in l_0^2(\mathbb N;\mathbb C^p)$ with $f_{2j+1}=0$ for $j<N$,  we obtain
     \begin{align}
 &\sum\limits_{j\geq N}\|(\widehat{\mathcal A}_{2j+1}-\widehat{\mathcal B}_{2j}\mathcal B_{2j}^{-1}\mathcal A_{2j+1})f_{2j+1}\|_{\mathbb C^p}^2
 \leq\sum\limits_{j\geq N}\left\|-\frac{c^2}{d_{j+1}}(\beta_{j+1}+d_{j+1}\mathbb I_p)\cdot\frac{1-\sqrt{1+c^2d_{j+1}^2}}{1+c^2d_{j+1}^2}f_{2j+1}\right\|_{\mathbb C^p}^2\nonumber
  \end{align}
 \begin{align}
& =\sum\limits_{j\geq N}\left\|\Big(({\bf J}_{X,\beta}'f)_{2j+1}-\Big(\frac{f_{2j}}{d_{j+1}}+\frac{f_{2j+2}}{d_{j+1}^{1/2}d_{j+2}^{1/2}}\Big)\Big)\frac{c^3d_{j+1}^2}{(1+c^2d_{j+1}^2)\Big(1+\sqrt{1+c^2d_{j+1}^2}\Big)}\right\|_{\mathbb C^p}^2\\
& \leq\sum\limits_{j\geq N}\left\|c^2d_{j+1}^2({\bf J}_{X,\beta}'f)_{2j+1}\right\|_{\mathbb C^p}^2+c^6\sum\limits_{j\geq N}\left\|d_{j+1}f_{2j}+d_{j+1}^{1/2}\cdot\frac{d_{j+1}}{d_{j+2}^{1/2}}f_{2j+2}\right\|_{\mathbb C^p}^2\leq\varepsilon^2\left\|{\bf J}_{X,\beta}'f\right\|_{l^2}^2+4C_1\|f\|_{l^2}^2,\nonumber
 \end{align}
where $C_1=\max\{c^4\varepsilon,c^5\varepsilon^{3/2}\}$. Thus condition \eqref{abs3} is verified.

Applying   Theorem \ref{abs_th}, we conclude that   $n_{\pm}({\bf J}_{X,\beta}) = n_{\pm}({\bf J}_{X,\beta}')$.

The second statement is immediate  from   Proposition \ref{JBabs}.
  \end{proof}
 \begin{remark}
 Note, that if $\beta=\mathbb O_p$. Then Corollaries \ref{crit-max-alpha} and \ref{crit-max-alpha'} remain true for Jacobi matrices ${\bf J}_{X,\alpha}$ and ${\bf J}_{X,\alpha}'$ replaced by ${\bf J}_{X,\beta}$ and ${\bf J}_{X,\beta}'$, respectively.
 \end{remark}
\subsection{Relations to the moment problem}
Here  we demonstrate the role of block Jacobi matrices  with maximal deficiency indices
in matrix moment problem.
Following M. Krein \cite{Krein, Krein49} one associates to the matrix $\bf J$ \eqref{Jm}  a difference matrix expression
%
\begin{equation}\label{LU}
(LV)_n=\mathcal B_{n-1}^*V_{n-1}+\mathcal B_n V_{n+1}+\mathcal A_nV_n, \quad  V_0=\mathbb I_p,\ V_{-1}=\mathbb O_p,\quad  V_n\in \Bbb C^{p\times p},\  n\in\mathbb N_0.
\end{equation}
It is known (see \cite{Krein, Krein49, Ber68}), that  the  solution to the Cauchy problem $(LV)_n=zV_n$ subject to the initial condition \eqref{LU} is a sequence of matrix  polynomials $\{P_n(z)\}_0^{\infty}$.  Subsequently  one finds
\begin{equation}
P_0(z)=\mathbb I,\quad  P_1(z)=\mathcal B_0^{-1}(z\mathbb I-\mathcal A_0),\quad P_2(z)=\mathcal B_1^{-1}((z\mathbb I-\mathcal A_1)P_1(z)\mathcal B_0),\quad \ldots\,.
\end{equation}

M. Krein \cite{Krein} (see also \cite{Ber68}) showed, that for all $z\in\mathbb C_\pm$ there is a matrix limit
\begin{equation}
 H(z)=\lim\limits_{k\to\infty}\left(\sum\limits_{n=0}^kP_n^*(\overline{z})P_n(z)\right)^{-1}\quad\text{and}
 \quad  \mathrm {rank}(H(z))=n_\pm(\bf J).
\end{equation}
  He also established that certain matrix moment problem  is associated  to  every Jacobi matrix  ${\bf J}$ and this problem has a unique (normalized in a sense) solution if $n_-({\bf J})\cdot n_+({\bf J}) =0$.
 Moreover, this case is known as a definite case of the matrix moment problem.  If $n_{\pm} ({\bf J})=p$  (see \cite{Krein}), then the series
  \begin{equation}\label{reprod_kernel_Intro}
 \sum\limits_{n=0}^\infty P_n^*(\overline{z})P_n(z)  =: H^{-1}(z), \qquad z\in \Bbb C,
 \end{equation}
converges uniformly on  compact subsets of $\Bbb C$.
Note, that in this case the defect subspace  $\frak N_{z}$  is:
$$
\frak N_{z} := \ker ({\bf J}^* - zI) = \{\{P_n(z)h\}_0^{\infty}: \  h\in \Bbb C^p\}.
$$
In this case, it is  said that  the matrix ${\bf J}$ (and the corresponding matrix moment problem)
is in  the completely indeterminant case (\cite{Krein, Krein49}, \cite[Ch. VII, \S2]{Ber68}).
In this case for each matrix solution $\Sigma$ to the respective moment problem the series \eqref{reprod_kernel_Intro}  defines a reproducing kernel
for the subspace of entire matrix  functions generated in $L^2(\Bbb R; \Sigma)$ by the matrix
polynomials $\{P_n(z)\}_0^{\infty}$. In particular, the latter happens under the conditions \eqref{5.26} and \eqref{5.26B}
for Jabobi matrices ${\bf J}_{X,\alpha}$ and ${\bf J}_{X,\beta}$, respectively.

 \section{Application to Schr\"{o}dinger and Dirac operators with $\delta$-interactions}\label{Shr}

 \subsection{Discrete Schr\"{o}dinger operators with point interactions}
Here we apply Theorems \ref{d_spec_J} and \ref{J_disc}  to describe discrete  Schr\"{o}dinger operators
with matrix point interactions associated in $L^2(\mathbb{R}_+;\mathbb{C}^p)$ with  formal differential expressions
\begin{equation}\label{I_01}
\ell_{X,\gA}:=-\frac{\rD^2}{\rD x^2}+\sum_{x_{n}\in X}\gA_n\delta(x-x_n).
 \end{equation}
Here   $\delta$ denotes the  Dirac delta-function, $X=\{x_n\}_{1}^\infty\subset \mathcal I = (0,b)$, $b\le \infty$, is a strictly increasing sequence with $x_0:=0$, $x_{n+1}>x_{n}$,  $x_n\to b$,  and $\{\alpha_n\}_1^\infty \subset\mathbb C^{p\times p}$.
For rigorous definition of the minimal operator associated to
expression \eqref{I_01} one  introduces first  a preminimal operator $\mathrm{\bf H}^0_{X,\alpha}$ by setting
   \begin{equation}\label{deltaopH}
\begin{gathered}
\mathrm{\bf H}^0_{X,\alpha}:=-\frac{d^2}{dx^2}\otimes \mathbb I_p,\\
\dom(\mathrm{\bf H}^0_{X,\alpha})= \left\lbrace f\in W^{2,2}_{\comp} (\mathbb{R}_+\setminus X;\mathbb{C}^p): \begin{array}{c} f'(0)=0,\ f(x_k+)=f(x_k-)\\ f'(x_k+)-f'(x_k-)=\alpha_k f(x_k) \end{array} \right\rbrace.
\end{gathered}
\end{equation}
The  minimal operator  $\mathrm{\bf H}_{X,\alpha}$  associated  to~\eqref{I_01}  is defined to be the closure of $\mathrm{\bf H}^0_{X,\alpha}$ in $L^2(\mathbb{R}_+;\mathbb{C}^p)$.

It is established in \cite{KM, KM10} for $p=1$, and then in \cite{KMN} for arbitrary $p>1$,   that certain  spectral properties of the minimal operator $\mathrm{\bf H}_{X,\alpha}$, associated  to expression~\eqref{I_01}, are identical to that
of the minimal Jacobi operator ${\bf J}_{{X,\alpha}}^{(1)}(\mathrm{\bf{H}})$, associated in $l^2(\mathbb N_0;\mathbb C^p)$  to the block Jacobi matrix of the form
   \begin{equation}\label{IV.2.1_05}
{\bf J}_{{X,\alpha}}^{(1)}(\mathrm{\bf{H}})\!=\!\left(\begin{array}{cccccc}
\!\frac{1}{r_1^{2}}\,\widetilde{\alpha}_1\! & \!\frac{1}{r_1r_2\gd_2}\mathbb I_p\! & \!\mathbb O_p\!&   \mathbb O_p\!&   \mathbb O_p\!&   \dots\!\\
\!\frac{1}{r_1r_2\gd_2}\mathbb I_p \!&\!\frac{1}{r_2^{2}}\,\widetilde{\alpha}_2\! &\! \frac{1}{r_2r_3\gd_3}\mathbb I_p\! & \! \mathbb O_p\!& \! \mathbb O_p\!& \! \dots\!\\
\!\mathbb O_p\! &\! \frac{1}{r_2r_3\gd_3}\mathbb I_p\! &\! \frac{1}{r_3^{2}}\,\widetilde{\alpha}_3\!&\!  \frac{1}{r_3r_4\gd_4}\mathbb I_p\!& \! \mathbb O_p\!& \! \dots\!\\
\mathbb O_p & \mathbb O_p & -\frac{1}{r_3r_4\gd_4}\mathbb I_p& \frac{1}{r_4^2}\widetilde{\alpha}_4& \frac{1}{r_4r_5\gd_5}\mathbb I_p& \! \dots\!\\
\!\dots\! &\! \dots\!&\! \dots\! & \!\dots\!& \! \dots\!& \! \dots\!
\end{array}\right)\!.\!
    \end{equation}
Here $d_n:=x_n-x_{n-1}$, $r_n:=\sqrt{d_n+d_{n+1}}$ and
\begin{equation}\label{alpha-tilde}
\widetilde{\alpha}_n:=\alpha_n+\Big(\frac{1}{\gd_n}+\frac{1}{\gd_{n+1}}\Big)\mathbb I_p,\qquad n\in\mathbb N.
\end{equation}
In particular, it is proved in  \cite{KM, KM10} (for $p=1$) and \cite{KMN} (for $p > 1$) (see also Proposition \ref{Bound-op-A})
that:
   \begin{equation}\label{n(H)=n(J)}
\mathrm{(a)}\  n_\pm(\mathrm{\bf H}_{X,\alpha})=n_\pm({\bf J}_{{X,\alpha}}^{(1)}(\mathrm{\bf{H}}));\qquad\qquad\qquad\qquad\qquad\qquad\qquad\qquad\qquad\qquad\qquad\qquad\qquad\quad
\end{equation}

$(b)$ If  ${\bf J}_{{X,\alpha}}^{(1)}(\mathrm{\bf{H}})=({\bf J}_{{X,\alpha}}^{(1)}(\mathrm{\bf{H}}))^*$
and  $\lim\limits_{n\to\infty}d_n=0$, then the operators $\mathrm{\bf H}_{X,\alpha}$ and
${\bf J}_{{X,\alpha}}^{(1)}(\mathrm{\bf{H}})$ are discrete simultaneously.

Now we complete investigation  of $\mathrm{\bf H}_{X,\alpha}$ from \cite{KMN} by establishing several its discreteness conditions.
   \begin{theorem}
Let   $\mathrm{\bf H}_{X,\alpha}$ be the minimal operator associated  to~\eqref{I_01}  in $L^2(\mathbb{R}_+;\mathbb{C}^p)$
and  let $\lim\limits_{n\to\infty}d_n=0$.  Let also $\mathcal
A:=\diag\big\{\frac{\widetilde{\alpha}_1 }{r_1^{2}}, \frac{\widetilde{\alpha}_2
}{r_2^{2}},\ldots\big\}$ with $\widetilde{\alpha}_n$ given by \eqref{alpha-tilde}  and $\ker\mathcal A=\{0\}$.
  Assume also  that
at least one of the following conditions  is satisfied:
  \begin{equation}\label{self_adj_Sh1}
(i)\quad\limsup_{n\to\infty}\frac{r_n}{k_n}\Big\|\widetilde{\alpha}_n^{-1}\Big\|<\frac{1}{2},\quad k_n:=\min\{r_{n-1}d_n;r_{n+1}d_{n+1}\};\qquad\qquad\quad
\end{equation}
\begin{equation}\label{self_adj_Sh2}
(ii)\quad\limsup_{n\to\infty}r_n^s\Big(\frac{1}{r^s_{n-1}d_n^s}+\frac{1}{r_{n+1}^sd_{n+1}^s}\Big)
\Big\|\widetilde{\alpha}_n^{-1}\Big\|^s<\frac{1}{2^{s-1}}, \quad s\in[1;+\infty);\qquad
\end{equation}
\begin{equation}\label{like_Mirzoev}
(iii)\quad\limsup_{n\to\infty}\frac{1}{r_{n+1}^s}\left(\frac{r^s_{n}}{d_{n+1}^s}
\Big\|\widetilde{\alpha}_n^{-1}\Big\|^s+\frac{r^s_{n+2}}{d_{n+2}^s}
\Big\|\widetilde{\alpha}_{n+2}^{-1}\Big\|^s\right)<\frac{1}{2^{s-1}}, \quad s\in[1;+\infty).
\end{equation}
Then the operator $\mathrm{\bf H}_{X,\alpha}$ is selfadjoint and its spectrum is discrete.

In addition, if any of conditions  \eqref{self_adj_Sh1}, \eqref{self_adj_Sh2}, \eqref{like_Mirzoev} holds with the inequality sign replaced by the equality sign, then $\mathrm{\bf H}_{X,\alpha}$  is selfadjoint.
       \end{theorem}
\begin{proof}
We prove the result  assuming the  condition  $(i)$. The  conditions $(ii)$ and $(iii)$ are treated similarly.

Alongside the operator $\mathrm{\bf H}_{X,\alpha}$ consider the block Jacobi matrix
${\bf J}_{{X,\alpha}}^{(1)}(\mathrm{\bf{H}})$ given by \eqref{IV.2.1_05}. Now
$$
\mathcal A_{n}= \frac{1}{r_n^{2}}\,\widetilde{\alpha}_n \qquad\text{and}\qquad \mathcal B_{n}=-\frac{1}{r_nr_{n+1}\gd_{n+1}}\mathbb I_p.
$$
Since $\cA = \diag\big\{\frac{\widetilde{\alpha}_1 }{r_1^{2}}, \frac{\widetilde{\alpha}_2
}{r_2^{2}},\ldots\big\}$ is discrete,
there is $N\in\mathbb N$ such that the  matrices $\widetilde{\alpha}_n$ are invertible
for each  $n\geq N$. It easily follows from \eqref{self_adj_Sh1}  that
 \begin{equation}
\limsup_{n\to\infty}\|\mathcal A_{n}^{-1}\cdot\mathcal B_{n}\|=\limsup_{n\to\infty}r_n^2\cdot\|\widetilde{\alpha}_n^{-1}\|\cdot\frac{1}{r_nr_{n+1}d_{n+1}}=\limsup_{n\to\infty}\frac{r_n}{r_{n+1}d_{n+1}}\cdot\|\widetilde{\alpha}_n^{-1}\|<\frac{1}{2};
  \end{equation}
  \begin{equation}
\limsup_{n\to\infty}\|\mathcal A_{n}^{-1}\cdot\mathcal B_{n-1}^*\|=\limsup_{n\to\infty}r_n^2\cdot\|\widetilde{\alpha}_n^{-1}\|\cdot\frac{1}{r_{n-1}r_{n}d_{n}}=\limsup_{n\to\infty}\frac{r_n}{r_{n-1}d_{n}}\cdot\|\widetilde{\alpha}_n^{-1}\|<\frac{1}{2}.
  \end{equation}
Thus,  the matrix ${\bf J}_{{X,\alpha}}^{(1)}(\mathrm{\bf{H}})$ meets conditions
\eqref{adj1ep} of Theorem \ref{d_spec_J},  hence  it is selfadjoint and discrete.

  Similarly, conditions \eqref{self_adj_Sh2}, \eqref{like_Mirzoev}  imply conditions
  \eqref{self_adj1a}, \eqref{self_adj*a},  respectively, for  the matrix ${\bf J}_{{X,\alpha}}^{(1)}(\mathrm{\bf{H}})$.
  Thus, Theorem \ref{d_spec_J}  applies to the matrix  ${\bf J}_{{X,\alpha}}^{(1)}(\mathrm{\bf{H}})$ in all three cases
  \eqref{self_adj_Sh1}--\eqref{like_Mirzoev}  and ensures its   selfadjointness and  discreteness.

Finally,  according to the statements $(a)$ and $(b)$ (see \eqref{n(H)=n(J)})  mentioned above the operators
$\mathrm{\bf H}_{X,\alpha}$ and ${\bf J}_{{X,\alpha}}^{(1)}(\mathrm{\bf{H}})$ are
selfadjoint and discrete simultaneously.
\end{proof}

It is also proved in \cite{KM, KM10} (for $p=1$) and in \cite{KMN} (for $p>1$) that  certain  spectral properties of the Hamiltonian
$\mathrm{\bf H}_{X,\alpha}$
are closely related  to that
of the  minimal Jacobi operator ${\bf J}_{{X,\alpha}}^{(2)}(\mathrm{\bf{H}})$, associated in $l^2(\mathbb N_0;\mathbb C^p)$
to another  block Jacobi matrix of the form
   \begin{equation}\label{IV.2.1Sh''}
{\bf J}_{{X,\alpha}}^{(2)}(\mathrm{\bf{H}}) = \left(
\begin{array}{cccccc}
  \mathbb O_p  & \frac{1}{\gd_1^{2}}\mathbb I_p & \mathbb O_p  & \mathbb O_p & \mathbb O_p   &  \dots\\
   \frac{1}{\gd_1^{2}}\mathbb I_p  &  -\frac{1}{\gd_1^{2}}\mathbb I_p &
   \frac{1}{\gd_1^{3/2}\gd_2^{1/2}}\mathbb I_p & \mathbb O_p  & \mathbb O_p &  \dots\\
  \mathbb O_p  & \frac{1}{\gd_1^{3/2}\gd_2^{1/2}}\mathbb I_p  & \frac{\alpha_1}{\gd_2}  &
   \frac{1}{\gd_2^{2}}\mathbb I_p & \mathbb O_p &   \dots\\
  \mathbb O_p  & \mathbb O_p  & \frac{1}{\gd_2^{2}}\mathbb I_p &  -\frac{1}{\gd_2^{2}}\mathbb I_p &
  \frac{1}{\gd_2^{3/2}\gd_3^{1/2}}\mathbb I_p &  \dots\\
  \mathbb O_p  & \mathbb O_p  & \mathbb O_p  & \frac{1}{\gd_2^{3/2}\gd_3^{1/2}}\mathbb I_p &
   \frac{\alpha_2}{\gd_3} &   \dots\\
\dots& \dots&\dots&\dots&\dots&\dots\\
 \end{array}%
\right)\,.
   \end{equation}
   In particular, statements $(a), (b)$ (see \eqref{n(H)=n(J)}) remain  valid with ${\bf J}_{{X,\alpha}}^{(1)}(\mathrm{\bf{H}})$ replaced by ${\bf J}_{{X,\alpha}}^{(2)}(\mathrm{\bf{H}})$.

Other conditions for $\mathrm{\bf H}_{X,\alpha}$ to be discrete  can be  extracted from
Theorem~\ref{J_disc}.
   \begin{theorem}\label{Shr_disc_2}
Let   $\mathrm{\bf H}_{X,\alpha}$ be the minimal operator associated  to~\eqref{I_01}  in $L^2(\mathbb{R}_+;\mathbb{C}^p)$
and  let $\lim\limits_{n\to\infty}d_n=0$. Assume also that one of the following conditions hold:

$(i)$ let the diagonal part  ${\mathcal A}^{(1)}:=\diag\big\{\frac{\widetilde{\alpha}_1
}{r_1^{2}}, \frac{\widetilde{\alpha}_2 }{r_2^{2}},\ldots\big\}$ of the Jacobi matrix
${\bf J}_{{X,\alpha}}^{(1)}(\mathrm{\bf{H}})$
be  discrete, let $\ker {\mathcal A}^{(1)}=\{0\}$, and let
\begin{equation}\label{self_adj_Sh3}
\limsup_{n\to\infty}\frac{1}{d_{n+1}}\Big\||\widetilde{\alpha}_n|^{-1/2}\cdot|\widetilde{\alpha}_{n+1}|^{-1/2}\Big\|<\frac{1}{2};
\end{equation}

$(ii)$ let the  diagonal matrix $\mathcal A':=\diag\big\{
\frac{\alpha_1}{\gd_2},\frac{\alpha_2}{\gd_3},\ldots\big\}$ (a part of the diagonal  ${\mathcal
A}^{(2)}$  of the matrix  ${\bf J}_{{X,\alpha}}^{(2)}(\mathrm{\bf{H}})$ \eqref{IV.2.1Sh''})
be  discrete, let $\ker {\mathcal A}^{'}=\{0\}$, and let
   \begin{equation}\label{self_adj_Sh4}
\limsup_{n\to\infty}\frac{\big\||\alpha_n|^{-1/2}\big\|}{d^{1/2}_n}<\frac{1}{2},\qquad \limsup_{n\to\infty}\frac{\big\||\alpha_n|^{-1/2}\big\|}{d^{1/2}_{n+1}}<\frac{1}{2}.
 \end{equation}
Then:

$(1)$ $n_+(\mathrm{\bf H}_{X,\alpha}) = n_-(\mathrm{\bf H}_{X,\alpha})\le p$ and
the spectrum of  any selfadjoint extension of $\mathrm{\bf H}_{X,\alpha}$ is discrete.
In particular,  $\mathrm{\bf H}_{X,\alpha}$ is discrete  provided that  it is
selfadjoint.

$(2)$  If in addition,  $\{d_n\}_1^\infty\not\in l^2(\mathbb N)$,   then $\mathrm{\bf
H}_{X,\alpha}$ is selfadjoint and discrete.
   \end{theorem}
\begin{proof}
$(1.i)$ To apply Theorem~\ref{J_disc}  to  the Jacobi  matrix  ${\bf
J}_{{X,\alpha}}^{(1)}(\mathrm{\bf{H}})$ we note that its  diagonal part ${\mathcal
A}^{(1)}$ is discrete and  general condition  \eqref{cond_disc} turns into condition
\eqref{self_adj_Sh3}.  Thus, Theorem~\ref{J_disc}(1)  guarantees relations $n_+({\bf
J}_{{X,\alpha}}^{(1)}(\mathrm{\bf{H}})) = n_-({\bf
J}_{{X,\alpha}}^{(1)}(\mathrm{\bf{H}})) \le p$ and discreteness
of  any  selfadjoint extension of ${\bf J}_{{X,\alpha}}^{(1)}(\mathrm{\bf{H}})$.

$(1.ii)$ Next to apply Theorem~\ref{J_disc}  to  the second Jacobi  matrix  ${\bf
J}_{{X,\alpha}}^{(2)}(\mathrm{\bf{H}})$ we note first that combining condition
$\lim\limits_{n\to\infty}d_n=0$ with  discreteness property of  $\mathcal A'$ yields
discreteness of the  diagonal part ${\mathcal A}^{(2)}$ of  ${\bf
J}_{{X,\alpha}}^{(2)}(\mathrm{\bf{H}})$. Besides, now general condition
\eqref{cond_disc} turns into the conditions \eqref{self_adj_Sh4} for ${\bf
J}_{{X,\alpha}}^{(2)}(\mathrm{\bf{H}})$ and hence is satisfied.
Thus, Theorem~\ref{J_disc}(1) applies and
ensures  the relations $n_+({\bf J}_{{X,\alpha}}^{(2)}(\mathrm{\bf{H}})) = n_-({\bf
J}_{{X,\alpha}}^{(2)}(\mathrm{\bf{H}})) \le p$ as well as the discreteness property of each
selfadjoint extension of ${\bf J}_{{X,\alpha}}^{(2)}(\mathrm{\bf{H}})$.

Again statements $(a)$ and $(b)$ (see \eqref{n(H)=n(J)}) allow one to retranslate just proved properties of the
matrices ${\bf J}_{{X,\alpha}}^{(1)}(\mathrm{\bf{H}})$ and ${\bf
J}_{{X,\alpha}}^{(2)}(\mathrm{\bf{H}})$ into the respective properties of  the
Schrodinger  operator $\mathrm{\bf H}_{X,\alpha}$.

$(2)$ The proof is immediate by combining the  statement (1)  with  Proposition
\ref{self-adj_H_X_Carlem} (see \eqref{d_n2}) that ensures  selfadjointness of the Hamiltonian $\mathrm{\bf
H}_{X,\alpha}$.
\end{proof}

In Propositions \ref{Shr_disc_5} and \ref{D_disc_5} we identify  $l^\infty(\mathbb N)$  with its separable subspace  $c_0(\mathbb N)$.
\begin{proposition}\label{Shr_disc_5}
Let   $\mathrm{\bf H}_{X,\alpha}$ be the minimal operator associated  to expression \eqref{I_01}  in $L^2(\mathbb{R}_+;\mathbb{C}^p)$.
Let also  $\{d_n\}_{n=1}^\infty\in l^{2q}(\mathbb N)$ for $q\in(\frac12, \infty]$, let $\mathcal A':=\diag\big\{
\frac{\alpha_1}{\gd_2},\frac{\alpha_2}{\gd_3},\ldots\big\}$, and  let conditions \eqref{self_adj_Sh4} hold.
If, in addition,  $(\mathcal A')^{-1}\in\mathcal S_q(l^2(\mathbb N_0;\mathbb C^p))$, then:

$(i)$  $n_+(\mathrm{\bf H}_{X,\alpha}) = n_-(\mathrm{\bf H}_{X,\alpha})\le p$

$(ii)$ $(\widetilde{\mathrm{\bf H}}_{X,\alpha}-i\mathbb I)^{-1}\in\mathcal S_q(L^2(\mathbb R_+;\mathbb C^p))$
for any selfadjoint extension  $\widetilde{\mathrm{\bf H}}_{X,\alpha}$ of $\mathrm{\bf H}_{X,\alpha}$.

Moreover,  if  $\mathrm{\bf H}_{X,\alpha}= \mathrm{\bf H}_{X,\alpha}^*$, then
$({\mathrm{\bf H}}_{X,\alpha}-i\mathbb I)^{-1} \in\mathcal S_q(L^2(\mathbb R_+;\mathbb C^p))$.
   \end{proposition}
   \begin{proof}
For any $n\in \Bbb N$ we  set
    $$
\mathrm{\bf H}_{n,0} := -\frac{d^2}{dx^2}\upharpoonright\dom(\mathrm{\bf H}_{n,0}), \quad
\dom(\mathrm{\bf H}_{n,0})=\{f\in W^{2,2}[x_{n-1},x_n]:\, f(x_{n-1}+) =  f'(x_n-)=0\}.
    $$
Clearly, $\mathrm{\bf H}_{n,0}=\mathrm{\bf H}_{n,0}^*$. Let  $\Pi_{\mathrm H}=\{\mathcal H,
\Gamma_{0,\mathrm H},\Gamma_{1,\mathrm H}\}$ be the  boundary triplet for $\mathrm{\bf H}_X^*$  defined in
 Theorem \ref{th_bt_2S},
and  let $\mathrm{\bf H}_0:= \mathrm{\bf H}_X^*\upharpoonright\ker\Gamma_{0,\mathrm H}=\mathrm{\bf H}_0^*$ (see Theorem \ref{th_bt_2S} and eq. \eqref{Pi1}). It is easily seen that
$\mathrm{\bf H}_0 = \bigoplus\limits_{n\in\mathbb N}\mathrm{\bf H}_{n,0}$
and its   spectrum  is
  \begin{equation}\label{spectrum_of_H_0}
\sigma(\mathrm{\bf H}_0) =\bigcup\limits_{n\in \Bbb N}\sigma(\mathrm{\bf H}_{n,0})= \bigcup\limits_{n\in \Bbb N}\left\{\frac{\pi^2(2k+1)^2}{4d_n^2}\right\}_{k\in\mathbb N_0}.
  \end{equation}
Moreover,  $\sigma(\mathrm{\bf H}_0)$ is  of constant  multiplicity  $p$.  Since $\{d_n\}_{1}^\infty \in c_0(\mathbb N)$,
formula \eqref{spectrum_of_H_0}  implies  $0\notin \sigma(\mathrm{\bf H}_0)$.  It follows from \eqref{spectrum_of_H_0}
that  for any $q>\frac12$ the following equivalence holds
\begin{equation}\label{Sh-H0}
\{d_n\}_{n=1}^\infty\in l^{2q}(\mathbb N)\Longleftrightarrow  \mathrm{\bf H}_0^{-1} \in\mathcal S_q(L^2(\mathbb R_+;\mathbb C^p)).
\end{equation}
Further,  combining condition
$\{d_n^2\}_{n=1}^\infty\in l^{q}(\mathbb N)$ with   $(\mathcal A')^{-1}\in\mathcal S_q(l^2(\mathbb N_0;\mathbb C^p))$ yields $\big(\mathcal A^{(2)}\big)^{-1}\in\mathcal S_q(l^2(\mathbb N_0;\mathbb C^p))$.
In turn, combining this inclusion with  conditions \eqref{self_adj_Sh4} and applying Theorem~\ref{J_disc}
to the Jacobi operator ${\bf J}_{{X,\alpha}}^{(2)}(\mathrm{\bf{H}})$ ensures that
$n_+(\mathrm{\bf J}_{X,\alpha}^{(2)}(\mathrm{\bf{H}})) = n_-(\mathrm{\bf J}_{X,\alpha}^{(2)}(\mathrm{\bf{H}}))\le p$ and for any selfadjoint extension ${\widetilde{\bf J}}_{{X,\alpha}}^{(2)}(\mathrm{\bf{H}})$ of  ${\bf J}_{{X,\alpha}}^{(2)}(\mathrm{\bf{H}})$
 the inclusion holds $\big({\widetilde{\bf
J}}_{{X,\alpha}}^{(2)}(\mathrm{\bf{H}})-i\mathbb I\big)^{-1}\in\mathcal S_q(l^2(\mathbb N_0;\mathbb C^p))$.

On the other hand, in accordance with Proposition \ref{BoundOp} the minimal Jacobi operator
${\bf J}_{{X,\alpha}}^{(2)}(\mathrm{\bf{H}})$ given  by  \eqref{IV.2.1Sh''} is unitarily equivalent to
the boundary operator ${\bf B}_{{X,\alpha}}^{(2)}(\mathrm{\bf{H}})$  of the realization  ${\mathrm{\bf H}}_{X,\alpha}$ in the boundary triplet $\Pi_{\mathrm H}$. Hence,   $n_{\pm}({\mathrm{\bf H}}_{X,\alpha})=n_\pm({\bf B}_{{X,\alpha}}^{(2)}(\mathrm{\bf{H}}))=n_\pm({\bf J}_{{X,\alpha}}^{(2)}(\mathrm{\bf{H}}))$.
Therefore  there is a bijective  correspondence between selfadjoint extensions
$\widetilde{\mathrm{\bf H}}_{X,\alpha} = \widetilde{\mathrm{\bf H}}_{X,\alpha}^*$ of  $\mathrm{\bf H}_{X,\alpha}$
and selfadjoint extensions  ${\widetilde{\bf B}}_{{X,\alpha}}^{(2)}(\mathrm{\bf{H}})=\big({\widetilde{\bf B}}_{{X,\alpha}}^{(2)}(\mathrm{\bf{H}})\big)^*$ of
${\bf B}_{{X,\alpha}}^{(2)}(\mathrm{\bf{H}})$. The latter is given by
$$
\widetilde{\mathrm{\bf H}}_{X,\alpha} = {\mathrm{\bf H}}_{X}^*\upharpoonright\dom(\widetilde{\mathrm{\bf H}}_{X,\alpha}), \quad  \dom(\widetilde{\mathrm{\bf H}}_{X,\alpha})=\{f\in\dom (\mathrm{\bf H}_{X,\alpha}^*):
 \Gamma_{1,\mathrm H}f={\widetilde{\bf
B}}_{{X,\alpha}}^{(2)}(\mathrm{\bf{H}})\Gamma_{0,\mathrm H}f\},
$$
i.e.  ${\widetilde{\bf
B}}_{{X,\alpha}}^{(2)}(\mathrm{\bf{H}})$ is a boundary operator of the realization
$\widetilde{\mathrm{\bf H}}_{X,\alpha}$ in the boundary triplet $\Pi_{\mathrm H}$.

Since the boundary operator ${\widetilde{\bf B}}_{{X,\alpha}}^{(2)}(\mathrm{\bf{H}})$
is unitarily equivalent to the operator  ${\widetilde{\bf
J}}_{{X,\alpha}}^{(2)}(\mathrm{\bf{H}})$,  the inclusion  $\big({\widetilde{\bf
B}}_{{X,\alpha}}^{(2)}(\mathrm{\bf{H}})-i\mathbb I\big)^{-1}\in\mathcal S_q(l^2(\mathbb N_0;\mathbb C^p))$ holds alongside
the inclusion  $\big({\widetilde{\bf
J}}_{{X,\alpha}}^{(2)}(\mathrm{\bf{H}})-i\mathbb I\big)^{-1}\in\mathcal S_q(l^2(\mathbb N_0;\mathbb C^p))$.

Further,   in accordance with
Proposition \ref{prop_II.1.4_02} the following equivalence holds
   \begin{eqnarray}\label{rescompar1}
(\widetilde{\mathrm{\bf H}}_{X,\alpha} - i\mathbb I)^{-1} - (\mathrm{\bf H}_0 - i\mathbb I)^{-1}
\in \mathcal S_q(L^2(\mathbb R_+;\mathbb C^p))\Longleftrightarrow
({\widetilde{\bf
B}}_{{X,\alpha}}^{(2)}(\mathrm{\bf{H}}) - i\mathbb I)^{-1} \in \mathcal S_q(l^2(\mathbb N_0;\mathbb C^p)).
    \end{eqnarray}
Combining this equivalence with \eqref{Sh-H0} implies
$(\widetilde{\mathrm{\bf H}}_{X,\alpha} - i\mathbb I)^{-1} \in\mathcal S_q(L^2(\mathbb R_+;\mathbb C^p))$.

The  second  statement is immediate  from the first one.
   \end{proof}

\begin{remark}
$(i)$ Let $\{\mu_{n,j}\}_{j=1}^p = \sigma(\alpha_n^{-1})$ be the spectrum of $\alpha_n^{-1}$, $n\in\mathbb N$. Then condition $(\mathcal A')^{-1}\in\mathcal S_q(l^2(\mathbb N_0;\mathbb C^p))$ turns into the condition $\{|\mu_{n,j}|d_n\}_{n=1}^\infty\in l^q(\mathbb N)$
for each $j\in \{1,\ldots,p\}$.

$(ii)$ Using the Weyl asymptotic instead of explicit description of the spectrum \eqref{spectrum_of_H_0} one can extend  Proposition \ref{Shr_disc_5} to the case of  Schr\"{o}dinger operators
$\ell_{X,\gA}+Q$ with $Q\in L^1_{loc}(\mathbb R_+;\mathbb C^{p\times p})$.
\end{remark}

\begin{remark}
$(i)$ In \cite{BroMir18} selfadjointness (but not discreteness) under condition
\eqref{like_Mirzoev} (with $s=2$) was obtained by another method (see discussion in
Remark \ref{s-a-remark} (ii)).

$(ii)$ The condition $\{d_n\}_1^\infty\not\in l^2(\mathbb N)$ as a  test  for  the Hamiltonian
$\mathrm{\bf H}_{X,\alpha}$ to be selfadjoint  was first discovered in \cite{KM10,
KM} in the scalar case ($p=1$). This result was
extended by different methods in \cite{KMN, MirSaf16} to the matrix Hamiltonians $\mathrm{\bf H}_{X,\alpha}$
case ($p > 1$).

$(iii)$ In  the scalar case $(p=1)$  discreteness conditions \eqref{self_adj_Sh3} and \eqref{self_adj_Sh4} for the Hamiltonian $\mathrm{\bf H}_{X,\alpha}$
coincides with that  obtained  in  \cite{KM10, KM}.

$(iv)$  In  \cite{KM10, KM}  for $p=1$, and in \cite{KMN, MirSaf16} for $p> 1$  it is shown that conditions
\begin{equation}\label{max_ind}
\{d_n\}_1^\infty\in l^2(\mathbb N) \qquad \text{and}\qquad  \sum\limits_{n=1}^\infty d_{n+1}\|\widetilde{\alpha}_n\|<\infty,
\end{equation}
imply  the maximality of deficiency indices: $n_\pm(\mathrm{\bf H}_{X,\alpha}) =p$.
The proof of this  result was substantially relied on Berezansky's result
(see \cite[Theorem VII.1.1]{Ber68}, \cite[Ex. I.5]{Akh}) and its generalization from \cite{KosMir99}.
The second condition in \eqref{max_ind} demonstrates sharpness of conditions \eqref{self_adj_Sh1}~--~\eqref{like_Mirzoev}, \eqref{self_adj_Sh3}.

Note also that for  $p=1$ the existence  of  realizations
$\mathrm{\bf H}_{X,\alpha}$  with nontrivial indices
$n_\pm(\mathrm{\bf H}_{X,\alpha}) = 1$ was discovered  by Shubin Christ and G. Stolz  \cite{Chr_Sto_94}
in the special case $d_n=\frac{1}{n}$, $\alpha_n=-2n-1$ implying $\widetilde{\alpha}_n=0$.
\end{remark}

\subsection{Dirac operator}\label{7.2D}
Here  we present several conditions for $GS$-realizations $\mathrm{\bf D}_{X,\alpha}$  of Dirac operator (see \eqref{delta}) to be discrete.
\begin{theorem}\label{D_disc_inf}
Assume that $|\cI| = \infty$. Let $\lim\limits_{n\to\infty}d_n=0$ and let the spectrum
of the diagonal matrix  $\mathcal A':=\diag\big\{
\frac{\alpha_1}{\gd_2},\frac{\alpha_2}{\gd_3},\ldots\big\}$
be  discrete. Assume also  that the following condition holds
\begin{equation}\label{prop_chihara_1}
\limsup_{n\to\infty}\||\alpha_n|^{-1/2}\|<\frac{1}{2\sqrt c}.
  \end{equation}
Then the $GS$-realization $\mathrm{\bf D}_{X,\alpha}$ in $L^2(\cI, \C^{2p})$ is selfadjoint and its spectrum is discrete.
\end{theorem}
\begin{proof}
By the Carleman test, the Jacobi matrix ${\bf J}_{X,\alpha}$ given by \eqref{IV.2.1}  is
selfadjoint.
Since $\lim\limits_{n\to\infty}d_n=0$ and spectrum of the matrix $\mathcal A'$ is  discrete, then the diagonal part $\mathcal A$ of the matrix~\eqref{IV.2.1} is also discrete.

To apply  Theorem \ref{J_disc}  to the matrix  ${\bf J}_{X,\alpha} = {\bf J}_{X,\alpha}^*$  we note that
condition  \eqref{cond_disc}    turns into  condition \eqref{prop_chihara_1}. Thus, the selfadjoint
Jacobi matrix ${\bf J}_{X,\alpha}$  meets conditions of Theorem \ref{J_disc}, hence the
spectrum of ${\bf J}_{X,\alpha}  = {\bf J}_{X,\alpha}^*$   is  discrete. By Proposition \ref{prop_IV.2.1_01} (iii), the  operator $\mathrm{\bf D}_{X,\alpha}$
is selfadjoint and  discrete too.
  \end{proof}
    \begin{theorem}\label{D_disc_line}
Assume that $|\cI| < \infty$. Let $\lim\limits_{n\to\infty}d_n=0$ and let the spectrum of the diagonal part $\mathcal
A':=\diag\big\{ \frac{\alpha_1}{\gd_2},\frac{\alpha_2}{\gd_3},\ldots\big\}$
be  discrete. If condition \eqref{prop_chihara_1} holds, then $n_+(\mathrm{\bf D}_{X,\alpha}) = n_-(\mathrm{\bf D}_{X,\alpha})\le p$ and  the spectrum of  each selfadjoint extension of $\mathrm{\bf D}_{X,\alpha}$
is discrete. In particular, if  $\mathrm{\bf D}_{X,\alpha} = \mathrm{\bf D}_{X,\alpha}^*$, then its spectrum is  discrete.
\end{theorem}
\begin{proof}
The proof is similar to that of Theorem \ref{D_disc_inf} and relies
on Theorem \ref{J_disc}  and Proposition \ref{prop_IV.2.1_01}.
\end{proof}

\begin{proposition}\label{D_disc_5}
Let   $\mathrm{\bf D}_{X,\alpha}$ be the minimal operator associated  to~\eqref{1.2Intro}  in $L^2(\mathcal I;\mathbb{C}^{2p})$, $|\cI| < \infty$.
Let also  $\{d_n\}_{n=1}^\infty\in l^{q}(\mathbb N)$, $q\in(1;\infty]$, and let $\mathcal A':=\diag\big\{
\frac{\alpha_1}{\gd_2},\frac{\alpha_2}{\gd_3},\ldots\big\}$. Assume also that $(\mathcal A')^{-1}\in\mathcal S_q(l^2(\mathbb N_0;\mathbb C^p))$. If, in addition, condition \eqref{prop_chihara_1} holds, then the resolvent of any selfadjoint extension $\widetilde{\mathrm{\bf D}}_{X,\alpha}$ of $\mathrm{\bf D}_{X,\alpha}$ is
of class $\mathcal S_q(L^2(\mathcal I;\mathbb C^{2p}))$, i.e. $(\widetilde{\mathrm{\bf D}}_{X,\alpha}-i\mathbb I)^{-1}\in\mathcal S_q(L^2(\mathcal I;\mathbb C^{2p}))$. If,  moreover,  $\mathrm{\bf D}_{X,\alpha}= \mathrm{\bf D}_{X,\alpha}^*$, then  $(\mathrm{\bf D}_{X,\alpha}-i\mathbb I)^{-1}\in\mathcal S_q(L^2(\mathcal I;\mathbb C^{2p}))$.
   \end{proposition}
 \begin{proof}
  For any $n\in \Bbb N$   define the following selfadjoint extension of the minimal Dirac operator $\mathrm{\bf D}_n$:
%
$$
\mathrm{\bf D}_{n,0}:=\mathrm{\bf D}\upharpoonright\dom(\mathrm{\bf D}_{n,0}),\quad \dom(\mathrm{\bf D}_{n,0})=\{f\in W^{1,2}[x_{n-1},x_n]:\, f_I(x_{n-1}+)=0,\, f_{II}(x_n-)=0\}.
   $$
    Let  $\Pi=\{\mathcal H,
\Gamma_0,\Gamma_1\}$  be the  boundary triplet for $\mathrm{\bf D}_X^*=\bigoplus\limits_{n\in\mathbb N}\mathrm{\bf D}_n^*$ defined in
 Theorem \ref{th_bt_2}, and let
$\mathrm{\bf D}_0:= \mathrm{\bf D}_X^*\upharpoonright\ker\Gamma_0=\mathrm{\bf D}_0^*$ (see Theorem \ref{th_bt_2} and equation \eqref{IV.1.1_12}). Clearly,  $\mathrm{\bf D}_0 = \bigoplus\limits_{n\in\mathbb N}\mathrm{\bf D}_{n,0}$
and its   spectrum  is
\begin{equation}\label{spectrum-D}
\sigma(\mathrm{\bf D}_0)=\bigcup\limits_{n\in \Bbb N}\sigma(\mathrm{\bf D}_{n,0})= \bigcup\limits_{n\in \Bbb N}\left\{\pm\sqrt{\frac{
c^2\pi^2}{4d^2_n}\,(2k+1)^{2}+\left(\frac{c^{2}}{2}\right)^{2}}\right\}_{k\in\mathbb N_0}.
\end{equation}
Moreover,  $\sigma(\mathrm{\bf D}_0)$ is  of constant  multiplicity  $p$.  Since $\{d_n\}_{1}^\infty \in c_0(\mathbb N)$,
formula \eqref{spectrum-D}  implies  $0\notin \sigma(\mathrm{\bf D}_0)$.  It follows from \eqref{spectrum-D}
that  for any $q>1$ the following equivalence holds
\begin{equation}\label{D_D0}
\{d_n\}_{n=1}^\infty\in l^{q}(\mathbb N)\Longleftrightarrow\mathrm{\bf D}_0^{-1}\in\mathcal S_q(L^2(\mathcal I;\mathbb C^{2p})).
\end{equation}

Further,  combining condition
$\{d_n\}_{n=1}^\infty\in l^{q}(\mathbb N)$ with   $(\mathcal A')^{-1}\in\mathcal S_q(l^2(\mathbb N_0;\mathbb C^p))$ yields $\big(\mathcal A^{(2)}\big)^{-1}\in\mathcal S_q(l^2(\mathbb N_0;\mathbb C^p))$. In turn, combining this inclusion with   \eqref{prop_chihara_1} and applying Theorem~\ref{J_disc}  to the boundary operator ${\bf B}_{{X,\alpha}}$ \eqref{IV.2.1_01} in the boundary triplet $\Pi$  (see Proposition \ref{prop_IV.2.1_01}) ensures $n_-({\bf B}_{{X,\alpha}})=n_+({\bf B}_{{X,\alpha}})\leq p$ and for any selfadjoint extension ${\widetilde{\bf B}}_{{X,\alpha}}$ of ${\bf B}_{{X,\alpha}}$ the inclusion holds $\big({\widetilde{\bf
B}}_{{X,\alpha}}-i\mathbb I\big)^{-1}\in\mathcal S_q(l^2(\mathbb N_0;\mathbb C^p))$.

In accordance with Proposition \ref{prop_IV.2.1_01}, $n_{\pm}({\mathrm{\bf D}}_{X,\alpha})=n_\pm({\bf B}_{{X,\alpha}})$.
Therefore
there is a bijective  correspondence between selfadjoint extensions
$\widetilde{\mathrm{\bf D}}_{X,\alpha} = \widetilde{\mathrm{\bf D}}_{X,\alpha}^*$ of  $\mathrm{\bf D}_{X,\alpha}$
and selfadjoint extensions  ${\widetilde{\bf B}}_{{X,\alpha}}=\big({\widetilde{\bf B}}_{{X,\alpha}}\big)^*$ of
${\bf B}_{{X,\alpha}}$. The latter is given by
$$
\widetilde{\mathrm{\bf D}}_{X,\alpha} = {\mathrm{\bf D}}_{X}^*\upharpoonright\dom(\widetilde{\mathrm{\bf D}}_{X,\alpha}), \quad  \dom(\widetilde{\mathrm{\bf D}}_{X,\alpha})=\{f\in\dom (\mathrm{\bf D}_{X,\alpha}^*):
 \Gamma_1f={\widetilde{\bf
B}}_{{X,\alpha}}\Gamma_0f\},
$$
i.e.  ${\widetilde{\bf
B}}_{{X,\alpha}}$ is a boundary operator of the realization
$\widetilde{\mathrm{\bf D}}_{X,\alpha}$ in the boundary triplet $\Pi$.

On the other hand, in accordance with
Proposition \ref{prop_II.1.4_02} the following equivalence holds
\begin{eqnarray}\label{rescompar1}
(\widetilde{\mathrm{\bf D}}_{X,\alpha} - i\mathbb I)^{-1} - (\mathrm{\bf D}_0 - i\mathbb I)^{-1}
\in \mathcal S_q(L^2(\mathcal I;\mathbb C^{2p}))\Longleftrightarrow
({\widetilde{\bf B}}_{X, \alpha} - i\mathbb I)^{-1} \in \mathcal S_q(l^2(\mathbb N_0;\mathbb C^p)).
    \end{eqnarray}
Combining this equivalence with \eqref{D_D0} gives $(\widetilde{\mathrm{\bf D}}_{X,\alpha} - i\mathbb I)^{-1}\in\mathcal S_q(L^2(\mathcal I;\mathbb C^{2p}))$.

The  second  statement is immediate  from the first one.
   \end{proof}

Proposition \ref{def_ind_B} allows one  to construct a \emph{symmetric but not selfadjoint} Jacobi operator ${\bf J}_{X,\alpha}$
satisfying at the same time condition \eqref{eq:lim_a_n=infty}.
    \begin{proposition}\label{max_s-a}
Let  ${\bf J}_{X,\alpha}$ be the minimal  Jacobi operator associated with the matrix  \eqref{IV.2.1},  and let
\begin{equation}
d_n= \frac{C_1}{(1+r)^{2(n-1)}n^2} \qquad \text{and}\qquad \alpha_n= rc \mathbb I_p\qquad \text{with} \qquad  r> 4.
\end{equation}
Then $n_{\pm}({\bf J}_{X,\alpha}) =p$ and  each selfadjoint extension is discrete.
At the same time condition \eqref{prop_chihara_1} is satisfied.
\end{proposition}
\begin{proof}
It is easily seen that $\sum\limits_{n=1}^\infty d_n< \infty$. The sequences  $d_n$ and $\alpha_n$ meet conditions \eqref{5.26} of  Theorem \ref{VarIndices}. Indeed,
\begin{equation}
\sum_{n=2}^{\infty}d_{n}\prod_{k=1}^{n-1}\left(1+\frac 1c\,\|\alpha_k\|_{\mathbb
C^{p\times p}}\right)^{2}\leq C_1\sum_{n=2}^{\infty}\frac{1}{(1+r)^{2(n-1)}n^2}(1+r)^{2(n-1)}<+\infty.
\end{equation}
Thus, by Proposition \ref{def_ind_B},  $n_{\pm}({\bf J}_{X,\alpha}) =p$.

 At the same time, conditions \eqref{cond_disc} (for $p>1$) and \eqref{eq:lim_a_n=infty} (for $p=1$) turn into condition \eqref{prop_chihara_1} and are obviously satisfied.
\end{proof}
\begin{remark}
Proposition \ref{max_s-a} demonstrates that conditions \eqref{self_adjA} -- \eqref{cond_disc} of Theorem \ref{J_disc} do not ensure selfadjointness of Jacobi matrix ${\bf J}_{X,\alpha}$ even in the scalar case $(p=1)$.
\end{remark}

  \section{Jacobi matrices  with intermediate deficiency indices}\label{sec5}

Here we complete Theorem \ref{D_disc_line} by constructing selfadjoint Jacobi matrix ${\bf J}_{X,\widehat{\alpha}}$ of the form \eqref{IV.2.1} with discrete spectrum.

\begin{theorem}\label{abs_thA'}
Let
$\widehat{\gA}_n =\mathrm{diag}(\alpha_{n,1},\alpha_{n,2},\ldots,\alpha_{n,p}) = \widehat{\gA}_n^*$ be diagonal  matrices, $n\in\mathbb N$,
and let   ${\bf J}_{X,\widehat{\alpha}}$ be a minimal Jacobi operator associated with a matrix \eqref{IV.2.1} with entries $\widehat{\gA}_n$ instead of $\gA_n$. Let also
\begin{equation}
|\alpha_{n,1}|\leq|\alpha_{n,2}|\leq\ldots\leq|\alpha_{n,p}|,\qquad \text{and}\qquad\{d_n\}_{n=1}^\infty\in l^1(\mathbb{N}).
\end{equation}
 \begin{itemize}
\item[(i)] If following condition  holds
\begin{equation}\label{4.14}
\sum_{n\in \N} \sqrt{d_{n}d_{n+1}}\,|\alpha_{n,1}|=+\infty,
 \end{equation}
then  the operator
${\bf J}_{X,\widehat{\alpha}}$    is selfadjoint,
${\bf J}_{X,\widehat{\gA}} ={\bf J}_{X,\widehat{\gA}}^*$.
\item[(ii)] Assume in  addition that the following conditions  hold
\begin{equation}\label{disc}
\lim_{n\to\infty} \frac{|\alpha_{n,1}|}{d_{n+1}} = \infty,\qquad \lim_{n\to\infty} \frac{c}{\alpha_{n,1}}>-\frac{1}{4}.
\end{equation}
  Then  the  spectrum of Jacobi operator ${\bf J}_{X,\widehat{\alpha}} ={\bf J}_{X,\widehat{\gA}}^*$  is discrete.
\end{itemize}
  \end{theorem}
  \begin{proof}
(i) First we  prove selfadjointness of  the operator ${\bf J}_{X,\widehat{\gA}}$.
  Since the matrices  $\widehat{\gA}_n=\mathrm{diag}(\alpha_{n,1},\alpha_{n,2},\ldots,\alpha_{n,p})$ are diagonal for each  $n\in\N$,  the operator ${\bf J}_{X,\widehat{\gA}}$  splits into the direct sum of $p$ scalar Jacobi operators ${\bf J}_{X,\widetilde{\alpha}_j}$, where $\widetilde{\alpha}_j=\{\alpha_{n,j}\}_{n=1}^\infty\subset\mathbb R$, $j\in\{1,\ldots,p\}$.
By the Dennis-Wall test (see \cite{Akh}, Problem 2, p.25), ${\bf J}_{X,\widetilde{\alpha}_j}$ is self-adjoint whenever
  \begin{equation}\label{4.15}
\sum_{n=1}^\infty\frac{d_{n+1}^{3/2}}{\nu(d_{n+1})}\left(\frac{d_{n}^{3/2}|\alpha_{n,j}|}{\nu(d_n)} + d_{n+2}^{1/2}\right)=+\infty\,,\qquad j\in\{1,\ldots,p\}.
  \end{equation}
Since  $\nu(d_{n})=\frac{cd_{n}}{\sqrt{1+c^2d_{n}^2}}\thicksim c\,d_n$ as $n\to\infty$, one gets  $\frac{d_{n}^{3/2}}{\nu(d_n)} \thicksim c^{-1}\,d_n^{1/2}$ as $n\to\infty$.
Taking these relations into account and noting that
$$
2\sum_{n\in \N} \sqrt{d_{n+1}}\,\sqrt{d_{n+2}}\le
 \sum_{n\in \N} (d_{n+1} + d_{n+2}) < +\infty,
$$
one concludes that the series \eqref{4.14} and \eqref{4.15} diverge only simultaneously.
Thus series \eqref{4.15}  diverges  and the operator ${\bf J}_{X,\widehat{\gA}}$ is selfadjoint, i.e.
$n_{\pm}({{\bf J}}_{X,\widehat{\gA}}) = 0$.

(ii) Since the operator ${\bf J}_{X,\widehat{\gA}}$ is selfadjoint, then according to \cite[Theorem 8]{Chi62} (see Remark \ref{chihara}) conditions \eqref{disc} guarantee the discreteness of the spectrum of  ${\bf J}_{X,\widehat{\gA}}$.
  \end{proof}

\begin{corollary}\label{selfadj}
Let $\{d_n\}_{n=1}^\infty\in l^1(\mathbb{N})$ and let ${\bf J}_{X,\alpha}$ be the Jacobi matrix of form~\eqref{IV.2.1}. Assume that $\gA:=\{\alpha_n\}_1^\infty \subset\C^{p\times
p}$, where  $\gA_n = \widehat{\gA}_n + \widetilde{\gA}_n =\gA_n^*$ and matrices
 $\widehat{\gA}_n=\mathrm{diag}(\alpha_{n,1},\alpha_{n,2},\ldots,\alpha_{n,p}) = \widehat{\gA}_n^*$
 are  diagonal.  If the condition \eqref{4.14} holds and
  \begin{equation}\label{alpha_bound}
\|\widetilde{\gA}_n\|_{\mathbb C^{p\times p}} = \mathcal{O}(d_{n+1})\qquad \text{for}\
n\to \infty,
  \end{equation}
 then  the  (minimal) Jacobi operator ${\bf J}_{X,\alpha}$ is selfadjoint,
${\bf J}_{X,\gA} = {\bf J}_{X,\gA}^*$.
\end{corollary}

Next we construct Jacobi matrices ${\bf J}_{X,\alpha}$ with intermediate deficiency indices  $0 < n_\pm({\bf J}_{X,\alpha})=p_1 <p$. Let $p = p_1+p_2$, where $p_1, \ p_2\in \N$.
To this end consider block--matrix representation of  $\gA_n\in \C^{p\times p}$
with respect to the orthogonal decomposition $\C^{p} =  \C^{p_1}\oplus \C^{p_2}$:
\begin{equation}\label{eq:block1}
\gA_n =
\begin{pmatrix} \gA_n^{11} & \gA_n^{12} \\ \gA_n^{21} & \gA_n^{22} \end{pmatrix}= \alpha_n^* \in\C^{p\times p},\quad  \alpha_n^{ij}\in\C^{p_i\times p_j}, \quad i,j\in\{1,2\}.
  \end{equation}

\begin{corollary}\label{B_ind_p1}
Let $\{d_n\}_{n=1}^\infty\in l^1(\mathbb{N})$ and let ${\bf J}_{X,\alpha}$ be the Jacobi matrix of the form \eqref{IV.2.1}
with $\gA =\{\alpha_n\}_1^\infty$  and $\alpha_n$ admitting representations  \eqref{eq:block1}.
Assume in  addition that:
  \begin{itemize}
\item[(i)] the matrices $\gA_n^{11}$, $n\in\N$ satisfy condition \eqref{5.26} with $p$ replaced by $p_1$;
\item[(ii)] $\|\gA_n^{12}\|_{\mathbb C^{p_1\times p_2}} = \mathcal{O}(d_{n+1})$, as $n\to \infty$;
\item[(iii)]
the matrices $\gA_n^{22} = \widehat{\gA}_n^{22} + \widetilde{\gA}_n^{22}$, where $\widehat{\gA}_n^{22}=\mathrm{diag}(\alpha_{n,1},\alpha_{n,2},\ldots,\alpha_{n,p_2})$ satisfy condition \eqref{4.14} and $\widetilde{\gA}_n^{22}$ satisfy condition \eqref{alpha_bound} with $p$ replaced by $p_2$.
\end{itemize}
Then $n_{\pm}({\bf J}_{X,\gA})=p_1$.
\end{corollary}
\begin{proof}
 It follows from conditions $(ii)$ and $(iii)$ that
${\bf J}_{X,\gA}$ is a bounded perturbation of the matrix ${\bf J}_{X,\widehat{\gA}}$ with ${\widehat{\alpha}}_n=\{{\widehat{\alpha}}_n\}_1^\infty$, where
${\widehat{\alpha}}_n=\gA_n^{11}\oplus \widehat{\gA}_n^{22},$ and $\gA^{11}=\{\gA_n^{11}\}_{1}^\infty$, $\widehat{\gA}^{22}=\{\widehat{\gA}_n^{22}\}_{1}^\infty$. Hence $n_\pm({\bf J}_{X,\gA})=n_\pm({\bf J}_{X,\widehat{\gA}})$.

Note that the decomposition
$\mathbb C^p=\mathbb C^{p_1}\oplus \mathbb C^{p_2}$ yields representation of the operator ${\bf J}_{X,\widehat{\gA}}$ as ${\bf J}_{X,\widehat{\gA}}={\bf J}_{X,\gA^{11}}\oplus{\bf J}_{X,\widehat{\gA}^{22}}$.
Therefore,
$$
n_\pm({\bf J}_{X,\gA})=n_\pm({\bf J}_{X,\widehat{\gA}})=n_\pm({\bf J}_{X,\gA^{11}})+n_\pm({\bf J}_{X,\widehat{\gA}^{22}}) =
p_1 + 0 = p_1.
$$
We have used here that in accordance with Corollary  \ref{selfadj},  $n_\pm({\bf J}_{X,\widehat{\gA}^{22}})=0$.
\end{proof}

We give a special case of the Corollary \ref{B_ind_p1} with diagonal block--matrix sequence $\gA_n:=\{\alpha_n\}_1^\infty$
\begin{equation}\label{eq:block2}
\gA_n=\alpha_n^* =
\begin{pmatrix} \gA_n^{11} & \mathbb O_{p_1\times p_2} \\ \mathbb O_{p_2\times p_1} & \gA_n^{22} \end{pmatrix}\in\C^{p\times p},\quad  \alpha_n^{ij}\in\C^{p_i\times p_j}, \quad i,j\in\{1,2\}. \end{equation}

\begin{corollary}\label{B_ind_p1cor}
Let $\{d_n\}_{n=1}^\infty\in l^1(\mathbb{N})$ and let ${\bf J}_{X,\alpha}$ be the Jacobi matrix of the form \eqref{IV.2.1}
with $\gA =\{\alpha_n\}_1^\infty$  and $\alpha_n$ admitting representations  \eqref{eq:block2}.
Assume in  addition that:
\begin{itemize}
\item[(i)] the matrices $\gA_n^{11}$, $n\in\N$ satisfy condition \eqref{5.26} with $p$ replaced by $p_1$;
\item[(ii)]
the diagonal  matrices $\gA_n^{22} = \mathrm{diag}(\alpha_{n,1},\alpha_{n,2},\ldots,\alpha_{n,p_2})$ satisfy condition \eqref{4.14} with $p$ replaced by $p_2$.
\end{itemize}
Then $n_{\pm}({\bf J}_{X,\gA})=p_1$.
\end{corollary}
\begin{proof}
Proof follows from Corollary \ref{B_ind_p1}.
\end{proof}

\begin{corollary}\label{abs_thA''}
Let $\widehat{\gA}_n$  be diagonal  matrices, $\widehat{\gA}_n =\mathrm{diag}(\alpha_{n,1},\alpha_{n,2},\ldots,\alpha_{n,p}) = \widehat{\gA}_n^*$, $n\in\mathbb N$,
and let   $\widehat{{\bf J}}_{X,\widehat{\alpha}}$ be the  Jacobi
matrix given by  \eqref{IV.2.1_01''} with entries $\widehat{\gA}_n$ instead of $\gA_n$. Let also $|\alpha_{n,1}|\leq|\alpha_{n,2}|\leq\ldots\leq|\alpha_{n,p}|$.
Assume also that $$|\alpha_{n,1}|\leq|\alpha_{n,2}|\leq\ldots\leq|\alpha_{n,p}|\qquad\text{and}\qquad\{d_n\}_{n=1}^\infty\in l^1(\mathbb{N}).$$
\begin{itemize}
\item[(i)] If conditions \eqref{abs1A}--\eqref{abs3A}, \eqref{4.14} are satisfied
then  the  minimal Jacobi operator $\widehat{{\bf J}}_{X,\widehat{\alpha}}$   is selfadjoint,
$\widehat{{\bf J}}_{X,\widehat{\gA}} = \widehat{{\bf J}}_{X,\widehat{\gA}}^*$.
\item[(ii)] If an addition conditions \eqref{disc} hold
 then  the  spectrum of Jacobi operator $\widehat{{\bf J}}_{X,\widehat{\alpha}}$   is discrete.
\end{itemize}
  \end{corollary}
\begin{proof}
(i) Considering  the operator $\widehat{{\bf J}}_{X,\widehat{\gA}}$ as a perturbation
 of ${{\bf J}}_{X,\widehat{\gA}}$  and following reasoning of the proof of
 Theorem \ref{abs_thA} we conclude that conditions \eqref{abs1A}--\eqref{abs3A}
imply  conditions \eqref{abs1}, \eqref{abs2}, \eqref{abs3'}  of Corollary \ref{abs_cor}. Therefore combining Theorem \ref{abs_thA'} (i) with the conclusion of Corollary \ref{abs_cor}  yields
$n_{\pm}(\widehat{{\bf J}}_{X,\widehat{\gA}})= n_{\pm}({{\bf J}}_{X,\widehat{\gA}}) =0$, i.e. $\widehat{{\bf J}}_{X,\widehat{\gA}} = \widehat{{\bf J}}_{X,\widehat{\gA}}^*$.

(ii) It follows from Theorems \ref{abs_th} (b) and \ref{abs_thA'} (ii), that  the Jacobi operator $\widehat{{\bf J}}_{X,\widehat{\alpha}}$  has a discrete spectrum.
  \end{proof}

  \begin{corollary}
Let  $\mathrm{\bf D}_{X,\widehat{\alpha}}$   be the minimal operator generated in $L^2(\mathcal I;\mathbb{C}^p)$ by expression \eqref{1.2Intro} on the domain~\eqref{delta}   with $\widehat{\gA}_n =\mathrm{diag}(\alpha_{n,1},\alpha_{n,2},\ldots,\alpha_{n,p}) = \widehat{\gA}_n^*$, $|\alpha_{n,1}|\leq|\alpha_{n,2}|\leq\ldots\leq|\alpha_{n,p}|$. Let also  $|\cI| < \infty$ and $\lim\limits_{n\to\infty}d_n=0$. If conditions \eqref{4.14} and \eqref{disc} hold, then the operator $\mathrm{\bf D}_{X,\widehat{\alpha}}$ is selfadjoint and  discrete.
\end{corollary}
\begin{proof}
It follows from Theorem \ref{abs_thA'} and Proposition \ref{prop_IV.2.1_01} (iii).
\end{proof}

For a Jacobi matrix ${\bf J}_{X,\beta}$ of the form \eqref{IV.3.1_04'} $\gB=\{\beta_n\}_1^\infty$ similar results on selfadjointness and deficiency indices are valid. We present only the results on selfadjointness.

\begin{theorem}\label{abs_thB'}
Let $\{d_n\}_{n=1}^\infty\in l^1(\mathbb{N})$ and let ${\bf J}_{X,\widehat{\beta}}$ be the  Jacobi
matrix given by \eqref{IV.3.1_04'} with entries $\widehat{\gB}=\{\widehat{\beta}_n\}_1^\infty$ instead of $\gB$. Here  $\widehat{\gB}_n =\mathrm{diag}(\beta_{n,1},\beta_{n,2},\ldots,\beta_{n,p}) = \widehat{\gB}_n^*$ are diagonal  matrices,
and $|\beta_{n,1}|\leq|\beta_{n,2}|\leq\ldots\leq|\beta_{n,p}|$, $n\in\mathbb N$.
 If
\begin{equation}\label{4.14B}
\sum_{n\in \N} \sqrt{d_{n}d_{n+1}}\,|\beta_{n,1}|=+\infty,
 \end{equation}
then  the  minimal Jacobi operator ${\bf J}_{X,\widehat{\beta}}$ is selfadjoint,
${\bf J}_{X,\widehat{\gB}} ={\bf J}_{X,\widehat{\gB}}^*$.
  \end{theorem}
\begin{corollary}\label{selfadjB}
Let $\{d_n\}_{n=1}^\infty\in l^1(\mathbb{N})$,  $\gB_n = \widehat{\gB}_n + \widetilde{\gB}_n=\gB_n^*$ and for $\widehat{\gB}=\{\widehat{\beta}_n\}_1^\infty$ all conditions of Theorem \ref{abs_thB'} are satisfied. If in addition
\begin{equation}\label{beta_bound}
\|\widetilde{\gB}_n\|_{\mathbb C^{p\times p}} = \mathcal{O}(d_{n})\quad \text{for}\  n\to \infty,
\end{equation}
then the  Jacobi operator ${\bf J}_{X,\beta}$
is selfadjoint, i.e. ${\bf J}_{X,\gB} = {\bf J}_{X,\gB}^*$.
\end{corollary}

\section{Comparison with known results}
\subsection{Comparison with the results of Kostyuchenko -- Mirzoev and Berezanskii}\label{sec6}
Consider  infinite block Jacobi matrix
\begin{equation}\label{Jacobi_m}
{\bf J}=\left(
    \begin{array}{cccccc}
      \mathcal A_0 & \mathcal B_0 & \mathbb O_p & \mathbb O_p & \mathbb O_p & \ldots \\
      \mathcal B_0^* & \mathcal A_1 & \mathcal B_1 & \mathbb O_p & \mathbb O_p & \ldots \\
      \mathbb O_p & \mathcal B_1^* & \mathcal A_2 & \mathcal B_2 & \mathbb O_p & \ldots \\
      \vdots & \vdots & \vdots & \vdots & \vdots & \ddots
          \end{array}
  \right),
\end{equation}
where $\mathcal A_n,\ \mathcal B_n \in \mathbb C^{p\times p}$, $\mathcal A_n=\mathcal A_n^*$  and  $\det \mathcal B_n
\not =0$, $n\in \Bbb N_0$.

As we have already mentioned in the Introduction other
conditions for Jacobi matrices to have  maximal deficiency  indices
were obtained also in \cite{KosMir01}.
Here we compare Theorems \ref{abs_thA}, \ref{abs_thB} with results from \cite{KosMir01}.

Following \cite{KosMir01}, we introduce the matrices $\mathcal C_n$ assuming
$\mathcal C_0:={\mathcal B_1^*}^{-1}$, $\mathcal C_1:=\mathbb I_p$, and
 \begin{equation}\label{C_nnn}
\mathcal C_{n}=\left\{\begin{array}{l}
      (-1)^j\mathcal B_{2j-1}^{-1}\mathcal B_{2j}^*\ldots \mathcal B_2^*\mathcal B_1^{-1},\quad \text{if}\ \ n=2j, \\
      (-1)^j\mathcal B_{2j}^{-1}\mathcal B_{2j-1}^*\ldots \mathcal B_2^{-1}\mathcal B_1^*,\quad \text{if}\ \ n=2j+1,\ \ j\in\mathbb N.
    \end{array}\right.
\end{equation}

\begin{theorem}[\cite{KosMir01}]\label{KosMirTh}
Let the sequence $\{\mathcal C_n\}_0^\infty$ be given by the equalities \eqref{C_nnn}. If the conditions
\begin{equation}\label{C_n}
\sum\limits_{n=1}^{+\infty}\|\mathcal C_n\|^2<+\infty,
\end{equation}
\begin{equation}\label{C_nA_n}
\sum\limits_{n=1}^{+\infty}\|\mathcal C_n^*\mathcal A_n\mathcal C_n\|<+\infty
\end{equation}
hold, then $\bf J$ of the form \eqref{Jacobi_m} is in completely indeterminate case, i.e. $n_\pm({\bf J})=p$.
\end{theorem}

This result was applied to the  Schr\"{o}dinger operator  with $\delta$--interactions
 on the semiaxis $(\sum\limits_{n\in\mathbb N} d_n =\infty)$
in \cite{KM} (scalar case,  $p=1$) and in  \cite{KMN, MirSaf16} (matrix case, $p>1$).

First  we show that matrices ${\bf J}_{X,\alpha}$ never  satisfy
conditions of  Theorem \ref{KosMirTh} for any $\alpha$.

\begin{proposition}\label{KosMirInd}
Let  ${\bf J}_{X,\alpha}$ be the Jacobi matrix of form \eqref{IV.2.1} and let
$\mathcal C_n$ be the matrices of the form \eqref{C_nnn} constructed from the entries of  ${\bf J}_{X,\alpha}$.
Then the series \eqref{C_nA_n} diverges, hence ${\bf J}_{X,\alpha}$ does not meet
conditions of  Theorem \ref{KosMirTh}.
   \end{proposition}
\begin{proof}
In our case the entries $\mathcal B_n$ and $\mathcal A_{n}$  are
\begin{equation}\label{BA}
\mathcal B_n = \left\{\begin{array}{ll}
             \frac{\nu(\gd_{j+1})}{d_{j+1}^2}\mathbb I_p, & n=2j, \\
             \frac{\nu(\gd_{j+1})}{\gd_{j+1}^{3/2}\gd_{j+2}^{1/2}}\mathbb I_p, & n=2j+1,
           \end{array}\right.\qquad \mathcal A_{n}=\left\{\begin{array}{ll}
             \frac{\alpha_j}{\gd_{j+1}}, & n=2j, \\
             -\frac{\nu(\gd_{j+1})}{d_{j+1}^2}\mathbb I_p, & n=2j+1.
           \end{array}\right.
\end{equation}
In accordance with \eqref{C_nnn} and \eqref{BA} we derive for  $n=2j+1$
 \begin{eqnarray}\label{1212}
\mathcal C_{2j+1}&=&(-1)^j\mathcal B_{2j}^{-1}\mathcal B_{2j-1}\ldots \mathcal B_4^{-1}\mathcal B_3\mathcal B_2^{-1}\mathcal B_1\nonumber\\&=&(-1)^{j}\frac{d_{j+1}^2}{\nu(\gd_{j+1})}\cdot \frac{\nu(\gd_{j})}{\gd_{j}^{3/2}\gd_{j+1}^{1/2}}\cdot\frac{d_{j}^2}{\nu(\gd_{j})}\cdot \frac{\nu(\gd_{j-1})}{\gd_{j-1}^{3/2}\gd_{j}^{1/2}}\cdot\ldots\cdot\\&\cdot&
 \frac{d_{3}^2}{\nu(\gd_{3})}\cdot \frac{\nu(\gd_{2})}{\gd_{2}^{3/2}\gd_{3}^{1/2}}\cdot\frac{d_{2}^2}{\nu(\gd_{2})}\cdot \frac{\nu(\gd_{1})}{\gd_{1}^{3/2}\gd_{2}^{1/2}}\mathbb I_p
=(-1)^{j}\frac{\nu(\gd_{1})d_{j+1}^{3/2}}{\nu(\gd_{j+1})d_1^{3/2}}\mathbb I_p.\nonumber
                     \end{eqnarray}

Let us check that \eqref{C_nA_n} is violated for the matrix ${\bf J}_{X,\alpha}$.
Combining \eqref{BA} with \eqref{1212}  for $n=2j+1$, $j\in\mathbb N$, we obtain
\begin{eqnarray}\label{9.7}
\sum\limits_{j=1}^{+\infty}\|\mathcal C_{2j+1}^*\mathcal A_{2j+1}\mathcal C_{2j+1}\|&=&\sum\limits_{j=1}^{+\infty}\|\mathcal C_{2j+1}\|^2\cdot\|\mathcal A_{2j+1}\| =\sum\limits_{j=1}^{+\infty} \frac{\nu^2(\gd_{1})d_{j+1}^{3}}{\nu^2(\gd_{j+1})d_1^{3}}\cdot\frac{\nu(\gd_{j+1})}{d_{j+1}^2}\nonumber\\
&=&\frac{\nu^2(\gd_{1})}{d_1^3}
\sum\limits_{j=1}^{+\infty} \frac{d_{j+1}}{\nu(\gd_{j+1})}=\frac{\nu^2(\gd_{1})}{d_1^3}
\sum\limits_{j=1}^{+\infty} \frac{d_{j+1} \sqrt{1+c^2d_{j+1}^2}}{c\gd_{j+1}}\nonumber\\
&=&\frac{\nu^2(\gd_{1})}{cd_1^3}
\sum\limits_{j=1}^{+\infty}\sqrt{1+c^2d_{j+1}^2}=+\infty.\quad
\end{eqnarray}
Thus the series \eqref{C_nA_n} diverges.

It follows from \eqref{C_nnn} and \eqref{BA} that for  $n=2j$
 \begin{eqnarray}\label{C_2j}
\mathcal C_{2j}&=&(-1)^j\mathcal B_{2j-1}^{-1}\mathcal B_{2j-2}\ldots \mathcal B_4\mathcal B_3^{-1}\mathcal B_2\mathcal B_1^{-1}\nonumber\\&=&(-1)^{j}\frac{d_{j}^{3/2}d_{j+1}^{1/2}}{\nu(\gd_{j})}\cdot \frac{\nu(\gd_{j})}{\gd_{j}^2}\cdot\frac{d_{j-1}^{3/2}d_j^{1/2}}{\nu(\gd_{j-1})}\cdot \frac{\nu(\gd_{j-1})}{\gd_{j-1}^2}\cdot\ldots\cdot \nonumber\\&\cdot&
\frac{\nu(\gd_{3})}{d_{3}^2}\cdot \frac{\gd_{2}^{3/2}\gd_{3}^{1/2}}{\nu(\gd_{2})}\cdot\frac{\nu(\gd_{2})}{d_{2}^2}\cdot \frac{\gd_{1}^{3/2}\gd_{2}^{1/2}}{\nu(\gd_{1})}\mathbb I_p
=(-1)^{j}\frac{d_{j+1}^{1/2}d_1^{3/2}}{\nu(\gd_{1})}\mathbb I_p.
                     \end{eqnarray}

 Combining \eqref{BA} with \eqref{C_2j}  for $n=2j$, $j\in\mathbb N$, we obtain
$$
\sum\limits_{j=1}^{+\infty}\|\mathcal C_{2j}^*\mathcal A_{2j}\mathcal C_{2j}\|=\sum\limits_{j=1}^{+\infty} \frac{d_{j+1}d_1^3}{\nu^2(\gd_{1})}\cdot \frac{\|\alpha_j\|_{\mathbb C^{p\times p}}}{\gd_{j+1}}=\frac{d_1^3}{\nu^2(\gd_{1})}
\sum\limits_{j=1}^{+\infty} \|\alpha_j\|_{\mathbb C^{p\times p}}.
$$
Thus  the series \eqref{C_nA_n} for $n=2j$ converges if and only if $\{\alpha_n\}_1^\infty\in l^1(\mathbb N;\mathbb C^{p\times p})$.

Combining this statement with condition \eqref{9.7}, we conclude that series \eqref{C_nA_n} always diverges.
\end{proof}

\begin{remark}
Clearly Proposition \ref{KosMirInd} is of interest only in the case $\{d_n\}_{n=1}^\infty\in l^1(\mathbb{N})$.
Indeed, if $\sum\limits_{n=1}^{+\infty}d_n = \infty$, then the Carleman test \eqref{Car} implies
$n_{\pm}({\bf J}_{X,\alpha}) =0$.
\end{remark}

Next we describe the area of applicability of Theorem \ref{KosMirTh} to matrices
 ${\bf J}_{X,\beta}$.
  \begin{proposition}\label{KosMirIndB}
Let  $\{d_n\}_{n=1}^\infty\in l^1(\mathbb{N})$. Let  ${\bf J}_{X,\beta}$ be the Jacobi matrix of form \eqref{IV.3.1_04'} and let
$\mathcal C_n$ be the matrices of the form \eqref{C_nnn} constructed from the entries of  ${\bf J}_{X,\beta}$.
Then the series \eqref{C_nA_n} converges if and only if $\{\beta_n\}_1^\infty\in l^1(\mathbb N;\mathbb C^{p\times p})$.
  \end{proposition}
\begin{proof}
In our case the entries $\mathcal B_n$ and $\mathcal A_{n}$  are
\begin{equation}\label{BAbeta}
\mathcal B_n = \!\left\{\!\begin{array}{ll}
             \frac{\nu(\gd_{j+1})}{d_{j+1}^2}\mathbb I_p,\! &\!\! n=2j, \\
             \frac{\nu(\gd_{j+1})}{\gd_{j+1}^{3/2}\gd_{j+2}^{1/2}}\mathbb I_p,\! &\!\! n=2j+1,
           \end{array}\right.\
            \mathcal A_{n}=\!\left\{\!\begin{array}{ll}
             \mathbb O_p,\!  &\!\! n=2j, \\
             -\frac{\nu^{2}(\gd_{j+1})}{d_{j+1}^{3}}(\beta_{j+1} + \gd_{j+1}\mathbb I_p),\!  &\!\! n=2j+1.
           \end{array}\right.
\end{equation}

For the case $n=2j$ entries $\mathcal A_n=\mathbb O_p$, hence the convergence of series \eqref{C_nA_n} is obvious.

In accordance with \eqref{C_nnn}, \eqref{1212} and \eqref{BAbeta} we derive for
$n=2j+1$

 \begin{eqnarray}\label{1212beta}
&\mathcal C_n=(-1)^j\frac{\nu(\gd_{1})d_{j+1}^{3/2}}{\nu(\gd_{j+1})d_1^{3/2}}\mathbb I_p.
\end{eqnarray}

Combining \eqref{BAbeta} with \eqref{1212beta}  for $n=2j+1$, $j\in\mathbb N$, we obtain
\begin{eqnarray}\label{9.11}
&\sum\limits_{j=1}^{+\infty}\|\mathcal C_{2j+1}^*\mathcal A_{2j+1}\mathcal
C_{2j+1}\|=\sum\limits_{j=1}^{+\infty}\|\mathcal C_{2j+1}\|^2\cdot\|\mathcal A_{2j+1}\|
 =\sum\limits_{j=1}^{+\infty}
\frac{\nu^2(\gd_{1})d_{j+1}^{3}}{\nu^2(\gd_{j+1})d_1^{3}}\cdot\frac{\nu^{2}(\gd_{j+1})}{d_{j+1}^{3}}\|\beta_{j+1}+d_{j+1}\mathbb
I_p\|_{\mathbb C^{p\times
p}}\nonumber\\
&=\frac{\nu^2(\gd_{1})}{d_1^{3}}\sum\limits_{j=1}^{+\infty}\|\beta_{j+1}+d_{j+1}\mathbb
I_p\|_{\mathbb C^{p\times p}}.
\end{eqnarray}
It suffices to consider the case $\|\beta_{j}\|_{\mathbb C^{p\times p}}>d_{j}$ for all $j\in \mathbb N$. Indeed, for those indices $j_k$, for which $\|\beta_{j_k}\|_{\mathbb C^{p\times p}}<d_{j_k}$, he corresponding part of the series \eqref{9.11} converges. We have,
\begin{equation}\label{9.12}
\frac{\nu^2(\gd_{1})}{d_1^{3}}\sum\limits_{j=1}^{+\infty} (\|\beta_{j+1}\|_{\mathbb
C^{p\times p}}-d_{j+1})\leq\sum\limits_{j=1}^{+\infty}\|\mathcal C_{2j+1}^*\mathcal
A_{2j+1}\mathcal C_{2j+1}\|
\leq\frac{\nu^2(\gd_{1})}{d_1^{3}}\sum\limits_{j=1}^{+\infty}
(\|\beta_{j+1}\|_{\mathbb C^{p\times p}}+d_{j+1}).
\end{equation}

Since $\{d_n\}_{n=1}^\infty\in l^1(\mathbb{N})$, then due to bilateral assessment \eqref{9.12}, series \eqref{C_nA_n} converges if and only if $\{\beta_n\}_1^\infty\in l^1(\mathbb N;\mathbb C^{p\times p})$.
  \end{proof}

   \begin{remark}\label{Rem_Kos_Mir_B}
 Proposition \ref{KosMirIndB}  shows  that conditions of Theorem \ref{KosMirTh}
in comparison with  the conditions of Theorem \ref{VarIndicesBeta} are too restrictive
to be applied to Jacobi matrices ${\bf J}_{X,\beta}$ satisfying condition \eqref{5.26B}. The conditions of Proposition \ref{KosMirIndB} coincide with that of Corollary \ref{corB}.
However, the conditions of Corollary \ref{corB1}, hence the conditions of Theorem \ref{VarIndicesBeta}
are substantially  weaker than the conditions of Proposition \ref{KosMirIndB}. To demonstrate
this fact let us  consider the following  simple examples:

$(i)$ Let  $d_n=2^{-n^2}$ and $\|\beta_{n}\|_{\mathbb C^{p\times
p}}=1$. Clearly,  conditions of Corollary \ref{corB1} (i) are  satisfied
and  $n_\pm({\bf J}_{X,\beta})=p$.
At the same time,   $\{\beta_n\}_1^\infty \notin l^1(\mathbb N;\mathbb C^{p\times p})$, and
conditions of Corollary \ref{corB} (Proposition \ref{KosMirIndB}) are  violated.

$(ii)$  Let  $d_n=2^{-n}$ and $\|\beta_{n}\|_{\mathbb C^{p\times
p}} = 2^{-1}c(\sqrt{2}-1)$.  Then the matrix ${\bf J}_{X,\beta}$ meets  conditions of
Corollary \ref{corB1} (ii) with  $d=2$, hence  $n_\pm({\bf J}_{X,\beta})=p$.
At the same time,   $\{\beta_n\}_1^\infty \notin l^1(\mathbb N;\mathbb C^{p\times p})$, and
conditions of Corollary \ref{corB} (Proposition \ref{KosMirIndB}) are  violated.

Thus,   Corollary \ref{corB1} ensures equalities $n_\pm({\bf J}_{X,\beta})=p$ for both matrices
${\bf J}_{X,\beta}$ described above while these matrices do not meet conditions of Theorem \ref{KosMirTh}.
However even if Theorem   \ref{KosMirTh} ensures equalities  $n_\pm({{\bf J}}_{X,\gB})=p$, it
does  not ensure relations $n_\pm(\widehat{{\bf J}}_{X,\gB})= p$, even if conditions
of Theorem \ref{abs_th} hold, i.e.  conditions of Theorem   \ref{KosMirTh}
are not invariant under the mapping ${\bf J}_{X,\gB} \to  \widehat{{\bf J}}_{X,\gB}$.
\end{remark}



The following result demonstrating the sharpness of the Carleman condition \eqref{Car}  was obtained by Berezansky
\cite[Theorem VII.1.5]{Ber68} in the scalar case $(p= 1)$
  and was extended to the matrix case  $(p > 1)$ by Kostychenko and Mirzoev  \cite{KosMir99}.
\begin{theorem}[\cite{KosMir99}]\label{KosMir-max3}
Let  ${\bf J}$  be the block Jacobi matrix \eqref{Jacobi_m}.
Assume that  $\|\mathcal A_n\|\leq C$ and the following  inequalities hold
\begin{equation}\label{Ber-cond}
\|\mathcal B_{n-1}\|\cdot\|\mathcal B_{n+1}\|\leq\|\mathcal B_n^{-1}\|^{-2}.
\end{equation}
Then Jacobi matrix ${\bf J}$ has maximal deficiency indices, $n_\pm({\bf J})=p$ provided that Carleman condition \eqref{Car}
is violated, i.e.
%
%
$\sum_{n=1}^\infty\|\mathcal B_n\|^{-1}<+\infty.$
%
\end{theorem}
%

Next we compare our results  with Theorem \ref{KosMir-max3}.

\begin{proposition}
Let  ${\bf J}_{X,\alpha}$ and  ${\bf J}_{X,\beta}$ be block Jacobi matrices
satisfying conditions of Propositions \ref{def_ind_B} and \ref{JBabs}, respectively.
Assume also that $\beta_n=-d_n$, i.e.  ${\bf J}_{X,\beta}$ has zero diagonal. Then  $n_\pm({\bf J}_{X,\alpha})=n_\pm({\bf J}_{X,\beta})=p$ while
the  matrices ${\bf J}_{X,\alpha}$ and  ${\bf J}_{X,\beta}$
never meet conditions of Theorem \ref{KosMir-max3}.
     \end{proposition}
  \begin{proof}
$(i)$
Jacobi matrix ${\bf J}_{X,\alpha}$ \eqref{IV.2.1}  never meets conditions of Theorem \ref{KosMir-max3}
because the sequence of its  diagonal entries  $\mathcal A_n$ (see \eqref{BA}) is unbounded.
%
%

$(ii)$ Let $\beta_n=-d_n$, i.e. the diagonal of  ${\bf J}_{X,\beta}$ \eqref{IV.3.1_04'} is zero.
It follows from \eqref{5.26B} that  $\{d_n\}_{n=1}^\infty\in l^1(\mathbb{N})$.
Berezanskii's condition \eqref{Ber-cond} for   ${\bf J}_{X,\beta}$   with $c>0$ and   $n=2j+1$  takes the form
$$
 \frac{\nu(d_{j+2})}{d_{j+2}}\leq\frac{\nu(d_{j+1})}{d_{j+1}}\Rightarrow\frac{1}{\sqrt{1+c^2d_{j+2}^2}}\leq\frac{1}{\sqrt{1+c^2d_{j+1}^2}},
 \qquad j\in\mathbb N_0,
$$
and means that the sequence  $\{d_n\}_{n\in\mathbb N}$ monotonically increases.
However, the latter contradicts the condition   $\{d_n\}_{n\in\mathbb N}\in l^1(\N)$.
  \end{proof}

\subsection{Comparison with Dyukarev's results }\label{sec7}
Here we show that one of the main results of Dyukarev \cite[Theorem 2]{Dyuk10} can easily
be extracted from our Proposition \ref{J'beta_p}. To this end we compare his Jacobi matrix ${\bf J}_{\mathrm{Dyuk}}$
with the matrix ${\bf J}_{X,\gB}'$ of form \eqref{Beta''} for a special choice of matrix  entries.

\begin{theorem}[\cite{Dyuk10}, Theorem 2]\label{Duk}
Let integers $p\geq1$ and $p_1\geq0$ satisfy the condition $0\leq p_1\leq p$ and
the diagonal entries of the matrices $\widetilde{\mathcal B}_n$ and $\mathcal R_n$ be defined by the formulae
\begin{equation}\label{Duk1}
\widetilde{\mathcal B}_n=\mathrm{diag}\left(\underset{p_1}{\underbrace{\frac{1}{n+1},\ldots,\frac{1}{n+1}}},\underset{p-p_1}{\underbrace{1,\ldots,1}}\right),
\qquad n\geq0,
\end{equation}
\begin{equation}\label{Duk2}
\mathcal R_0=\mathbb I_p,\quad\mathcal R_n=\sqrt{\mathbb I_p+\widetilde{\mathcal B}_{n-1}^2},\quad n\geq1.
\end{equation}
Further, suppose that the blocks $\mathcal A_n$ and $\mathcal B_n$ of the Jacobi matrix ${\bf J}_{\mathrm{Dyuk}}$ \eqref{Jacobi_m_intro} have the form
\begin{equation}\label{Duk3}
\mathcal A_n=\mathbb O_p,\quad n\geq0,\quad \mathcal B_0=\widetilde{\mathcal B}_0,\quad \mathcal B_n=\widetilde{\mathcal B}_{n-1}^{-1}
\mathcal R_n\widetilde{\mathcal B}_n^{-1},\quad n\geq1.
\end{equation}
Then $n_\pm({\bf J}_{\mathrm{Dyuk}})=p_1$.
\end{theorem}
\begin{proof}
Below we deduce Theorem \ref{Duk} from our Proposition \ref{J'beta_p}.

It is easily seen from \eqref{Duk1}--\eqref{Duk3} that  entries  of ${\bf J}_{\mathrm{Dyuk}}$ allow representation
\begin{equation*}
\mathcal A_{n,\mathrm{Dyuk}}=\mathbb O_p,\quad n\geq0,\quad \mathcal B_{0,\mathrm{Dyuk}}=\mathbb I_p,
\end{equation*}
\begin{equation}\label{Duk_inter}
 \mathcal B_{n,\mathrm{Dyuk}}=\mathcal B_{n,\mathrm{Dyuk}}'\oplus\mathcal B_{n,\mathrm{Dyuk}}'',
\qquad n\geq1,
\end{equation}
where
\begin{equation}\label{Duk_B'}
 \mathcal B_{n,\mathrm{Dyuk}}'=\mathrm{diag}\left((n+1)\sqrt{n^2 +1},\ldots,(n+1)\sqrt{n^2 +1}\right)\in\mathbb C^{p_1\times p_1},
\qquad n\geq1,
\end{equation}
\begin{equation}\label{Duk_B''}
 \mathcal B_{n,\mathrm{Dyuk}}''=\mathrm{diag}\left(\sqrt2,\ldots,\sqrt2\right)\in\mathbb C^{(p-p_1)\times (p-p_1)},
\qquad n\geq1.
\end{equation}
 Let ${\bf J}_{\mathrm{Dyuk}}'$ be the Jacobi matrix with off-diagonal entries of form \eqref{Duk_B'} and let ${\bf J}_{\mathrm{Dyuk}}''$ be the Jacobi matrix with off-diagonal entries of form \eqref{Duk_B''}. Thus ${\bf J}_{\mathrm{Dyuk}}={\bf J}_{\mathrm{Dyuk}}'\oplus{\bf J}_{\mathrm{Dyuk}}''$.

Alongside with the Jacobi  matrix ${\bf J}_{\mathrm{Dyuk}}$ we consider the Jacobi matrix  ${\bf J}_{X,\gB}'\subset\mathbb C^{p_1\times p_1}$ of form \eqref{Beta''} with $\beta_n=-d_n\mathbb I_p$ and $\gd_{n} = \frac{c}{(n+1)\sqrt{n^2 +1}}$, $n\geq1$.
Since $\{\beta_n\}_1^{\infty}\in l^1(\Bbb N;\mathbb C^{p_1\times p_1})$,
Corollary \ref{corB} ensures that condition \eqref{5.26B} is satisfied, hence  Proposition \ref{JBabs}
implies maximality indices for the Jacobi matrix ${\bf J}_{X,\gB}$  given by \eqref{IV.3.1_04'}, i.e.
$n_\pm({\bf J}_{X,\gB})=p_1$.

On the other hand,  $\lim\limits_{n\to\infty}\frac{d_{n-1}^2}{d_{n}}=0$, and the matrix ${\bf J}_{X,\gB}$
meets conditions of  Proposition \ref{J'beta_p}. Therefore combining this proposition with the previous
equalities yields
\begin{equation}\label{J'_ind}
n_\pm({\bf J}_{X,\gB}')=  n_\pm({\bf J}_{X,\gB}) = p_1.
\end{equation}
These equalities allow us to compare Dukarev's matrix ${\bf J}_{\mathrm{Dyuk}}$ with the Jacobi
matrix ${\bf J}_{X,\gB}'$ which is simpler than the original  matrix ${\bf J}_{X,\gB}$.
With the above choice of $\beta_n$ and  $\gd_{n}$
the entries of the matrix ${\bf J}_{X,\gB}'$ are given  by
\begin{equation*}
{\mathcal A}_n=\mathbb O_{p_1},\quad n\geq0,
\end{equation*}
\begin{equation}\label{B'}
{\mathcal B}_n=\left\{\begin{array}{ll}
(j+1)\sqrt{j^2+1}\,\mathbb I_{p_1},& n=2j, \\
\sqrt{(j+1)(j+2)}\sqrt[4]{(j^2+1)((j+1)^2+1)}\,\mathbb I_{p_1},&n=2j+1,
                                             \end{array}\right.\quad j\geq0.
\end{equation}
To prove the relations   $n_\pm({\bf J}_{X,\gB}')=n_\pm({\bf J}_{\mathrm{Dyuk}}')=p$
let us check the conditions of Corollary \ref{abs_cor} for the pair $\{{\bf J}_{\mathrm{Dyuk}}',{\bf J}_{X,\gB}'\}$ treating the matrix ${\bf J}_{\mathrm{Dyuk}}'$ of form \eqref{Duk3} as an unperturbed  Jacobi matrix.

For $n=2j$ condition \eqref{abs1} is obvious, because ${\mathcal B}_n=\mathcal B_{n,\mathrm{Dyuk}}'$.

For $n=2j+1$ direct calculation shows that
\begin{equation}
\lim\limits_{j\to\infty}\left\|({\mathcal B}_{2j+1}-\mathcal B_{2j+1,\mathrm{Dyuk}}')\mathcal (B_{2j+1,\mathrm{Dyuk}}')^{-1}\right\|_{\mathbb C^{p_1\times p_1}}
=\frac{3}{4}.
\end{equation}
Therefore  condition \eqref{abs1} of Corollary \ref{abs_cor} is fulfilled with
 $N$ big enough, i.e.
 \begin{equation}
\underset{j\geq N}{\sup}\left\|({\mathcal B}_{2j+1}-\mathcal B_{2j+1,\mathrm{Dyuk}}')\mathcal (B_{2j+1,\mathrm{Dyuk}}')^{-1}\right\|_{\mathbb C^{p_1\times p_1}}<1.
 \end{equation}
Further,  for even $n=2j$ direct calculation yields
   \begin{equation}
\lim\limits_{j\to\infty}\left\|{\mathcal B}_{2j}-{\mathcal B}_{2j-1}(\mathcal B_{2j-1,\mathrm{Dyuk}}')^{-1}\mathcal B_{2j,\mathrm{Dyuk}}'\right\|_{\mathbb C^{p_1\times p_1}}=1,
\end{equation}
and  the pair $\{{\bf J}_{\mathrm{Dyuk}}',{\bf J}_{X,\gB}'\}$  meets the condition \eqref{abs2}
of Corollary \ref{abs_cor},  i.e.
  \begin{equation}
\underset{j\geq 0}{\sup}\left\|{\mathcal B}_{2j}-{\mathcal B}_{2j-1}(\mathcal B_{2j-1,\mathrm{Dyuk}}')^{-1}\mathcal B_{2j,\mathrm{Dyuk}}'\right\|_{\mathbb C^{p_1\times p_1}}<\infty.
\end{equation}
Similarly, for odd $n=2j+1$ we  derive
  \begin{equation}
\lim\limits_{j\to\infty}\left\|{\mathcal B}_{2j+1}-{\mathcal B}_{2j}(\mathcal B_{2j,\mathrm{Dyuk}}')^{-1}\mathcal B_{2j+1,\mathrm{Dyuk}}'\right\|_{\mathbb C^{p_1\times p_1}}=\frac{3}{2}\,,
 \end{equation}
 and condition \eqref{abs2} for $n=2j+1$ is verified too.

Condition \eqref{abs3'} is obviously satisfied  because $\mathcal A_n=\mathcal A_{n,\mathrm{Dyuk}}=\mathbb O_{p_1}$.
Thus, the pair $\{{\bf J}_{\mathrm{Dyuk}}',{\bf J}_{X,\gB}'\}$ meets all conditions of
Corollary \ref{abs_cor}, and hence $n_\pm({\bf J}_{\mathrm{Dyuk}}')=n_\pm({\bf J}_{X,\gB}')$. Therefore  combining the latter  with  equality \eqref{J'_ind} yields  $n_\pm({\bf J}_{\mathrm{Dyuk}}')=n_\pm({\bf J}_{X,\gB}')=n_\pm({\bf J}_{X,\gB})=p_1$.

Finally  we apply Carleman test \eqref{Car} to conclude  that $n_\pm({\bf J}_{\mathrm{Dyuk}}'')=0$
and
$$
n_\pm({\bf J}_{\mathrm{Dyuk}}) = n_\pm({\bf J}_{\mathrm{Dyuk}}'\oplus{\bf J}_{\mathrm{Dyuk}}'') =  n_\pm({\bf J}_{\mathrm{Dyuk}}') +  n_\pm({\bf J}_{\mathrm{Dyuk}}'') = p_1 + 0 =p_1.
$$
This proves the result.
\end{proof}

\appendix
\renewcommand{\thesection}{\Alph{section}}
\section{Appendix}\label{app}

\textbf{1. Boundary triplets.} Let $A$ be a densely defined closed symmetric
operator in a separable Hilbert space $\gH$ with equal deficiency
indices $\mathrm{n}_\pm(A)=\dim \cN_{\pm \I} \leq \infty,$ where
$\cN_z:=\ker(A^*-z)$ is the defect subspace.

\begin{definition}[\cite{DM91, DM95, DerMal17, Gor84}]\label{def_ordinary_bt}
A triplet $\Pi=\{\cH,\gG_0,\gG_1\}$ is called a {\rm (ordinary)
boundary triplet} for the adjoint operator $A^*$ if $\cH$ is an auxiliary
Hilbert space and $\Gamma_0,\Gamma_1:\  \dom(A^*)\rightarrow \cH$
are linear mappings such that the second abstract Green identity
\begin{equation}\label{II.1.2_green_f}
(A^*f,g)_\gH - (f,A^*g)_\gH = (\gG_1f,\gG_0g)_\cH -
(\gG_0f,\gG_1g)_\cH, \qquad f,g\in\dom(A^*),
\end{equation}
holds
and the mapping $\gG:=\begin{pmatrix}\Gamma_0\\\Gamma_1\end{pmatrix}:  \dom(A^*)
\rightarrow \cH \oplus \cH$ is surjective.
\end{definition}

First note that a boundary triplet 
for $A^*$ exists since the deficiency indices of $A$ are assumed to be
equal. Noreover, $\mathrm{n}_\pm(A) = \dim(\cH)$ and $A=A^*\upharpoonright\left(\ker(\Gamma_0) \cap \ker(\Gamma_1)\right)$ hold. Note also that a boundary triplet for $A^*$ is not unique.

A closed extension $\widetilde{A}$ of $A$ is called \emph{proper}
if $A\subseteq\widetilde{A}\subseteq A^*$. The set of all proper extensions of $A$ is denoted by $\Ext A.$

\begin{proposition}[\cite{DM91} -- \cite{DerMal17}]\label{Bound-op-A}
Let $A$ be a symmetric operator in $\frak H$, and let $\Pi=\{\mathcal H,\Gamma_0,\Gamma_1\}$ be a boundary triplet of the operator $A^*$ and $B\in\mathcal C(\mathcal H)$. Then:

$(i)$  $\widetilde{A}$ and $A_0:=A^*\upharpoonright\ker\Gamma_0$ are \emph{disjoint} (i.e. $\dom\widetilde{A}\cap\dom A_0=\dom A
$), if and only if the extension $\widetilde{A}$ is parameterized in the following way
\begin{equation}\label{bound_op}
\widetilde{A}=A^*\upharpoonright\dom\widetilde{A},\qquad \dom \widetilde{A}=\{f\in\dom A^*: \Gamma_1f=B\Gamma_0f\},\quad B\in\mathcal C(\mathcal H);
\end{equation}

$(ii)$ operator $A_B$ is symmetric (selfadjoint) if and only if so is $B$, and $n_\pm(A_B)=n_\pm(B)$;

$(iii)$ $A_{B_1}\subseteq A_{B_2}\Longleftrightarrow B_1\subseteq B_2$;

$(iv)$  if $A_0$ is discrete, then the operator $A_B=A_B^*$ is discrete if and only if $B=B^*$ is discrete.
\end{proposition}

In this case,  $\widetilde{A}:=A_B$ and the operator $B$ is called {\it the boundary operator} of the extension $A_{B}$.

\begin{proposition}[\cite{DM91}]\label{prop_II.1.4_02}
Let $\Pi=\{\cH,\gG_0,\gG_1\}$  be a boundary triplet for $A^*,$
$B_1,B_2\in \cC(\cH)$ and  let  ${\mathcal S}_q,\
q\in(0,\infty],$ be the Neumann-Schatten ideal. Then the following
equivalence holds
\begin{equation}\label{II.1.4_02}
(A_{B_1}-i)^{-1} - (A_{B_2}-i)^{-1}\in{\mathcal
S}_q(\gH)
\quad \Longleftrightarrow \quad (B_1 - i)^{-1}-
(B_2 - i )^{-1}\in{\mathcal S}_q(\cH).
\end{equation}
\noindent In particular, $(A_{{B}_1} - i)^{-1} - (A_0 -i)^{-1} \in {\mathcal S}_q(\gH) \Longleftrightarrow
\bigl(B_1 - i\bigr)^{-1} \in {\mathcal S}_q(\cH).$
\end{proposition}

\textbf{2. Dirac operator.} Let ${\mathrm D}_{n}$ be the minimal  operator generated in
$L^{2}([x_{n-1},x_{n}];\C^{2p})$  by the differential
expression (\ref{1.2Intro})
  \begin{equation}\label{3.6A}
{\mathrm {\bf D}}_{n}=\mathrm {\bf D} \upharpoonright\dom({\mathrm {\bf D}}_n), \quad\dom({\mathrm {\bf D}}_{n})=W^{1,2}_{0}([x_{n-1},x_{n}];\C^{2p}).
\end{equation}
Its adjoint ${\mathrm {\bf D}}^*_n$ is given by
${\mathrm {\bf D}}_{n}^*=\mathrm {\bf D}\upharpoonright\dom({\mathrm {\bf D}}_n^*)$, $\dom({\mathrm {\bf D}}_{n})=W^{1,2}([x_{n-1},x_{n}];\C^{2p}).
$

We define the minimal operator ${\mathrm {\bf D}}_X$ on $L^2(\mathcal I;\mathbb C^{2p})$ by
  \begin{equation}
{\mathrm {\bf D}}_X:=\bigoplus\limits_{n\in\mathbb N} {\mathrm {\bf D}}_n,  \quad \dom({\mathrm {\bf D}}_X)=W_0^{1,2}(\mathcal I\setminus X;\mathbb C^{2p})=\bigoplus\limits_{n\in\mathbb N}W^{1,2}_{0}([x_{n-1},x_{n}];\C^{2p}).
\end{equation}
where ${\mathrm {\bf D}}_n$, $n\in\mathbb N$, is given by \eqref{3.6A}.
Its adjoint ${\mathrm {\bf D}}^*_X$ is given by
  \begin{equation}\label{DX*}
{\mathrm {\bf D}}_X^*:=\bigoplus\limits_{n\in\mathbb N} {\mathrm {\bf D}}_n^*, \qquad\dom({\mathrm {\bf D}}_X^*)=\bigoplus\limits_{n\in\mathbb N}W^{1,2}([x_{n-1},x_{n}];\C^{2p}).
\end{equation}

The boundary triplet $\widetilde{\Pi}^{(n)}=\big\{\mathbb C^{2p},
\widetilde{\Gamma}_{0}^{(n)},\widetilde{\Gamma}_{1}^{(n)}\big\}$ for the Dirac operator on  $[x_{n-1},x_{n}]$ is constructed elementarily:
\begin{equation}\label{triple2}
\widetilde{\Gamma}_{0}^{(n)}f:=
\left(\begin{array}{c}
                 f_I(x_{n-1}+)\\
                 i\,c\, f_{II}(x_{n}-)
                       \end{array}\right),\qquad
\widetilde{\Gamma}_{1}^{(n)}f:= \left(\begin{array}{c}
                                                                           i\,c\, f_{II}(x_{n-1}+)\\
                                                                            f_I(x_{n}-)
                                                                          \end{array}\right)\,.
\end{equation}
However, the direct sum of boundary triplets, generally, is not a boundary triplet (see examples in \cite{KM}, \cite{MN2012}).

In papers \cite{CarMalPos13} (scalar case $p=1$) and \cite{BudMalPos17, BudMalPos18} (matrix case $p>1$) the boundary triplets for the operator $\mathrm {\bf D}_X^*$ were constructed  using the regularization technique developed in \cite{KM}, \cite{CarMalPos13} and \cite{MN2012}.

\begin{theorem}[\cite{BudMalPos17, CarMalPos13, BudMalPos18}]\label{th_bt_2} Let $X=\{x_{n}\}_{n=0}^\infty\subset\mathcal I=(0,b)$  and  $d^*(X)<+\infty$. Define the mappings
$$\Gamma_j^{(n)}: W^{1,2}([x_{n-1},x_n];\C^{2p})\to\C^{2p}\,, \quad n\in \N\,,\quad j\in \{0,1\}\,,
$$
by setting
\begin{equation}\label{IV.1.1_12}
\begin{gathered}
\Gamma_0^{(n)}f:=\left(\begin{array}{c}
                 \gd_n^{1/2}  f_I(x_{n-1}+)\\
                 i\,c\,\gd_n^{3/2}\sqrt{1+\frac{1}{c^{2}\gd_{n}^{2}}}\,  f_{II}(x_{n}-)
                       \end{array}\right),\\
                       \Gamma_1^{(n)}f:=\left(\begin{array}{c}
                                                                           i\,c\,\gd_n^{-1/2}\,( f_{II}(x_{n-1}+)- f_{II}(x_{n}-))\\
                                                                           \gd_n^{-3/2}\left(1+\frac{1}{c^{2}\,\gd_{n}^{2}}\right)^{-1/2}
                                                                           (f_I(x_{n}-)-f_I(x_{n-1}+)-i\,c\,\gd_n\, f_{II}(x_{n}-))
                                                                 \end{array}\right).
\end{gathered}
\end{equation}
Then:
\begin{itemize}
\item[(i)] for any $n\in \N$,  $\Pi^{(n)}=\{\C^{2p},
\Gamma_0^{(n)},\Gamma_1^{(n)}\}$ is a boundary triplet for
${\mathrm {\bf D}}_n^*$;

\item[(ii)] the direct sum $\Pi:=\bigoplus_{n=1}^\infty\Pi^{(n)} =
\{\cH, \Gamma_0,\Gamma_1\}$ with $\cH = l^{2}(\N;\C^{2p})$ and
${\Gamma}_{j} =\bigoplus_{n=1}^\infty {\Gamma}^{(n)}_{j},$ $j\in\{0,1\},$  is a boundary triplet for the operator ${\mathrm {\bf D}}_X^* = \bigoplus_{n=1}^\infty {\mathrm {\bf D}}_{n}^{*}$.
\end{itemize}
       \end{theorem}

The next sentence explains the appearance of Jacobi matrices ${\bf B}_{X,\gA}$ of the form
 \begin{equation}\label{IV.2.1_01}
{\bf B}_{X,\gA}=\left(
\begin{array}{cccccc}
  \mathbb O_p  & -\frac{\nu(\gd_{1})}{\gd_1^{2}}\mathbb I_p & \mathbb O_p  & \mathbb O_p & \mathbb O_p   &  \dots\\
   -\frac{\nu(\gd_{1})}{\gd_1^{2}}\mathbb I_p  &  -\frac{\nu(\gd_{1})}{\gd_1^{2}}\mathbb I_p &
   \frac{\nu(\gd_{1})}{\gd_1^{3/2}\gd_2^{1/2}}\mathbb I_p & \mathbb O_p  & \mathbb O_p &  \dots\\
  \mathbb O_p  & \frac{\nu(\gd_{1})}{\gd_1^{3/2}\gd_2^{1/2}}\mathbb I_p  & \frac{\alpha_1}{\gd_2}  &
   -\frac{\nu(\gd_{2})}{\gd_2^{2}}\mathbb I_p & \mathbb O_p &   \dots\\
  \mathbb O_p  & \mathbb O_p  & -\frac{\nu(\gd_{2})}{\gd_2^{2}}\mathbb I_p &  -\frac{\nu(\gd_{2})}{\gd_2^{2}}\mathbb I_p &
  \frac{\nu(\gd_{2})}{\gd_2^{3/2}\gd_3^{1/2}}\mathbb I_p &  \dots\\
  \mathbb O_p  & \mathbb O_p  & \mathbb O_p  & \frac{\nu(\gd_{2})}{\gd_2^{3/2}\gd_3^{1/2}}\mathbb I_p &
   \frac{\alpha_2}{\gd_3} &   \dots\\
\dots& \dots&\dots&\dots&\dots&\dots\\
 \end{array}%
\right)
   \end{equation}
     in the context of the Dirac operators.

 \begin{proposition}[\cite{BudMalPos17, CarMalPos13, BudMalPos18}]\label{prop_IV.2.1_01}
Let
$\Pi=\{\cH,\Gamma_0,\Gamma_1\}$ be the boundary triplet for the operator ${\mathrm {\bf D}}_{X}^*$ of the form \eqref{IV.1.1_12}. Also let ${\bf B}_{X,\alpha}$ be
the minimal Jacobi operator associated with a Jacobi matrix of the form~\eqref{IV.2.1_01}. Then ${\bf B}_{X,\alpha}$ is a boundary operator for $\mathrm {\bf D}_{X,\alpha}$, i.e.
 \begin{equation}\label{B-op-D}
 \begin{array}{c}
   \mathrm {\bf D}_{X,\alpha}=\mathrm {\bf D}_{\mathrm{\bf B}_{X,\alpha}}=\mathrm {\bf D}_{X}^*\upharpoonright\dom(\mathrm {\bf D}_{\mathrm{\bf B}_{X,\alpha}}), \\
   \dom(\mathrm {\bf D}_{\mathrm{\bf B}_{X,\alpha}})=\{f\in W^{1,2}(\cI\setminus X;\C^{2p}):
\Gamma_1f=\mathrm{\bf B}_{X,\alpha}\Gamma_0f\}.
 \end{array}
     \end{equation}
Moreover, following statements are true:

 $(i)$ $\mathrm{n}_\pm(\mathrm{\bf D}_{X,\alpha})=\mathrm{n}_\pm({\bf B}_{X,\gA})\leq p$. In particular,  $\mathrm{\bf D}_{X,\alpha}=\mathrm{\bf D}_{X,\alpha}^*$   if and only if ${\bf B}_{X,\gA}={\bf B}_{X,\gA}^*$;

 $(iii)$ If the operator $\mathrm{\bf D}_{X,\alpha}$  is selfadjoint and $\lim\limits_{n\to\infty}d_n=0$, then its spectrum is  discrete if and only if the Jacobi matrix ${\bf B}_{X,\gA}={\bf B}_{X,\gA}^*$  has  discrete spectrum.
      \end{proposition}
%

\textbf{3. Schr\"{o}dinger operator.} Let ${\mathrm {\bf H}}_{n}$ be the minimal  operator generated in
$L^{2}([x_{n-1},x_{n}];\C^{p})$  by the differential
expression
  \begin{equation}\label{3.6SH}
{\mathrm {\bf H}}_{n}=-\frac{d^2}{dx^2}, \qquad\dom({\mathrm {\bf H}}_{n})=W^{2,2}_{0}([x_{n-1},x_{n}];\C^{p}).
\end{equation}
Its adjoint ${\mathrm {\bf H}}^*_n$ is given by
${\mathrm {\bf H}}_{n}^*=-\frac{d^2}{dx^2}\upharpoonright\dom({\mathrm {\bf H}}_n^*)$, $\dom({\mathrm {\bf H}}_{n})=W^{2,2}([x_{n-1},x_{n}];\C^{p}).$

We define the minimal operator  ${\mathrm {\bf H}}_X$ on $L^2(\mathcal I;\mathbb C^{p})$ by
  \begin{equation}
{\mathrm {\bf H}}_X:=\bigoplus\limits_{n\in\mathbb N} {\mathrm {\bf H}}_n,  \qquad \dom({\mathrm {\bf H}}_X)=W_0^{2,2}(\mathcal I\setminus X;\mathbb C^{p})=\bigoplus\limits_{n\in\mathbb N}W^{2,2}_{0}([x_{n-1},x_{n}];\C^{p}),
\end{equation}
where ${\mathrm {\bf H}}_n$, $n\in\mathbb N$  is given by  \eqref{3.6A}.
Its adjoint ${\mathrm {\bf H}}^*_X$ is given by
  \begin{equation}\label{HX*}
{\mathrm {\bf H}}_X^*:=\bigoplus\limits_{n\in\mathbb N} {\mathrm {\bf H}}_n^*, \qquad\dom({\mathrm {\bf H}}_X^*)=\bigoplus\limits_{n\in\mathbb N}W^{2,2}([x_{n-1},x_{n}];\C^{p}).
\end{equation}

\begin{theorem}[\cite{KM, KMN}]\label{th_bt_2S} Let $X=\{x_{n}\}_{n=0}^\infty\subset\mathcal I=(0,b)$,  and  $d^*(X)<+\infty$. Define the mappings
$$\Gamma_{j,\mathrm H}^{(n)}: W^{2,2}([x_{n-1},x_n];\C^{p})\to\C^{p}\,, \quad n\in \N\,,\quad j\in \{0,1\}\,,
$$
by setting
\begin{equation}\label{Pi1}
                \Gamma_{0,\mathrm  H}^{(n)}f := \begin{pmatrix}
d_n^{1/2}f(x_{n-1}+) \\ d_n^{3/2}f'(x_{n}-) \end{pmatrix},\qquad
     \Gamma_{1,\mathrm  H}^{(n)}f := \begin{pmatrix}
    d_n^{-1/2}\big(f'(x_{n-1}+)-f'(x_{n}-)\big) \\ d_n^{-3/2}\big(f(x_{n}-)-f(x_{n-1}+)\big)-d_n^{-1/2}f'(x_n-)
        \end{pmatrix}.
    \end{equation}
Then:
\begin{itemize}
\item[(i)] for any $n\in \N$ a triplet  $\Pi_{\mathrm H}^{(n)}=\{\mathbb C^{2p},
\Gamma_{0,\mathrm H}^{(n)},\Gamma_{1,\mathrm H}^{(n)}\}$ is a boundary triplet for
${\mathrm {\bf H}}_n^*$;

\item[(ii)] the direct sum $\Pi_{\mathrm H}:=\bigoplus_{n=1}^\infty\Pi_{\mathrm H}^{(n)} =
\{\cH, \Gamma_{0,\mathrm H},\Gamma_{1,\mathrm H}\}$ with $\cH = l^{2}(\N;\mathbb C^{2p})$ and
${\Gamma}_{j,\mathrm  H} =\bigoplus_{n=1}^\infty {\Gamma}^{(n)}_{j,\mathrm  H},$ $j\in\{0,1\},$  is a boundary triplet for the operator ${\mathrm {\bf H}}_X^* = \bigoplus_{n=1}^\infty {\mathrm {\bf H}}_{n}^{*}$.
\end{itemize}
       \end{theorem}

\begin{proposition}[\cite{KM,KMN}]\label{BoundOp}
   Let $\Pi = \{\mathcal{H}, \Gamma_{0,\mathrm H},\Gamma_{1,\mathrm H}\}$ be the boundary triplet for ${\mathrm {\bf H}}_X^*$
  given by \eqref{Pi1}
and let ${\bf B}_{X,\alpha}^{(2)}(\mathrm {\bf H})$  be  a boundary operator for the Hamiltonian $\mathrm {\bf H}_{X,\alpha}$
with  respect to  the triplet $\Pi$, i.e.
 \begin{eqnarray}\label{Bop1}
&\mathrm {\bf H}_{X,\alpha}=\mathrm {\bf H}_{\mathrm{\bf B}_{X,\alpha}^{(2)}(\mathrm{\bf{H}})}=\mathrm {\bf H}_{X}^*\upharpoonright\dom(\mathrm {\bf H}_{\mathrm{\bf B}_{X,\alpha}^{(2)}(\mathrm{\bf{H}})}),\\
&\dom (\mathrm {\bf H}_{X,\gA}) = \left\{f\in W^{2,2}(\mathcal I\setminus X; \mathbb{C}^p):\, \Gamma_{1,\mathrm H}f={\bf B}_{X,\gA}^{(2)}(\mathrm{\bf{H}}) \Gamma_{0,\mathrm H} f \right\}.
     \end{eqnarray}
Then  ${\bf B}_{X,\alpha}^{(2)}(\mathrm{\bf{H}})$
is unitarily  equivalent to the Jacobi operator defined by~\eqref{IV.2.1Sh''} and thus
 $\mathrm{n}_\pm(\mathrm{\bf H}_{X,\alpha})=\mathrm{n}_\pm({\bf B}_{X,\gA}^{(2)}(\mathrm{\bf{H}}))\leq p$. In particular, $\mathrm{\bf H}_{X,\alpha}=\mathrm{\bf H}_{X,\alpha}^*$   if and only is ${\bf B}_{X,\gA}^{(2)}(\mathrm{\bf{H}})$ is selfadjoint.
    \end{proposition}

{\bf Acknowledgement}
We are grateful  to  Professor G. \'{S}widerski for useful remarks regarding the  first ArXiv version of
our  paper.

\end{document}